\documentclass[12pt,leqno]{amsart} 

\usepackage{amsmath,amssymb}

\theoremstyle{plain} 

\newtheorem{global-theorem}{Theorem}
\newtheorem{theorem}{Theorem}[section]
\newtheorem{lemma}[theorem]{Lemma}

\newtheorem{corollary}[theorem]{Corollary}
\newtheorem{conjecture}[theorem]{Conjecture}
\newtheorem{definition}[theorem]{Definition}
\newtheorem{proposition}[theorem]{Proposition}
\newtheorem{remark}[theorem]{Remark}
\newtheorem{question}[theorem]{Question}
\newtheorem{example}[theorem]{Example}
\newtheorem{condition}[theorem]{Condition}

\numberwithin{equation}{section}

\newcommand{\cc}{{\mathbb C}}
\newcommand{\pp}{{\mathbb P}}

\newcommand{\qq}{{\mathbb Q}}
\newcommand{\zz}{{\mathbb Z}}
\newcommand{\nn}{{\mathbb N}}

\newcommand{\Gm}{{\mathbb G}_{\rm m}}
\newcommand{\aaa}{{\mathbb A}}

\newcommand{\Dd}{{\mathcal D}}
\newcommand{\Ff}{{\mathcal F}}
\newcommand{\Oo}{{\mathcal O}}
\newcommand{\Uu}{{\mathcal U}}
\newcommand{\Rr}{{\mathcal R}}
\newcommand{\Mm}{{\mathcal M}}

\newcommand{\Kk}{{\mathcal K}}
\newcommand{\Ww}{{\mathcal W}}

\newcommand{\Cc}{{\mathcal C}}
\newcommand{\Gg}{{\mathcal G}}
\newcommand{\Ll}{{\mathcal L}}
\newcommand{\Ee}{{\mathcal E}}
\newcommand{\Xx}{{\mathcal X}}

\newcommand{\bF}{{\bf F}}

\newcommand{\hdot}{{\bullet}}

\newcommand{\stackquot}{{/\;\!\!\! /}}

\newcommand{\sk}{{\rm sk}}
\newcommand{\csk}{{\rm csk}}

\newcommand{\ArtSt}{{\bf ArtSt}}
\newcommand{\DMS}{{\bf DMSt}}
\newcommand{\stck}{{\bf St}}
\newcommand{\spre}{{\bf SP}}
\newcommand{\TOP}{{\bf TOP}}

\newcommand{\twoarrows}{\,
\overset{\longrightarrow}{\rule{0pt}{1.6pt}\smash[t]{{\scriptstyle \longrightarrow}}}
\,}

\newcommand{\twoarrowsforthree}{\overset{\longrightarrow}{\rule{0pt}{1pt}\smash{{\scriptstyle \longrightarrow}}}}

\newcommand{\threearrows}{\, \overset{\longrightarrow}{\rule{0pt}{4.8pt}\smash{\twoarrowsforthree}}
\,}

\begin{document}

\author[C. Simpson]{Carlos T. Simpson}
\address{CNRS, Laboratoire J. A. Dieudonn\'e, UMR 6621
\\ Universit\'e de Nice-Sophia Antipolis\\
06108 Nice, Cedex 2, France}
\email{carlos@unice.fr}
\urladdr{http://math.unice.fr/$\sim$carlos/} 

\title[Local systems]{Local systems on proper algebraic $V$-manifolds}

\begin{abstract}
We use coverings by smooth projective varieties 
then apply nonabelian Hodge techniques to study the topology of proper
Deligne-Mumford stacks as well as more general simplicial varieties.
\end{abstract}


\thanks{
This research is supported in part by the Agence Nationale de la Recherche
grant ANR-09-BLAN-0151-02 (HODAG)
}

%\dedicatory{To the memory of Eckart Viehweg}

\maketitle

%\tableofcontents

\section{Introduction} \label{sec-introduction}

This paper originated with the project of trying to understand how the
techniques of harmonic theory and moduli spaces would apply to local systems
over smooth proper Deligne-Mumford stacks.

The subject of DM-stacks
has a rich history.
The Kawamata-Viehweg vanishing theorem \cite{KawamataVanishing} \cite{ViehwegVanishing}
was originally proven by techniques involving
cyclic or polycyclic Galois
coverings of a smooth projective variety, ramified over a divisor with
simple normal crossings. In current-day terms, 
Matsuki and Olsson have explained it as an instance of Kodaira vanishing
over a root stack \cite{MatsukiOlsson}. Viewed in this light, the 
vanishing theorem could be considered as one of the first
major results about usual varieties where the geometry of Deligne-Mumford stacks
plays a crucial role. 

The coverings which appear in the original proofs
may be viewed as varieties covering the DM-stack. We will take
up this idea here to say, in Theorem \ref{maincovering},
that any smooth proper DM-stack $X$ is covered by
a map $\phi : Z\rightarrow X$ from a disjoint union of smooth projective varieties
such that
every point downstairs $x\in X$ admits at least one point
$z\in \phi ^{-1}(x)\subset Z$ where $\phi$ is etale. 
A technical contribution to this statement comes from the Chow lemma of 
Gruson and Raynaud
\cite{RaynaudGruson}.  
Proper coverings of stacks by schemes, with essentially similar constructions, have been considered by many authors, see for example \cite{KreschVistoli}, \cite{OlssonStarr},
\cite{OlssonCoverings}. 

The goal here is to use these covering varieties $Z$ to
understand local systems on $X$. In order to do this, it is natural to
look next at $Z\times _XZ$, but then resolve its singularities to get a smooth variety $Z_1$.
This is the beginning of a simplicial resolution
$$
Z_1 \twoarrows Z=Z_0 \rightarrow X
$$
and standard constructions allow it to be completed to a full one. Such simplicial
resolutions were used by Deligne for Hodge theory on singular varieties \cite{hodge3},
and would seem to represent interesting topological objects in their own right.
So we expand the level of generality by usually looking at a simplicial scheme $Z_{\hdot}$ such that the components $Z_k$ are smooth projective varieties.
In the differentiable category, 
these objects have been considered in \cite{Dupont} and \cite{Jeffrey}. 
No further topological generality would be gained by looking at simplicial objects
whose levels are proper algebraic spaces.  

The various moduli stacks of local systems on $X$ may now be expressed
as limits of the moduli stacks for the $Z_k$, Proposition \ref{modulistacks}.
The moduli stacks admit universal categorical quotients which are the 
various versions of the character variety \cite{LubotskyMagid}
of representations up to conjugacy. 
A natural question is to what extent these moduli stacks and their character varieties
behave like in the smooth projective case. 

After considering the general theory of moduli of local systems, we would like
to use the covering varieties to do nonabelian harmonic analysis over the
stack. In fact it turns out that we just have to apply the 
classical theory at each level of
the simplicial variety. The surjectivity of the etale locus of the coverings
allows us to interpret the result in terms of harmonic bundles on the 
original stack. 

Our discussion of nonabelian harmonic theory on stacks adds to 
a subject which has already been treated by several authors
\cite{BGHH} \cite{Dhillon} \cite{Gillet2} \cite{ToenThesis}, 
and for the case of root stacks it is closely related to
harmonic theory for parabolic bundles \cite{Biquard}
\cite{DaskalopoulosWentworth}
\cite{LiNarasimhan} 
\cite{Nakajima} \cite{TMochizuki2} \cite{TMochizuki3}
\cite{Poritz} \cite{SteerWren}. 
The relationship between local systems and ramified covers can be related to
the Chern class calculations of Esnault and Viehweg in \cite{EsnaultViehweg},
going back also to \cite{Ohtsuki},
and related formulae involving parabolic and orbifold bundles 
were studied in \cite{IyerSimpson1} \cite{IyerSimpson2}. 
This subject also connects with Viehweg's recent works such as 
\cite{ViehwegZuo2} \cite{ViehwegZuo}, since Shimura varieties are best
considered as DM-stacks and indeed symmetric spaces were a main part of the 
original motivation for
the notion of $V$-manifold \cite{Satake1} \cite{Satake2} which appears in our title.
The other main motivation came from $\Mm _g$ \cite{DeligneMumford}, but
as Campana has pointed out \cite{Campana04}, local systems on
orbifolds play an important
part in the theory of moduli of more general varieties too. Examples over
stacks locally of ADE-discriminant type up to dimension $6$ 
have been constructed in \cite{Megy}. See \cite{Uludag} for a classification
of orbifold structures over $\pp ^2$. 
Fascinating new examples have arisen with the notion of ``twisted curves'',
see \cite{Chiodo} for references.  

A general simplicial scheme can have a pretty arbitrary topological type,
for example any simplicial set with $X_k$ finite for each $k$ qualifies. For a
general $X_{\hdot}$ one should therefore modify the kind of question being asked---not
which topological types can occur, but rather how the topologies of the $X_k$ interact
with the full topological type of $|X_{\hdot}|$. This is a very interesting question
closely related to the notion of nonabelian weight filtration.

The role of the weight filtration 
is illustrated by looking at the mixed Hodge structure on the complete local
ring of the space of representations of $\pi _1$, Proposition \ref{mhsrepvar}, 
generalizing the recent paper \cite{EyssidieuxSimpson}. A somewhat delicate point to beware of is the choice of basepoints. Even though we only need to use the representation spaces for the first
two pieces of a simplicial resolution $Z_0$ and $Z_1$, the example 
\ref{coordinateplanes} of three planes
meeting in a point readily shows that one should be sure to choose basepoints 
meeting all components of $Z_2$. 

Whereas the intervention of the weight filtration is to be expected in a general
singular situation, one hopes that some kind of purity would hold for
smooth proper DM-stacks. 

For this, we can notice that the simplicial 
resolutions $Z_{\hdot}\rightarrow X$ arising for smooth Deligne-Mumford stacks have
the nice property that the image of
$\pi _1(Z_0)$ is of finite index in $\pi _1(X)$ (Condition \ref{finiteindex}). 
This guarantees that $\pi _1(X)$ doesn't
include loops which jump from one place to another in $Z_0$ by going through the space
$Z_1$ of $1$-simplices.  This condition allows us to recover much of the theory 
of moduli of local systems. 

The finite-index 
condition holds for simplicial hyperresolutions of singular varieties, whenever the singularities are normal or indeed geometrically unibranched. The
phenomenon we are trying to avoid is loops going through singular points and 
jumping from one branch to the other. We can therefore make the essentially 
straightforward observation that much of the theory known for the smooth projective case applies also to geometrically unibranched varieties, and in fact---combining the two examples---to geometrically unibranched DM-stacks. 

Some of the main properties are Hitchin's hyperk\"ahler structure, 
Theorem \ref{hyperkahler} and
the continuous action of $\cc ^{\ast}$ whose fixed points are 
variations of Hodge structure,
Lemma \ref{cstaraction} and Corollary \ref{cstarcont}. 
These results all lead to restrictions on which groups can occur as 
fundamental groups of proper geometrically unibranched DM-stacks.

We take the opportunity to explain how Deligne's theory of \cite{hodge3} applies to
get mixed twistor structures on the cohomology of semisimple local systems. 
Poincar\'e duality implies that these mixed structures are pure in the
case of a smooth proper DM-stack.

Near the end of the paper, we discuss some constructions involving finite group actions,
constructions which motivate the passage from smooth projective varieties
to DM-stacks. If $\Phi$ is a finite group acting on a smooth projective $X$ then
the stack quotient of the moduli stack $\Mm (X,G)\stackquot \Phi$ may be
interpreted as a piece of the moduli stack of $H$-local systems on the DM-stack quotient
$Y=X\stackquot \Phi$ for a suitable group $H$ (Corollary \ref{quotexpression}).
In the last section of the paper, we answer a question posed by D. Toledo many
years ago, showing that any group can be the fundamental group of an irreducible 
variety.

{\em Conventions:} 
All schemes are separated and of finite type over the field $\cc$ of complex numbers. 

{\em Acknowledgements:}
Many early thoughts on these techniques 
came while visiting Columbia University in 2007, and I would like to thank Johan de Jong and the Columbia math department for their hospitality.
I would also like to thank
Nicole Mestrano, Andr\'e Hirschowitz, Charles Walter, Bertrand Toen,
Kevin Corlette, Alessandro Chiodo,
Jaya Iyer, David Favero, and Damien M\'egy for discussions about DM-stacks, Domingo Toledo
for the question about $\pi _1$ of irreducible varieties,
Ludmil Katzarkov for discussions about degenerating 
nonabelian Hodge theory, as well as Philippe Eyssidieux and Tony Pantev for discussions on
mixed Hodge structures.

\section{The topology of simplicial schemes}

Let $\Delta$ be the category of nonempty finite linearly ordered sets
denoted $[n]=\{0, \ldots , n\}$. A {simplicial object} in a category
$\Cc$ is a functor $Y_{\hdot}:\Delta ^o\rightarrow \Cc$, with
levels denoted $Y_k:=Y_{\hdot}([k])$. Following \cite{hodge3} 
an {augmented simplicial object}
is a simplicial object $Y_{\hdot}$ together with another object $S\in \Cc$
and a natural transformation $p:Y_{\hdot}\rightarrow S$ from $Y_{\hdot}$ to the
constant simplicial object with values $S$. This may also be considered as
a functor $(\Delta \cup \{[-1]\} )^o\rightarrow \Cc$ where $[-1]=\emptyset$ is the 
empty linearly ordered set, with $Y_{-1}=S$. We usually write such an object as
$Y_{\hdot}\rightarrow S$, thinking of $\Cc$ as being contained in the category
of simplicial $\Cc$-objects by the constant-object functor.

If $\Cc = Top$ or $\Cc = \Kk$ where $\Kk = Hom (\Delta ^o,Sets)$ is the
Kan-Quillen model category of simplicial sets, then a simplicial $\Cc$-object will
be called a {simplicial space}. A simplicial space $Y_{\hdot}$ has a 
{topological realization} denoted $|Y_{\hdot}|$ which is a space,
defined as the quotient space of 
$$
\coprod _{k\in \Delta} Y_k\times R^k 
$$
by the relation $(\phi ^{\ast}(y),r)\sim (y,\phi _{\ast}r)$
whenever $y\in Y_k, r\in R^m$ and $\phi :[m]\rightarrow [k]$ is a morphism in $\Delta$.
Here $R^k$ are the standard $k$-simplices, fitting together into a cosimplicial space. 
The fat realization $\| X_{\hdot}\|$  
is defined in the same way but using only the injective maps in
$\Delta$. For $\Cc=Top$ some cofibrancy conditions \cite{nlab}
must be imposed in order to have
a homotopy equivalence $\| X_{\hdot}\| \stackrel{\sim}{\rightarrow}|X|$;
these conditions are automatic for $\Cc = \Kk$, and also hold when $X_{\hdot}$
is $s$-split, so they will be tacitly assumed in all statements. 

Suppose now 
$X_{\hdot}$ is a simplicial scheme.
Then applying the usual functor to underlying topological spaces levelwise we obtain
a simplicial space $X^{\rm top}_{\hdot}$ whose levels are the $(X_k)^{\rm top}$. 
The topological realization is a topological space 
$$
|X_{\hdot}|:= |X^{\rm top}_{\hdot}|.
$$
These spaces will be the main objects of our study. 

A simplicial scheme or space has split degeneracies, or is $s$-split in 
Deligne's terminology \cite{hodge3}, if each $X_m$ 
is a disjoint union given by the degeneracy maps
$$
X_m = X^N_m \sqcup \coprod _{k<m, m\twoheadrightarrow k} X^N_k.
$$
The first term is $X_0^N=X_0$. 
We usually assume this condition, which also implies the cofibrancy conditions
refered to above.

A local system on a simplicial space $Y_{\hdot}$ consists of a collection
$L_{\hdot}=\{L_k\}$ of local systems $L_k$ on $Y_k$, together with isomorphisms
$\phi ^{\ast}(L_k)\cong L_m$ whenever $\phi : [k]\rightarrow [m]$ induces $Y_m\rightarrow Y_k$,
and these isomorphisms should satisfy the natural compatibility conditions as
well as being the identity when $\phi$ is. This applies to local systems of abelian 
groups, vector spaces or modules over a ring, but also to local systems of sets
and hence to $G$-torsors which are local systems of $G$-principal homogeneous sets. 

We generally assume that our spaces are good enough that local systems correspond to representations of $\pi _1$. This assumption holds for the underlying topological 
spaces of schemes for example, but also for the realizations of simplicial spaces which
levelwise are good enough and satisfy the required cofibrancy conditions.

Reflecting the fact that realizations and fat realizations are homotopy equivalent,
the notion of local system is equivalent to the analogous notion defined using only
the injective maps $\phi : [k]\hookrightarrow [m]$.

If $L_{\hdot}$ is a local system of abelian groups on $Y_{\hdot}$, we can choose a
compatible system of injective resolutions $\Ff ^{\hdot}_k$ of $L_k$ over $Y_k$.
Taking sections gives a simplicial complex of abelian groups whose
{total complex} $t(\Ff ^{\hdot}_{\hdot}(Y_{\hdot}))$ is defined by
$$
t(\Ff ^{\hdot}_{\hdot}(Y_{\hdot}))^i = \bigoplus _{k+j=i}\Ff ^j_k(Y_k),
$$
with differential using the alternating sign of face maps. 
The cohomology $H^i(Y_{\hdot},L_{\hdot})$ is defined to be the cohomology of
this total complex. It is independent of the choice of resolution.

\begin{lemma}
Suppose $L$ is a local system over $|Y_{\hdot}|$. For each $k$, let $L_k$ be the
pushdown along $Y_k\times R^k\rightarrow Y_k$ of the restriction of $L$. 
Then the $L_k$ fit together to form a local system $L_{\hdot}$ on $Y_{\hdot}$
and this construction establishes an equivalence of categories between local
systems on $|Y_{\hdot}|$ and local systems on $Y_{\hdot}$. If $L$ is a local system 
of abelian groups on 
$|Y_{\hdot}|$ and $L_{\hdot}$ the corresponding local system on $Y_{\hdot}$ then 
there is a natural isomorphism 
$H^i(|Y_{\hdot}|,L)\cong H^i(Y_{\hdot},L_{\hdot})$.
\end{lemma}

\begin{corollary}
The groupoid of $G$-torsors on $|Y_{\hdot}|$, denoted 
$H^1(|Y_{\hdot}|,G)$, is the $2$-limit of the $\Delta $-diagram of 
groupoids $k\mapsto H^1(Y_{k},G)$.
\end{corollary}

This generalizes to higher nonabelian cohomology: if $T$ is an $n$-groupoid
then ${\rm Hom}(\Pi _n(|Y_{\hdot}|),T)$ is the $n+1$-limit of 
the functor from $\Delta$ to $nGPD$ given by $k\mapsto {\rm Hom}(\Pi _n(Y_{k}),T)$.

For basepoints, 
rather than simply choosing a single one, it is often necessary to
consider a map from a simplicial set. 
Suppose $U_{\hdot}$ is a simplicial set with a map $U_{\hdot}\rightarrow Y_{\hdot}$.
For each $k$ we obtain a collection of points $U_k$ mapping to $Y_k$; it is more
convenient not to require the map to be injective. The realization 
$|U_{\hdot}|$ is just the usual realization of the simplicial set, and
we obtain a map $|U_{\hdot}|\rightarrow |Y_{\hdot}|$. 

Say that $U_{\hdot}$ is $0$-truncated if 
the realization $|U_{\hdot}|$ is a $0$-truncated space, i.e. its homotopy groups vanish
in degrees $i\geq 1$. Equivalently, it is a disjoint union of contractible pieces.
A simplicial basepoint is a map $U_{\hdot}\rightarrow Y_{\hdot}$ such that
$U_{\hdot}$ is a $0$-truncated simplicial set with each $U_k$ finite. 

We mainly consider such a $U_{\hdot}$ which is a finite disjoint union of 
standard simplices. 
Let $h([k])$ denote the representable simplicial set represented by $[k]\in \Delta$,
thus  $h([k])_m = \Delta ([m],[k])$. It is contractible.
Suppose given a point $y\in Y_k$; this
induces a map $h([k])\rightarrow Y_{\hdot}$, and furthermore any map
is induced from a point $y\in Y_k$ in that way. 
As notation, write $\langle y\rangle := h([k])$ together with the
given map to $Y_{\hdot}$.
If $\{y_i\}$ is a collection of nondegenerate points $y_i\in Y_{k_i}$ such that
the $\langle y_i\rangle$ are disjoint, their union
$$
U_{\hdot}:= \coprod _i\langle y_i\rangle \rightarrow Y_{\hdot}
$$ 
is a
simplicial basepoint.  
It is often convenient to look at $v_0y_i \in Y_0$, the $0$-th vertex of $y_i$,
corresponding to $[0]\subset [k_i]$. It realizes to a point also denoted 
$ v_0y_i\in |\langle y_i\rangle |\subset |Y_{\hdot}|$. 

Suppose $L_{\hdot}$ is a local system on $Y_{\hdot}$ corresponding to $L$ on 
$|Y_{\hdot}|$. If $U_{\hdot}\rightarrow Y_{\hdot}$ is a map from a simplicial set,
then the restriction of $L_{\hdot}$ to $U_{\hdot}$ is a local system on
the realization $|U_{\hdot}|$. In particular, if $U_{\hdot}$ is $0$-truncated, then the restriction is trivializable on each contractible connected component
of $|U_{\hdot}|$, and a choice of trivialization is equivalent to a choice of
trivialization over any point of this component. 

Apply this to a simplicial basepoint 
$U_{\hdot}=\coprod _i\{\langle y_i \rangle \}$,
with $y_i\in Y_{k_i}$. The inclusion of the $0$-th vertex into the standard simplex 
$R^{k_i}$ yields an isomorphism 
$$
L_{k_i}(y_i)\cong L(v_0y_i).
$$
A trivialization of $L_{\hdot}$ restricted to $U_{\hdot}$ is therefore the
same thing as a collection of trivializations of $L_{k_i}(y_i)$ or
a collection of trivializations of $L(v_0y_i)$. 

If $U\rightarrow Y$ is a map of spaces and $G$ is a group, 
denote by $H^1(Y,U;G)$ the groupoid of $G$-torsors on $Y$ together with trivializations
of the pullbacks to $U$. 
If the image of 
$U$ meets each connected component of $Y$ then this groupoid is a discrete set.

If $U_{\hdot}\rightarrow Y_{\hdot}$ is simplicial basepoint, then we obtain a diagram
$$
k\mapsto H^1(Y_k,U_k;G)
$$
of groupoids. 

\begin{proposition}
\label{h1calc}
Suppose that $U_{\hdot}\rightarrow Y_{\hdot}$ is a simplicial basepoint. 
Suppose $G$ is a group. Suppose that $U_k$ meets all the
connected components of $Y_k$ for $k=0,1,2$. Then
$H^1(|Y_{\hdot}|,|U_{\hdot}|;G)$ is the equalizer of the pair of face maps
$$
H^1(Y_0,U_0; G)\twoarrows H^1(Y_1,U_1; G).
$$
Let $P:= \pi _0(|U_{\hdot}|)$. Then $G^P$ acts on this equalizer and
the quotient groupoid is $H^1(Y_{\hdot},G)$. 
\end{proposition}
\begin{proof}
The cohomology $1$-groupoid 
$H^1(|Y_{\hdot}|,|U_{\hdot}|;G)$ is the $2$-limit of 
the family of cohomology groupoids 
$H^1(Y_k,U_k; G)$ indexed by $k\in \Delta$. This only depends on the initial part for
$k=0,1,2$, as will be explained later in Lemma \ref{limdiag}. 
If $U_k$ meets all components of $Y_k$ for $k=0,1,2$ then 
the groupoids are discrete, and the $2$-limit is a $1$-limit of a diagram of sets,
which in turn is equal to the stated equalizer. 
\end{proof}

Notice that $(Y_2,U_2)$ doesn't enter into the expression for 
the cohomology groupoid
$H^1(|Y_{\hdot}|,|U_{\hdot}|;G)$. However, if $U_2$ doesn't meet all the connected
components of $Y_2$ then the expression may not be true as the
following example shows. 

\begin{example}
\label{coordinateplanes}
Let $X$ be a singular variety, union of three coordinate planes
meeting at the origin in $\pp ^3$. Let $Y_{\hdot}$ be the
standard simplicial resolution \cite{hodge3} 
with $Y_0$ being the disjoint union of three planes,
the nondegenerate part of $Y_1$ being the disjoint  union of three lines,
and the nondegenerate part of $Y_2$ being the origin.

If $U$ contains a basepoint on each of the
double intersections but not at the origin, then the equalizer in the
expression of Proposition \ref{h1calc} is different from $H^1(X,U,G)$.
\end{example}

To see this, let $X'$ be the pyramid consisting of three copies of $\pp ^1\times \pp ^1$
meeting along three disjoint lines. For a set of basepoints
$U\subset X$ not containing the origin, one can choose a similar collection $U'\subset X'$ for which the expression of the equalizer in
\ref{h1calc} is the same. In the case of $X'$ there are no triple intersections
so the equalizer expression is the correct one and it gives $H^1(X',U',G)$.
However, $\pi _1(|X'|)=\zz$ whereas $X$ was simply connected, so 
$H^1(X',U',G)\neq H^1(X,U,G)$.

\section{Deligne-Mumford stacks}

Let ${\rm Sch}$ denote the category of separated schemes of finite type over $\cc$.
Provided with the etale topology it becomes a site. 

Classically, a $1$-stack over ${\rm Sch}$  is viewed as a category fibered in
groupoids $\Xx \rightarrow {\rm Sch}$, satisfying a descent condition. 
Recall that a fibered category can be strictified to a presheaf of $1$-groupoids
by setting $X(S)$ equal to the groupoid of sections ${\rm Sch}/S\rightarrow \Xx$.
There is also a more topological approach. 

Let $\spre$ denote the category
of simplicial presheaves, with $\Ww$ defined as the class of Illusie weak equivalences.
Let $\spre_1\subset \spre$ be the subcategory of $1$-truncated simplicial presheaves $X$,
that is ones where $X(S)$ has $\pi _i=0$ for $i\geq 2$. 

Given a presheaf of $1$-groupoids, the corresponding presheaf of nerves is in $\spre_1$.
Conversely given a $1$-truncated simplicial presheaf, we can look at the presheaf of
Poincar\'e $1$-groupoids. For speaking of $1$-prestacks, 
these constructions, together with the strictification
construction described above, set up an essential equivalence between the
classical fibered-category point of view, and the category $\spre_1$.

Illusie weak equivalence
defines a class of morphisms still denoted by $\Ww$ in $\spre_1$. The $\Ww$-local objects in $\spre_1$ 
correspond to presheaves of $1$-groupoids or fibered categories which satisfy the
descent condition to be $1$-stacks, see \cite{Hollander} for
example. Denote by $\spre_{1,{\rm loc}}$ the subcategory of $\Ww$-local objects; one
may equivalently take the subcategory of fibrant objects for either the projective or 
injective model structures. 

Dwyer-Kan localization provides a simplicial or $(\infty , 1)$-category 
$$
\stck:= L_{DK}(\spre_{1,{\rm loc}}, \Ww ) \sim  L_{DK}(\spre_{1}, \Ww )
$$
of $1$-stacks on ${\rm Sch}$. It is $2$-truncated, that is to say the mapping spaces
are $1$-truncated, so in Lurie's terminology it corresponds to a $(2,1)$-category.
This is a $2$-category in which all $2$-morphisms are invertible. This is the same as
the classical $2$-category of $1$-stacks over the site ${\rm Sch}$, a compatibility
well-known particularly from Hollander's work \cite{Hollander}.

The above viewpoint involving localization is useful for defining the {topological
realization} of a stack. The topological realization functor on simplicial presheaves,
considered in \cite{realization}, \cite{TelemanSimpson}, \cite{DuggerIsaksen}, is denoted
$$
|\;\; | : \spre \rightarrow Top
$$
where we are using $Top$ as shorthand for the Kan-Quillen model category of simplicial sets. It sends Illusie weak equivalences to weak equivalences, so it passes to
the Dwyer-Kan localizations. Let $\TOP$ denote the $(\infty , 1)$-category which is the
Dwyer-Kan localization of $Top$ by the weak equivalences. Then we get
an $(\infty , 1)$-functor
$$
|\;\; | : L_{DK}(\spre_{1,{\rm loc}},\Ww ) \rightarrow \TOP
$$
which is written as a realization functor for stacks
$$
|\;\; | : \stck \rightarrow \TOP.
$$ 
Note that $|X|$ is equivalent to the realization of any simplicial presheaf which is Illusie weak-equivalent to $X$. From this, follows the compatibility of realization
with etale hypercoverings. If $Y_{\hdot}$ is a simplicial scheme, then since objects
of ${\rm Sch}$ determine representable presheaves, we obtain a simplicial presheaf.
An {etale hypercovering} of a stack $X$ is a morphism in $L_{DK}(\spre, \Ww )$
$$
Y_{\hdot}\rightarrow X
$$
such that the matching maps 
$$
Y_k\rightarrow {\rm match}_k(Y_{\hdot}\rightarrow X)
$$
are coverings in the etale topology.  Here the simplicial coordinate is included in $Y_{\hdot}$ but not in $X$ to emphasize that we are considering this as an augmented
simplicial object in $L_{DK}(\spre , \Ww )$, but it may also be viewed as just a morphism
in $L_{DK}(\spre , \Ww )$. The fact that $X$ is a stack rather than a scheme 
doesn't affect the definition of hypercovering, see Remark \ref{matchstack} below. 

An etale hypercovering is, when viewed as a morphism of simplicial
presheaves, an Illusie weak equivalence.

In this situation, $k\mapsto |Y_k|$ is a simplicial space denoted $|Y|_{\hdot}$.
We have a weak equivalence of spaces 
$$
| (|Y|_{\hdot})| \sim |X|.
$$
In other words, the topological realization of $X$ may be calculated by first choosing
an etale hypercovering, then taking the associated simplicial space, and taking the 
topological realization of that in the sense used at the start of the paper. This brings us back to Noohi's construction of the topological realization of a stack \cite{NoohiRealization}, and similar constructions considered by Gepner, Henriques \cite{GepnerHenriques} and Ebert \cite{Ebert}. 

If $Z_{\hdot}$ is a simplicial object in $\DMS$ then $k\mapsto |Z_k|$ is a simplicial
space, whose realization also denoted $|Z_{\hdot}|$
is functorial in $Z_{\hdot}$. For a simplicial scheme this coincides up to weak equivalence 
with the realization defined previously. In particular, 
if $Z_{\hdot}\stackrel{a}{\rightarrow} X$ is a morphism from a simplicial scheme to a stack,
considering the target as a constant simplicial object which has the same realization,
we obtain a map of spaces
\begin{equation}
\label{simpmap}
|Z_{\hdot}|\stackrel{|a|}{\longrightarrow} |X|.
\end{equation}
In the case of the etale hypercovering $Y_{\hdot}$ this is the weak equivalence
considered above; we shall be interested in it for a proper surjective hypercovering.

A $1$-stack $X$ is a  {Deligne-Mumford (DM) stack} if it has a presentation of the form $X=Z/R$ where 
$Z$ is a separated scheme of finite type over $\cc$, and $R\rightarrow Z\times Z$ is
a groupoid in the category of schemes 
such that each projection $R\rightarrow Z$
is etale.  For smooth DM-stacks, this 
notion is the algebraic analogue of Satake's notion of $V$-manifold
\cite{Satake1} \cite{Satake2} or ``orbifold'',
with the added feature that the generic stabilizer group can be nontrivial. But even if
we start with a $V$-manifold, natural substacks
can have nontrivial generic stabilizer so that possibility remains geometrically motivated and should be included.

The collection of DM-stacks naturally forms a $2$-category which we denote by $\DMS$, 
a full sub-$2$-category of $\stck$. The $2$-category structure comes about because
one can have nontrivial natural automorphisms of morphisms $f:X\rightarrow Y$.
This phenomenon occurs particularly if the automorphism group in $Y$ at the general point
of the image of $f$ is  nontrivial. Note however that if $Y$ is a scheme or
algebraic space, then maps from any stack to $Y$ have no nontrivial automorphisms. 

It is instructive to consider the case where $Y=V\stackquot G$ is a quotient stack of a scheme $V$
by the action of a finite group $G$. In this case, a map $X\rightarrow Y$
is a pair $(T,\phi )$ where $T\rightarrow X$ is a $G$-torsor and $\phi : T\rightarrow V$
is an equivariant map. An isomorphism between two maps $(T,\phi )\cong (T',\phi ')$
is an isomorphism of $G$-torsors $u:T\cong T'$ such that $\phi 'u =\phi$.

Following the previous discussion, let $\spre _{DM}\subset \spre_{1,{\rm loc}}$ denote the
full subcategory of simplicial presheaves corresponding to 
$1$-stacks which are Deligne-Mumford. Then 
$$
\DMS := L_{DK}(\spre_{DM},\Ww )
$$
is the $(\infty , 1)$-category defined by Dwyer-Kan localization along the Illusie weak
equivalences (which, for $\Ww$-local objects, are the same thing as the objectwise
weak equivalences of simplicial presheaves or, in a terminology more adapted to $1$-stacks,
objectwise equivalences of $1$-groupoids). Again this is $2$-truncated, i.e. it is
really a $(2,1)$-category, and we denote also by $\DMS$ the same considered as a classical $2$-category in which the $2$-morphisms are invertible. 

This $2$-category has a $1$-truncation $\tau _{\leq 1}\DMS$. It is the category whose
objects are DM-stacks and whose morphisms are equivalence classes of morphisms.
The projection functor 
$$
\DMS \rightarrow \tau _{\leq 1}\DMS
$$
does not have a section, as one can already see on examples of the form $BG$ for a finite
group $G$. Hence, when we speak of
a ``map between DM-stacks'' it means a morphism of simplicial presheaves or
fibered categories. Thus, by the ``category of DM-stacks'' we really
mean either $\spre_{DM}$ or the more classical category whose objects are categories fibered
in $1$-groupoids over ${\rm Sch}$. In these categories there will usually be several different morphisms representing the same equivalence class. 

The $2$-functor $\DMS\rightarrow \TOP$ gives us some additional structure. Suppose
$X,Y$ are DM-stacks. Then $Hom  _{\DMS}(X,Y)$ is a groupoid, and its realization
maps to the space $Hom _{\TOP}(|X|, |Y|)$. Given a map $X\rightarrow Y$, 
this gives a map of spaces from the classifying space of the finite group
of  natural automorphisms of $f$, 
to the the mapping space:
$$
B(Aut _{\DMS (X,Y)}(f))\rightarrow Hom _{\TOP}(|X|, |Y|).
$$
The first part of this structure is just the map of groups
$$
Aut _{\DMS (X,Y)}(f)\rightarrow \pi _1(Hom _{\TOP}(|X|, |Y|),|f|)
$$
but the map of spaces contains extra structure which would be
interesting to study further.

\section{The structure of DM-stacks} \label{sec-structure}

One of the original goals of this work was to get information about the
topology of DM-stacks. In preparation for the construction of smooth projective
covering varieties, we first
recall some standard structural results.
Many references are available: we have found \cite{Gomez} to be useful and concise,
\cite{ToenThesis} discusses a wide range of topics, numerous papers of Olsson and
other co-authors \cite{OlssonCoverings} \ldots provide invaluable viewpoints,
and \cite{Alper} is a guide to the extensive
literature; 
in the future \cite{StacksProject} will provide a definitive reference.

A closed substack is a morphism $Y\rightarrow X$ such that on any etale
chart $Z_i\rightarrow X$, the fiber product $Y\times _XZ_i$ is a closed
subscheme of $Z_i$. This amounts to specifying a closed substack on each
chart, compatible with the glueing equivalence relation. The intersection of
any number of closed substacks is again a closed substack. Notice, however,
that a morphism from a point is not generally a closed substack, for
example the only closed substacks of $BG$ are $\emptyset$ and $BG$ itself. 

A {Cartier divisor} $D$ on $X$ is the specification for each etale chart $p:Z_p\rightarrow X$ of a
Cartier divisor $D_p$ on $Z_p$, such that if $Z_q \stackrel{f}{\rightarrow} Z_p \rightarrow X$ is a
diagram of etale charts then $f^{\ast}(D_p)= D_q$. In this paper the word {divisor} will mean a Cartier divisor.
For a scheme or an algebraic space this definition coincides with the usual one. 
If $f:X\rightarrow Y$ is a morphism of DM-stacks and $D$ is a divisor on $Y$ then, if no irreducible component of $X$ maps into $D$ we
can define the pullback $f^{\ast}(D)$. The divisors $D_p$ in the definition above are also the  pullbacks $D_p = p^{\ast}(D)$.  
We say that $D$ has {normal crossings} if
for any etale chart $p:Z_p \rightarrow X$ the divisor $p {\ast}(D)=D_p$ has normal crossings. A divisor may be identified with a
closed substack. The etale charts for the substack $D$ are the $D_{p_i}$ for etale charts $p_i$ covering $X$. 

A DM-stack $X$ is {separated} if the diagonal map $X\rightarrow X\times X$ is proper,
which is equivalent to a valuative criterion or also to saying that the
map $R\rightarrow Z\times Z$ in the groupoid defining $X$ is proper. 

A DM-stack 
$X$ is proper if and only if it is separated and satisfies the
valuative criterion, saying that for any discrete valuation ring $A$ with fraction field $K$ and any map 
$Spec (K)\rightarrow X$ there exists an extension to a map $Spec(A')\rightarrow X$
where $A'$ is the normalization of $A$ in a finite extension $K'$ of $K$.
This is equivalent to the existence of a surjective covering map from a proper scheme
\cite{LaumonMB} \cite{OlssonCoverings} \cite{Gomez}.

Recall the results of Keel and Mori \cite{KeelMori}. For any separated finite-type 
DM-stack $X$ there exists
an algebraic space $X^{\bf c}$ called the {coarse moduli space} together
with a finite map $X\rightarrow X^{\bf c}$.  It is universal for maps from $X$ to
a separated algebraic space, and furthermore if $Y$ is an algebraic space mapping to
$X^{\bf c}$ then $X\times _{X^{\bf c}}Y\rightarrow Y$ is also universal for maps to an algebraic
space. 

Locally over $X^{\bf c}$ in the etale topology, $X$ is a quotient stack. 
Toen \cite[Proposition 1.17]{ToenThesis} refers to Vistoli
\cite[Proof of 2.8]{Vistoli} for this statement; see also Kresch \cite{Kresch}.

The functorial resolution of singularities of Bierstone-Milman \cite{BierstoneMilman} and Villamayor \cite{Villamayor} implies resolution of
singularities for Deligne-Mumford stacks:

\begin{proposition}
Suppose $X$ is a reduced separated DM-stack of finite type over $Spec (\cc )$.
Then there exists a surjective proper birational 
morphism $Z\rightarrow X$ of DM-stacks, an isomorphism over the
dense Zariski open substack of smooth points of $X$, such that $Z$ is smooth.
\end{proposition}
\begin{proof}
Functoriality of the resolution procedure for etale morphisms means that 
the glueing procedure described in \cite[\S 7.1]{BierstoneMilmanFunctoriality},
see also \cite{Villamayor},
extends to the case of etale open coverings of a DM-stack.
\end{proof}

One of the main constructions of Deligne-Mumford stacks is to look
at the {Cadman-Vistoli root stacks}.
Let $X$ be a smooth projective variety, and $D=D_1+ \ldots + D_k$ a divisor with normal crossings broken 
up into its components
$D_i$ which are assumed to be irreducible and smooth. Fix a sequence of strictly positive integers $n_1,\ldots , n_k$. 
Cadman \cite{Cadman} defines and studies a stack $Z:= X[\frac{D_1}{n_1}, \ldots , \frac{D_k}{n_k}]$ with a morphism $Z\rightarrow X$. 
Often we choose the same $n$ for each component. Vistoli had also considered these stacks,
see \cite{AGV}.

In a philosophical sense, the 
technique of root stacks may be traced back to Viehweg's use of cyclic
coverings branched along a normal crossings divisor \cite{ViehwegVanishing}
and Kawamata's covering lemma \cite{Kawamata}.
This covering technique has been used by many authors since then; for a recent example
see Urz\'ua \cite{Urzua}.  

The stack $Z$ can be explicitly presented as a quotient stack locally in the Zariski topology of $X$, indeed the
construction of etale charts in
Cadman \cite{Cadman} actually gives a local quotient structure. 
Over a neighborhood in $X$ where $D_i$ have equations $f_i=0$, the chart
is the subvariety of $X\times \aaa ^k$ given by $f_i = u_i^{n_i}$. 

Within the local charts, one can remark that
there is a standard divisor denoted $R= R_1+ \ldots + R_k$ in $ X[\frac{D_1}{n_1}, \ldots , \frac{D_k}{n_k}]$, and 
$n_i\cdot R_i = p^{\ast}(D_i)$ where $p$ is the projection from the Cadman stack back to $X$. In particular if 
all the $n_i$ are the same $n$ then $n\cdot = D$. Note also that $R$ has normal crossings, as can be seen in the local charts.

\begin{lemma}
\label{coveringhowto}
Suppose $f:Y\rightarrow X$ is a finite Galois covering from a normal variety, 
unramified outside $D$,
with Galois group $\Phi$ corresponding to a representation $\varphi : \pi _1(X-D)\rightarrow \Phi$. Then $f$ lifts to a map $\tilde{f}:Y\rightarrow Z$
if and only if, for each point $x\in D_{i_1}\cap \cdots \cap D_{i_r}$ 
the kernel of the map from the 
local fundamental group 
$$
\zz ^r\rightarrow \Phi
$$
is contained in $n_{i_1}\zz \oplus \cdots \oplus n_{i_r}\zz\subset \zz ^r$. 
The map $\tilde{f}$ is an etale covering space if and only if 
equality holds for the kernel at each point $x$. In this case $Y$ is smooth.
\end{lemma}

The ``Kawamata covering lemma'' \cite[Theorem 17]{Kawamata} gives us projective
varieties covering the root stack. I first learned about this kind of idea when
reading Viehweg's paper \cite{ViehwegVanishing}, even though his technical
approach, investigating further the singularities of purely cyclic coverings
arising from the crossing points, is different from Kawamata's.  

\begin{lemma}
\label{coveringlemma}
If $X$ is a smooth variety with simple normal crossings divisor $D=D_1+\ldots +D_k$,
and if $Z= X[\frac{D_1}{n_1}, \ldots , \frac{D_k}{n_k}]$ is a root stack, then
for any $z\in Z$
there exists a smooth variety $Y$ with a finite, flat morphism 
$r:Y\rightarrow Z$ such that 
$r$ is a finite etale covering over a neighborhood of $z$. 
\end{lemma}
\begin{proof}
Recall the procedure from \cite{Kawamata}. For each divisor component $D_i$,
choose a very ample divisor $K_i$ such that $D_i+K_i$ is a multiple of $n_i$ in
${\rm Pic}(X)$. Then choose representatives $K_i^j\sim K_i$, such that the full divisor
$$
\Dd _K := \sum _i D_i + \sum _{i,j} K_i^j
$$
has normal crossings. For each $i,j$ there is a cyclic covering branched 
along $D_i$ and $K_i$ determined by
choosing an $n_i$-th root of $D_i+K^j_i$. These coverings determine subgroups of 
$\pi  _1(X-\Dd _K)$, and Kawamata shows (in a more algebraic notation)
that if enough $K^j_i$ are chosen for
each $i$, then 
the intersection of all of these subgroups satisfies the
conditions of Lemma \ref{coveringhowto}. That gives a smooth variety $Y$ branched
over $\Dd _K$ and mapping to the root stack over $\Dd _K$. Composing with the projection
$$
Y\rightarrow X[\ldots , \frac{D_i}{n_i},\ldots , \frac{K^j_i}{n_i},\ldots ]
\rightarrow X[\ldots , \frac{D_i}{n_i},\ldots ]
$$
gives a finite flat map $r$. The $K^j_i$ are chosen arbitrarily in very ample 
linear systems, so we can assume that they miss the given point $z$,  
in which case $r$ will be etale over $z$. 
\end{proof}

The Cadman-Vistoli root stack satisfies a good extension property for morphisms. 

\begin{proposition}
\label{cadman-map}
Suppose $(X,D)$ is a smooth variety with normal crossings divisor as above. Suppose $Y$ is an irreducible  DM-stack with 
coarse moduli space 
$Y^{\bf c}$. 
Suppose given a diagram 
$$
\begin{array}{ccc}
X-D & \rightarrow & Y \\
\downarrow && \downarrow \\
X & \rightarrow & Y^{\bf c}.
\end{array}
$$
Suppose $n_i$ are strictly positive integers. A lifting over the root stack 
$$
\tilde{f}: X[\frac{D_1}{n_1},\ldots , \frac{D_k}{n_k}] \rightarrow Y
$$
fitting into commutative diagrams with the given maps, is unique up to unique isomorphism
if it exists. Furthermore, there exists a choice of $n_i$ such that a lifting exists. 
\end{proposition}
\begin{proof}
Consider first the unicity statement when $n_i=1$, i.e. for extensions to $X$. 
For this, we can localize 
in the etale topology over $Y^{\bf c}$. By Keel-Mori, this means that we can assume $Y= Z\stackquot G$ for a finite group $G$ acting on an algebraic space $Z$.
The given map $X-D\rightarrow Z\stackquot G$ corresponds to a pair $(T,\phi )$ where
$T$ is a $G$-torsor on $X-D$ and $\phi :T\rightarrow Z$ is $G$-equivariant. 
An extension to $X$ consists of $(\overline{T},\overline{\phi})$ where
$\overline{T}$ is an extension of $T$ to a $G$-torsor on $X$ and $\overline{\phi}$
extends $\phi$. Since $X$ is smooth---indeed geometrically unibranched would be
sufficient here, a preview of the phenomenon to be met in 
Theorem \ref{unibranch} later---the extension $\overline{T}$ is unique up to unique isomorphism, and of course $\overline{\phi}$
is unique since $X$ contains no embedded points.  This shows the unicity up to unique isomorphism for extensions from $X-D$ to $X$. 

For extensions over a root stack, use local smooth charts for the root stack and unicity
of the extension on these charts from the previous paragraph, to get unicity up to
unique isomorphism for extensions 
$$
X[\frac{D_1}{n_1},\ldots , \frac{D_k}{n_k}] \rightarrow Y.
$$

Now to construct an extension, in view of the unicity, we can localize in the etale topology over $X$, hence we can also localize in the etale topology over $Y^{\bf c}$.
Therefore assume $Y= Z\stackquot G$ for a finite group $G$ acting on an algebraic space $Z$. In this case $Y^{\bf c}= Z/G$ is
the usual quotient. 

The map $X-D\rightarrow Z\stackquot G$ corresponds to a pair $(T,\phi )$
where $T$ is a $G$-torsor on $X-D$ and $\phi : T\rightarrow Z$ is a $G$-equivariant
map. The local fundamental group of $X-D$ 
near a point $x\in D_{i_1}\times \cdots \times D_{i_r}$ of $D$ is of the form $\zz ^r$,
but $G$ is finite so its action
on $T$ factors through a quotient of the form $\zz /n_{i_1}\times \zz /n_{i_r}$.
Let $n_i$ be a common multiple of the integers appearing here for all points of 
$D_i$. Then $T$ extends to a torsor over the root stack 
$$
\overline{T}\rightarrow X[\frac{D_1}{n_1},\ldots , \frac{D_k}{n_k}].
$$
Note that the total space of $\overline{T}$ itself is a smooth algebraic space,
and the inverse image of the divisor is a divisor with normal crossings $R\subset \overline{T}$. It remains to extend $\phi$. However, $\overline{T}$ is 
a normal space and $Z\rightarrow Z/G$ is a finite map.
It follows that one can extend the given map $\phi : \overline{T}-R\rightarrow Z$ to 
a map $\overline{\phi}:\overline{T}\rightarrow Z$,
from knowing that the extension $\overline{T}\rightarrow Z/G$ exists. 

This may be seen on local affine charts:
write $\overline{T}={\rm Spec}(A)$, $Z={\rm Spec}(B)$, so $\overline{T}-R = 
{\rm Spec}(A_g)$ where $g$ is the function defining the divisor $R$, 
and $Z/G= {\rm Spec}(B^G)$. The extension $B^G\subset B$ is finite,
and the map $B\rightarrow A_g$ sends $B^G$ to $A$, it follows from normality of $A$
that $B$ maps into $A\subset A_g$, in a unique way hence $G$-equivariantly. 
This provides the required map 
$\overline{T}\rightarrow Z$ corresponding to an extension
$$
X[\frac{D_1}{n_1},\ldots , \frac{D_k}{n_k}]
\rightarrow Z\stackquot G .
$$
Going back to $Y$ and globalizing over $Y^{\bf c}$ gives the required extension
to prove the lemma. 
\end{proof}

\section{Proper surjective hypercoverings by smooth projective varieties} 
\label{sec-spcs}

We use the notations of \cite{hodge3}. 
Suppose $X_{\hdot}\stackrel{a}{\rightarrow} S$ is an augmented simplicial scheme. For each $k\geq 0$, the coskeleton construction
defines the
{matching object} 
$$
{\rm match}_k(X_{\hdot}\rightarrow S):= \csk (\sk _{k-1}X_{\hdot} )_{k}
$$
in Deligne's notation \cite{hodge3}, and we have a natural ``matching'' map 
\begin{equation}
\label{maptomatch}
X_k\rightarrow {\rm match}_k(X_{\hdot}\rightarrow S).
\end{equation}
At $k=0$ the matching map is just $X_0\rightarrow S$ and at $k=1$ it is
$X_1\rightarrow X_0\times _SX_0$. For $k\geq 2$ the matching map is independent of
the augmentation $X_0\rightarrow S$.

A morphism $X_{\hdot}\rightarrow S$ is a {proper surjective hypercovering}
if the matching maps are proper surjections, if $X_0\rightarrow S$ is a proper 
surjection, and if $X_{\hdot}$ has split
degeneracies. An etale hypercovering is given by requiring that the
matching maps be coverings in the etale topology, i.e. admit sections etale-locally. 

\begin{remark}
\label{matchstack}
The notion of proper surjective (resp. etale) hypercovering 
extends to the case where $S$ is a separated DM-stack, indeed then 
$X_0\times _SX_0$ is an algebraic space so the proper surjectivity of the matching
maps at $X_0$ and $X_1$ are well-defined conditions. 
\end{remark}

It is a well-known fact that coverings of proper DM-stacks by projective varieties
exist \cite{KreschVistoli} \cite{OlssonCoverings}:

\begin{lemma}
\label{easycover}
If $X$ is a proper DM-stack then there exists a surjective proper map
$Z\rightarrow X$ where $Z$ is a smooth projective variety.
\end{lemma}
\begin{proof}
One could first apply the general existence of proper coverings \cite{OlssonCoverings}
and then apply the Chow lemma and resolve singularities; or alternatively,
resolve first the singularities of $X$ and then apply Theorem \ref{maincovering} below.
\end{proof}

\begin{theorem}
\label{fullresolution1}
A proper DM-stack $X$ admits a proper surjective hypercovering with split degeneracies, 
by smooth
projective varieties.  
Any two such hypercoverings can be topped off by a third one. 
\end{theorem}
\begin{proof}
Use the previous lemma to choose $Z_0\rightarrow X$.
Notice that $Z_0\times _XZ_0\rightarrow Z_0\times Z_0$ is finite so
$Z_0\times _XZ_0$ is a projective variety. Continue from there using Deligne's technique
\cite{hodge3}.
\end{proof}

Suppose $f:Z\rightarrow X$ is a morphism of DM stacks. We say that $f$ is {surjective where etale} if,
 letting $Z'\subset Z$ be the open substack where $f$ is etale, we have $Z'\rightarrow X$ surjective.
This is equivalent to saying that any point $x\in X$ admits at least one lift $z\in Z$ such that $f$ is etale near $z$.
On the other hand 
$f$ will not be etale or even finite in a neighborhood of a different point of the fiber over $x$.

The precise form of the Chow lemma proven by Raynaud and Gruson in \cite{RaynaudGruson} allows us to 
obtain a good form of the
Chow lemma for smooth proper DM-stacks. This improves Deligne-Mumford's 
statement 4.12 of \cite{DeligneMumford}, or 
rather it gives a refined  
statement which, if they had given a proof, they would undoubtedly have proven along the way. See also Kresch-Vistoli \cite{KreschVistoli}, Olsson, and Starr \cite{OlssonStarr},
\cite{OlssonCoverings} for statements about existence of coverings.

The techniques of \cite{RaynaudGruson} have been applied in many similar situations.
See Rydh \cite{Rydh} for a recent application, and 
de Jong \cite{DeJong} for a more classical utilisation. It is also interesting to note the
extensive and detailed AMS review of \cite{RaynaudGruson} by Masaki Maruyama.
Maehara refers to these techniques, while speaking of Kawamata-Viehweg coverings, in his paper \cite{Maehara}.

\begin{theorem}
\label{maincovering}
Suppose $X$ is a smooth and proper DM-stack of finite type. Then there exists a morphism $f:Z\rightarrow X$ 
such that
$f$ is surjective where etale and $Z$ is a smooth projective variety.
\end{theorem} 
\begin{proof} 
Fix a point $x\in X$. We will find $f:Z\rightarrow X$ with a lifting $z\in Z$ of $x$ such that $Z$ is smooth 
projective and $f$ is etale at $z$. 

Start with an etale neighborhood $p:U \rightarrow X$ with a lifting $u\in U$ of the point $u$, $U$ an affine 
scheme of finite type over $\cc$, 
and $p$ an etale morphism. This exists by the
definition of DM-stack. Note that $U$ is smooth and quasiprojective.

Let $U\subset P$ be a completion to a smooth projective variety. 
Let $C\subset P\times X^{\bf c}$ be the closure of the graph of the map 
$U\rightarrow X^{\bf c}$ to the coarse moduli space. It is a proper algebraic
space containing $U$ as a Zariski open subset.

By resolution of singularities for algebraic spaces \cite{Villamayor} 
\cite{BierstoneMilman} we can resolve the 
singularities of $C$ 
without touching $U$, which gives a diagram of algebraic spaces
$$
\begin{array}{ccc}
U & \hookrightarrow & Y \\
& \searrow & \downarrow \\
&& C
\end{array}
$$
where $Y$ is a smooth proper algebraic space and $D:= Y-U$ is a divisor with normal crossings. Write 
$D= D_1 + \ldots + D_k$ and we may assume that the $D_i$ 
are irreducible and smooth. Raynaud and 
Gruson \cite[Cor. 5.7.14]{RaynaudGruson}, refering also to 
Knutson \cite{Knutson}, prove this version of the Chow lemma: if
$Y$ is a separated proper algebraic space and $U\subset Y$ is an open subset such that $U$ is a quasiprojective variety, then
there is a blow-up $\widetilde{Y}\rightarrow Y$ which is an isomorphism over $U$ such that $\widetilde{Y}$ is projective.
This means that after replacing $Y$ by $\widetilde{Y}$ which is the same over $U$, 
then again resolving singularities of the complementary divisor, we can 
suppose that $Y$ is projective. 

We are now in the situation of Lemma \ref{cadman-map} with a diagram
$$
\begin{array}{ccc}
U & \rightarrow & X \\
\downarrow && \downarrow \\
Y & \rightarrow & X^{\bf c}.
\end{array}
$$
so there there is $n$, which for convenience can be assumed the same for all
divisor components, such that the map extends over the root stack
to a map
$$
Y[\frac{D_1}{n}, \ldots , \frac{D_k}{n}] \rightarrow X.
$$

We next note that there is a morphism from a projective scheme 
$Z\rightarrow Y[\frac{D_1}{n}, \ldots , \frac{D_k}{n}]$,
which is finite and projects to a cover of $Y$ ramified along a subset which misses $u$.
This is exactly the covering lemma \ref{coveringlemma}. 

Hence, for any point $z\in Z$ lying over $u\in U$, 
the morphism $Z\rightarrow U$ is etale 
at $z$. Thus, we obtain a map $Z\rightarrow X$ as desired for the proof of the theorem relative to one point. It 
suffices to take a finite
disjoint union of such varieties $Z$ in order to get a map surjective where etale. 
\end{proof}

There is 
a variant of this result which takes into account a divisor with normal crossings. 

\begin{proposition}
Suppose $X$ is a smooth proper and separated DM-stack of finite type and $D\subset X$ a divisor with normal crossings. 
Then there exists a morphism $f:Z\rightarrow X$ 
such that
$f$ is surjective where etale, and $Z$ is a smooth projective variety, such that $f^{\ast}(D)$ is a divisor whose associated reduced divisor has normal crossings. 
\end{proposition} 
\begin{proof}
Left to the reader.
\end{proof}

\begin{question}
\label{singularquestion}
To what extent does the statement of Theorem \ref{maincovering} hold for
DM-stacks which are not smooth? 
\end{question}

The above construction of covering spaces provides the starting point for 
the construction of a hypercovering. Follow the technique of \cite{hodge3}
with a little extra care at the first stage
to conserve a trace of the surjective-where-etale property.

Suppose $X$ is a smooth proper DM stack. Let $p:Z\rightarrow X$ be a
morphism from a smooth projective variety, such that 
$X$ is covered by the open set $Z'$ where $p$ is etale.

Consider $R_1:= Z\times _XZ$ and $K_1:= Z\times _X Z\times _XZ$. These are algebraic spaces.
Note that $Z$, $R_1$ and $K_1$ form the first part of a simplicial object. However, $R_1$ and $K_1$ are
not smooth. 

\begin{lemma}
In fact $R_1$ and $K_1$ are projective.
\end{lemma}
\begin{proof}
The map $R_1\rightarrow Z\times Z$ is finite because $X$ is a proper, whence separated DM stack.
Since $Z$ is projective, we get that $R_1$ is projective. The same holds for $K_1$. 
\end{proof}

Let $R':= (Z'\times _XZ) \cup (Z \times _XZ')$. Then $R'\subset R_1$ is
a smooth open subset of $R_1$. 
It decomposes 
$$
R' = R^{\prime , N} \sqcup Z'
$$
where $Z'\rightarrow R'\rightarrow Z\times _XZ$ is the diagonal.
Notice that $Z'\rightarrow R'$ is a closed immersion, so it is one of the
connected components of $R'$. 

Let $R\rightarrow R_1$ be a resolution of singularities of the union of
irreducible components of $R_1$ which meet $R'$, and  which is an 
isomorphism over $R'\subset R$. Choose in particular $Z$ as the completion of the
component $Z'\subset R'$. In this way, $R'\subset R$ is an open dense subset,
and 
$$
R=R^N\sqcup Z.
$$
This insures the split degeneracy condition for $R$. 

The matching object at the next stage is the equalizer 
$$
M \rightarrow R\times R\times R \twoarrows Z\times Z\times Z
$$
of the two maps sending $(u,v,w)$ to $(\partial _0u,\partial _0v, \partial _1w)$
and $(\partial _0w,\partial _1u,\partial _1v)$ respectively.  

Notice that $M$ is  projective. 
Rather than continue with a more careful choice such as was done with $R$
(and which the reader is invited to do),
we can just set $K^N\rightarrow M$ to be a resolution of singularities from a
smooth projective variety, isomorphism over the smooth locus. Put
$$
K_2:= K^N\sqcup R^N\sqcup R^N \sqcup Z.
$$
Notice that $K':=Z'\times _XZ'\times _XZ'$ splits in the same way and we can choose
a map $K'\rightarrow K_2$ respecting this decomposition. Let $K$ be the union of
components of $K_2$ containing $K'$. The map $K\rightarrow M$ is still surjective. 

We now have a diagram 
$$
K\threearrows R\twoarrows Z \stackrel{f}{\rightarrow} X
$$
plus the degeneracies going in the other direction, 
which looks like the beginning of an augmented simplicial object.
In other words, the compositions which would be
equal for a simplicial set, are also equal here. For maps into $X$ one must replace ``equality'' by ``isomorphism'',
and be careful about coherences. In the first three terms the elements are smooth projective varieties. Denote 
by $\partial_0,\partial_1$ the two maps from $R$ to $Z$, and by $\partial_{01}$, $\partial_{02}$ and $\partial_{12}$ the three maps from 
$K$ to $R$. Note that we have open dense subsets 
$$
Z'\times _XZ'\subset R' \subset R
$$
and 
$$
Z'\times _XZ' \times _XZ' =K'\subset K,
$$
and these open subsets form the beginning of the standard simplicial object for $Z'\rightarrow X$. The required
equalities of maps $K\rightarrow Z$ follow because these open subsets are dense, and these serve to define
three maps $v_0,v_1,v_2: K\rightarrow Z$:
\begin{equation}
\label{equalities}
v_0:= \partial_0\circ \partial_{01} = \partial_0\circ \partial_{02},\;\;
v_1:= \partial_1\circ \partial_{01} = \partial_0\circ \partial_{12},\;\;
v_2:= \partial_1\circ \partial_{02} = \partial_1\circ \partial_{12}.
\end{equation}

We  have a natural isomorphism $\alpha : f\circ \partial_0 \cong f\circ \partial_1$, and the coherence conditions say that
the hexagon made with three copies $\partial_{ij}^{\ast} \alpha$ and the three equalities above composed with $f$,
commutes. Of course if $X$ were an orbifold, that is a DM stack with trivial generic stabilizer, then
generically surjective maps from an irreducible smooth variety into $X$ wouldn't have any nontrivial isomorphisms,
so in this case there would have been no need to speak of $\alpha$ and the coherence condition. However,
even in this case
we will meet the hexagonal coherence condition when looking at bundles.

\begin{theorem}
\label{fullresolution}
A smooth proper DM-stack $X$ admits a proper surjective hypercovering with split degeneracies 
by smooth
projective varieties, obtained by completing the
partial simplicial object $(Z,R,K)$ constructed above.
Again, any two such hypercoverings can be topped off by a third one. 
\end{theorem}
\begin{proof}
For $X$ smooth,
starting from the first part constructed above, Deligne's technique
\cite{hodge3} allows us to finish. 
\end{proof}

Suppose $Z_{\hdot}$ is a 
simplicial scheme with smooth projective levels, 
and split degeneracies. 
Denote by $(Z,R,K)$ the first three terms. 
Keep the  notations from before, with
two morphisms $\partial_0,\partial_1:Z\rightarrow X$ and three morphisms $\partial_{ij}: R\rightarrow Z$
for $0\leq i < j \leq 2$, and equalities \eqref{equalities} over $K$.
The simplicial object then starts with $Z_0=Z$, $Z_1=R= R^N\sqcup Z$  and 
$Z_2=K= K^N\sqcup R^N\sqcup R^N\sqcup Z$. 

A {descent datum} for $(Z,R,K)$
is a bundle $E$ on $Z$, and an isomorphism $\varphi : \partial_0^{\ast}E\cong \partial_1^{\ast}E$ on $R$
such that the hexagon formed by alternating the $\partial_{ij}^{\ast} \varphi$ with the equalities 
\eqref{equalities}, commutes. A {morphism} between descent 
data $(E,\varphi ) \rightarrow (E',\phi )$ is a morphism $E\rightarrow E'$ commuting with
the isomorphisms. Given a bundle $F$ on $X$, 
its pullback to $Z$ is provided with a
natural descent datum. 

A descent datum for $(Z,R,K)$ according to this definition
is automatically compatible with the degeneracy map
$s_0:Z\rightarrow R$ in the sense that $s_0^{\ast}(\phi )=1_E$.
Indeed, if $s_1$ and $s_2$ denote the two degeneracy maps 
from $R$ to $K$, then 
$$
\partial_1^{\ast}s_0^{\ast}(\phi )\circ \phi = 
s_2^{\ast}\partial_{12}^{\ast}(\phi )\circ s_2^{\ast}\partial_{01}^{\ast}(\phi )
= s_2^{\ast}\partial_{02}^{\ast}(\phi ) = \phi ,
$$
so $\partial_1^{\ast}s_0^{\ast}(\phi )=1_{\partial_1^{\ast}(E)}$ since $\phi$ is invertible,
so
$$
s_0^{\ast}(\phi )=
s_0^{\ast}\partial_1^{\ast}s_0^{\ast}(\phi ) = s_0^{\ast}(1_{\partial_1^{\ast}(E)}) = 1_E.
$$

\begin{lemma}
\label{vbdescent1}
 The category of descent data for vector bundles (maybe with extra structure)
over $Z_{\hdot}$ is equivalent to the category of explicit descent data on
$(Z,R,K)$. 
\end{lemma}
\begin{proof}
A descent datum on $Z_{\hdot}$ restricts in an obvious way to a descent datum on $(Z,R,K)$.
Suppose given a descent datum $(E,\varphi )$ over $(Z,R,K)$. 
The $j$-th vertex map $[0]\rightarrow [k]$ in $\Delta$ induces $v_j:Z_k\rightarrow Z_0=Z$.
A path of edges relating the $i$-th and $j$-th vertices in $[k]$ gives, using $\varphi$,
an isomorphism of bundles between $v_i^{\ast}(E)$ and $v_j^{\ast}(E)$ on $X_k$. 
When two paths differ by the boundary of a $2$-simplex, the equalities required of $\varphi$
imply that the two isomorphisms are the same. But any two paths can be connected by a sequence
of transformations along boundaries of $2$-simplices, so any two paths induce the same isomorphism.
This canonically identifies all of the $v_j^{\ast}(E)$ to a unique bundle which can be called $E_k$.
It is now easy to see that the $E_k$ are naturally functorial for pullbacks along the simplicial maps
$X_k\rightarrow X_m$. This constructs the essential inverse to the restriction functor. 
\end{proof}

\begin{remark}
\label{ztimesz}
If $(Z,R,K)$ is the start of a proper surjective hypercovering of a
proper DM-stack $X$, and $(E,\varphi )$ is a descent datum for a bundle or
local system on $(Z,R,K)$. Then $\varphi$ determines a continuous isomorphism
between ${\rm pr}_1^{\ast}(E)$ and ${\rm pr}_2^{\ast}(E)$ over $Z\times _XZ$.
\end{remark}
\begin{proof}
The map $R\rightarrow Z\times _XZ$ is proper and surjective, with $\varphi$ defined
over $R$. There is a map from $R\times _{(Z\times _XZ)}R$ to the matching
object: given $r,r'\in R$ mapping to $(z,z')\in Z\times _XZ$, associate $(r,1_{z'},r')$
in the matching object. As $K$ surjects to the matching object, and 
$\varphi (1_{z'})=1_{E(z')}$, the cocycle condition over $K$
says that $\varphi (r')=\varphi (1_{z'})\circ \varphi (r)$. Thus, $\varphi$ descends
as a continuous function to $Z\times _XZ$. 
\end{proof}

Return now to the hypothesis that $X$ is a smooth DM-stack, 
and $Z_{\hdot}\rightarrow X$ is a proper surjective hypercovering
with $Z_k$ smooth projective, 
starting off with $(Z,R,K)$ chosen according to
the procedure described before Theorem \ref{fullresolution}.
Thus $Z\rightarrow X$ is assumed to be surjective where
etale, and $R$ chosen as a completion of the smooth $R'$ as above. 

There is also a natural pullback 
functor from bundles on $X$ to descent data on $(Z,R,K)$. 

When $X$ is smooth, the extra information given by the surjective-where-etale property
allows us to transfer analytic constructions from $Z_{\hdot}$ back to $X$,
to get things like the definition and existence of harmonic metrics. It might
be possible to descend these things along proper surjective hypercoverings too,
and in that way get around Theorem \ref{maincovering} entirely, but that would
require a much more detailed study of descent for bundles along proper surjective maps,
a subject discussed in \cite{B2S1} \cite{B2S2}.

Notice that $Z_{\hdot}$ may be chosen to contain an etale hypercovering as
an open simplicial subvariety, but is not itself an etale hypercovering.
This is because the places where the maps are not etale lead to singularities
in the matching objects, so
one needs to use resolution of singularities at each stage in order to have $Z_k$ 
smooth. Nonetheless an explicit treatment of the first part of the resolution
yields descent for bundles.

\begin{lemma}
\label{vbdescent2}
The category of bundles on $X$ is naturally equivalent via this pullback functor to the category 
of descent data on $Z_{\hdot}$ or the partial simplicial object $(Z,R,K)$.
\end{lemma}
\begin{proof}
Recall that $Z'\subset Z$ is the Zariski dense 
open set over which the projection $p$ is
etale. 
By descent for the map $Z'\rightarrow X$, we obtain a bundle $F$ over $X$, with
an isomorphism 
$\psi ': p^{\ast}(F)|_{Z'}\cong E|_{Z'}$ compatible with the descent data over $Z'$.
 
Recall that $R':= Z'\times _XZ \cup Z \times _XZ'$ is a smooth open subset of $R$,
which itself contains $Z'\times _XZ$ as an open subset. The descent datum
yields an isomorphism $\partial_1^{\ast}(E|_{Z'})\cong \partial_2^{\ast}E$ over $Z'\times _XZ$,
but 
$$
\partial_1^{\ast}(E|_{Z'})\cong \partial_1^{\ast}(p^{\ast}(F)|_{Z'})\cong p_R^{\ast}(F)|_{Z'\times _XZ}
$$
where $p_R:R\rightarrow X$ is the projection,
whence 
$$
p_R^{\ast}(F)|_{Z'\times _XZ} \cong \partial_2^{\ast}(E)|_{Z'\times _XZ}.
$$
The map $\partial_2:Z'\times _XZ\rightarrow Z$ is an etale covering, and 
we are now given an isomorphism between $p^{\ast}(F)$ and $E$, locally with respect
to this covering. Using the cocycle condition for the descent data over $K$,
this isomorphism is the same as the previous one over $Z'$. The property of 
extending an isomorphism from a Zariski open set to the whole of $Z$ is etale-local
on the complementary closed subset, so this shows that our isomorphism extends to 
a global isomorphism  $\psi : p^{\ast}(F)\cong E$ on $Z$. It is compatible with the
given descent data since the open subsets $Z'$ and $R'$ are dense. 

We have shown that the pullback functor from bundles to descent data is
essentially surjective. To show that it is fully faithful, given a
morphism between descent data it restricts to a morphism between descent
data on $(Z',Z'\times _XZ')$ so descends to a morphism between bundles.
\end{proof}

Suppose $\lambda \in \cc$. 
A {$\lambda$-connection} on a descent datum $(E,\varphi )$ is a $\lambda$-connection $\nabla$ on $E$,
such that $\varphi$ intertwines the pullbacks $\partial_0^{\ast}\nabla $ on $\partial_0^{\ast}E$ and 
$\partial_1^{\ast}\nabla $ on $\partial_1^{\ast}E$. This definition extends to objects defined over a base scheme $S$,
for any $\lambda \in \Gamma (S,\Oo _S)$. 

\begin{lemma}
Suppose $(E,\varphi )$ is the descent datum corresponding to a bundle $F$ on $X$. 
Given a $\lambda$-connection $\nabla _E$ on $(E,\varphi )$ there is a unique $\lambda$-connection 
$\nabla _F$ on $F$ such that $f^{\ast}\nabla _F = \nabla _E$ via the isomorphism $f^{\ast}F\cong E$.
\end{lemma}
\begin{proof}
As before, if $Z$ were replaced by $Z'$ this would be the classical etale descent.
In particular, $\nabla _E |_{Z'}$ descends to a unique connection $\nabla _F$ on $F$. But now,
$f^{\ast}\nabla _F$ and $\nabla _E$ are two $\lambda$-connections on $E$ over the smooth variety $Z$,
with the same restriction to the dense open subset $Z'$. Therefore they are equal. Unicity of $\nabla _F$
follows by descent only over $Z'$.  
\end{proof}

We can similarly descend $\Cc ^{\infty}$ vector bundles, hermitian metrics on them, and differential-geometric structures such as 
differential operators. It is left to the reader to formulate these statements.

Cohomological descent along proper surjective hyperresolutions was the main technique
used by Deligne to apply Hodge theory to the
topology of singular varieties. It is the main reason for
looking at simplicial schemes, but the same techniques also apply
to get proper surjective hyperresolutions for DM-stacks as stated in 
Theorem \ref{fullresolution} above.

Proper surjective cohomological descent \cite{SaintDonat}
\cite{Vermeulen} \cite{hodge3} then says that for any local system $L$ on $S$,
\begin{equation}
\label{cohdescent}
H^i(S,L)\stackrel{\cong}{\rightarrow} H^i(X_{\hdot} , a^{\ast}L).
\end{equation} 

\begin{lemma}
\label{cohdescentstack}
The isomorphism \eqref{cohdescent}
holds also in the case when $S$ is a separated DM-stack.
\end{lemma}
\begin{proof}
Choose an etale hypercovering $Z_{\hdot}\rightarrow S_{\hdot}$,
then we get a bisimplicial algebraic space $\{ X_k\times _X Z_l\}_{(k,l)\in \Delta \times \Delta}$. Cohomological descent for the proper surjective topology gives cohomological
descent in the $k$-variable down to $Z_l$, and the etale hypercovering induces an equivalence of
realizations so we have cohomological descent in the $l$-variable, down  
to $X_k$ and to $S$. These allow us to conclude by a spectral sequence argument. 
\end{proof}

\begin{lemma}
\label{lsdescends}
Suppose $X_{\hdot}\rightarrow S$ is a proper surjective
hypercovering to a separated DM-stack. 
If $L$ is a local system of sets over $X_{\hdot}$, then it descends: there exists a local system $L_S$
on $S$ such that $a^{\ast}L_S\cong L$.
\end{lemma}
\begin{proof}
Suppose first that $S$ itself is a separated scheme of finite type over $\cc$. 
Suppose $y\in S$. Let $X_{\hdot}(y)$ denote the fiber
over $y$, that is $X_k(y)$ is the inverse image of $y$ in $X_k$. It is nonempty.

Working with local systems and the fundamental group involves only the pieces
$X_0,X_1,X_2,X_3$ of the hypercovering, so for the purposes of the present
argument we truncated. This makes it so there are only finitely many $k$. 

The descent data for the local system over the hypercovering
imply that $L$ is trivial, i.e. isomorphic to a constant local system, when restricted
to $X_{\hdot}(y)$. To see this, choose a lifting $z\in X_0$ mapping to $y$,
and note that the restriction of $L$ to $X_0\times _SX_{\hdot}$ from the 
second factor, is isomorphic to the pullback of $L|_{X_0}$ from
the first factor. Restricting to $\{z\}\times _SX_{\hdot}=X_{\hdot}(y)$
gives the desired trivialization $L|_{X_{\hdot}(y)}\cong \underline{L_0(z)}$. 
Note that this trivialization is compatible with the descent data. 

Next, note that there exist usual open neighborhoods 
$X_k(y)\subset W_k\subset X_k$
such that $L|_{W_k}$ has a trivialization compatible with the one
constructed previously on $X_k(y)$. To see this, choose an open covering
$U^i_k$ on which $L$ is trivialized, then refine it so that any nonempty
$U^i_k\cap X_k(y)$ are connected. The trivialization of 
$L|_{X_k(y)}$ induces a well-defined trivialization of each $L|_{U^i_k}$
for those open sets $U^i_k$ meeting $X_k(y)$; for the other open sets choose
arbitrarily. Pass then to relatively
compact $V^i_k\subset U^i_k$ which still cover $X_k$. Now, for any connected
component of some $V^i_k\cap V^j_k$ on which the transition isomorphism $g_{ij}$
for $L$ is not the identity, the
closure of that connected component misses $X_k(y)$. Taking the complement of the
closures of such connected components of $V^i_k\cap V^j_k$ gives a neighborhood 
$W_k$ of $X_k(y)$ covered by $V^i_k\cap W_k$, such that the transition functions 
for $L|_{W_k}$ with respect to the covering and the given trivializations, are 
identities. Patching together gives a global trivialization of $L|_{W_k}$
compatible with the previous one on $X_k(y)$.

The previous trivializations of $L|_{X_k(y)}$ were compatible with the descent
data, so possibly reducing the size of $W_k$ we may assume that this is true
for our trivializations of $L|_{W_k}$.
Properness of the finitely many maps $X_k\rightarrow S$ in play, implies that there
is a usual open neighborhood $y\in S'\subset S$ such that the inverse image of
$S'$ (which we call $W'_k$) is contained in $W_k$. The trivialization of $L|_{W'_k}$ then becomes an
isomorphism between $L|_{W'_k}$ and the pullback of the constant local system on $S'$,
compatible with the descent data. In other words, we have descended 
$L|_{W'_k}$ to a constant local system on $S'$. For any point $y\in S$ we obtain such
a neighborhood, and the descended local system is unique up to canonical isomorphism
(as can be seen by proper surjective descent for sections of local systems).
Hence the descended local systems on neighborhoods glue together to give a descent of
the local system to $L_S$ on $S$ whose pullback to $X_{\hdot}$ is $L$. 
\end{proof}

\begin{proposition}
\label{topdescentstack}
Suppose $X_{\hdot}\rightarrow S$ is a proper surjective
hypercovering to a separated DM-stack. Then this induces a
map on topological realisations
$|X_{\hdot}|\stackrel{|a|}{\longrightarrow} |S|$ which is a weak homotopy equivalence. 
\end{proposition}
\begin{proof}
The induced map $|a|$ is from \eqref{simpmap} above.  
To prove that $|a|$ is a weak equivalence, using Quillen's criterion and
cohomological descent (Lemma \ref{cohdescentstack} above), it suffices to verify in
addition that $|a|$ induces an isomorphism on fundamental groups at any basepoint 
$x\in X_0$. This is shown by the preceding lemma. 
\end{proof}

\begin{remark}
\label{simpproj}
Suppose $S_{\hdot}$ is a simplicial projective variety or even a simplicial
object in $\DMS$. Then there is a map $X_{\hdot}\rightarrow S_{\hdot}$ which is
a weak equivalence for the proper surjective topology, with the $X_k$ being
smooth projective varieties. So topologically speaking we don't lose any
generality by passing to simplicial smooth projective varieties.
\end{remark}

One can also define the de Rham cohomology of a scheme or stack, using 
some form of crystalline cohomology, see
\cite{TelemanSimpson} and  \cite{OlssonCrystalline} for example. 

In \cite{Tsuzuki}, it is shown that the cohomological descent isomorphism \eqref{cohdescent}
also holds for de Rham cohomology. A bisimplicial argument 
shows that this is also true when $S$ is a separated DM-stack.

To complete the picture of de Rham descent, we note that the analogue of Lemma
\ref{lsdescends} also holds.  

\begin{lemma}
\label{vbicdescends}
Suppose $X_{\hdot}\rightarrow S$ is a proper surjective
hypercovering to a separated DM-stack. Suppose $F_{\hdot}$ is a
compatible system of de Rham local systems with regular singularities
on $X_{\hdot}$, that is
$F_k$ is a stratification on the crystalline site of $X_k$ provided with
pullback isomorphisms $F_k|_{X_m}\cong F_m$ whenever $m\rightarrow k$ in $\Delta$.
Then there exists a de Rham local system $G$ on $S$ with $a^{\ast}(G)\cong F_{\hdot}$.
\end{lemma}
\begin{proof}
The de Rham local systems with regular singularities correspond to local systems by
the Riemann-Hilbert correspondence (see \cite{TelemanSimpson} about this question for stacks),
so this follows from Lemma
\ref{lsdescends}. It would be interesting to have a direct algebraic proof, which could apply
to irregular de Rham local systems too. 
\end{proof}

A further interesting question is the descent of vector bundles along proper surjective
hypercoverings. Some understanding of this issue would be helpful in order to obtain
an algebraic construction of descent for de Rham local systems.

\section{Moduli of local systems on simplicial varieties}

Let $\ArtSt$ denote the $2$-category of Artin algebraic stacks of finite type. 
Suppose $\Cc$ is a category, and suppose
given a $2$-functor $\bF : \Cc ^o\rightarrow \ArtSt$.
Suppose $X_{\hdot}$ is a simplicial object of $\Cc$, that is to say a functor $\Delta ^o\rightarrow \Cc$. Then we can define $\bF (X_{\hdot})$ as the $2$-limit of
the diagram $\bF \circ X_{\hdot}:\Delta \rightarrow \ArtSt$.

Concretely an object $E_{\hdot}$ of $\Ff (X_{\hdot})$ consists of 
a collection of objects $E_k$ of $\Ff (X_k)$ together with
isomorphisms $X_{\phi}^{\ast}(E_k)\cong E_m$ whenever $\phi : k\rightarrow m$
is a map in $\Delta$ inducing $X_{\phi}:X_m\rightarrow X_k$,
and these isomorphisms are required to satisfy the obvious compatibility
conditions for compositions $k\rightarrow m \rightarrow l$ and identities.

\begin{lemma}
\label{limdiag}
The $2$-limit $\Ff (X_{\hdot})$
depends only on the start of the simplicial object,
in fact it is the $2$-limit of the diagram 
$$
\Ff (X_0) \twoarrows \Ff (X_1 )\threearrows \Ff (X_2) .
$$
An object $E_{\hdot}$ of $\Ff (X_{\hdot})$ may also be viewed
as just
an object $E_0$ of
$\Ff (X_0)$ together with an isomorphism $\partial _0^{\ast}E_0\cong 
\partial _1^{\ast}E_0$, over $X_1$,
satisfying the cocycle condition when pulled back to $X_2$.
\end{lemma}
\begin{proof}
The same as for Lemma \ref{vbdescent1} (see the paragraph just before \ref{vbdescent1} 
for why we don't need to include the degeneracies in the diagram).  
\end{proof}

The terminology ``simplicial family'' will sometimes be useful to describe objects
of the form $E_{\hdot}$. Morphisms in $\Ff (X_{\hdot})$ have corresponding descriptions. 

The above construction 
applies in particular to the category $\Cc$ of smooth projective varieties.
Various functors include:
\newline
$X\mapsto \Mm _B(X,G)$ the moduli stack of representations
of $\Pi _1(X)$ in an algebraic group $G$;
\newline
$X\mapsto \Mm _{DR}(X,G)$ the moduli stack of pairs $(P,\nabla )$ where $P$ is a
principal $G$-bundle and $\nabla$ an integrable algebraic connection;
\newline
$X\mapsto \Mm _{H}(X,G)$ the moduli stack of pairs $(P,\theta )$ where $P$ is a
principal $G$-bundle with $\theta$ an integrable Higgs field of semiharmonic type (see Definition \ref{semiharmonic});
\newline
$X\mapsto \Mm _{Hod}(X,G)$ the moduli stack of triples  $(\lambda , P,\nabla )$ where $P$ is a
principal $G$-bundle and $\nabla$ an integrable algebraic $\lambda$-connection
of semiharmonic type
(which specializes to the preceding two in the cases $\lambda =1,0$);
\newline
$X\mapsto \Mm _{DH}(X,G)$ the analytic Deligne-Hitchin
moduli stack obtained by glueing two copies of $\Mm _{Hod}$, noting that here
we use $\ArtSt ^{\rm an}$ the $2$-category of analytic Artin stacks.

The $2$-limit construction gives: 

\begin{proposition}
\label{modulistacks}
These functors extend to moduli stacks of various types of local systems
denoted $\Mm _{\eta}(X_{\hdot},G)$ for $\eta = B,DR,H,Hod,DH$,
defined for a simplicial
object $X_{\hdot}$ in the category of smooth projective varieties and 
a linear algebraic group $G$. 
\end{proposition}

A more explicit description may also be given. The notion of local system
on $X_{\hdot}$ was discussed above, and indeed it applies to local systems
with values in any
$1$-groupoid. If $S$ is a scheme then $BG(S)$ is the groupoid
of $G$-torsors over $S$, and 
$\Mm _B(X_{\hdot}, G)(S)$ is the $1$-groupoid of local systems with values in 
$BG(S)$. In other words, an object in $\Mm _B(X_{\hdot}, G)(S)$ consists of a locally
constant sheaf of $G$-torsors over $S$ on each space $|X_k|$ together with isomorphisms
between the pullbacks functorial in $k\in \Delta$. 

For a base scheme $S$ with a group scheme $G/S$ and function 
$\lambda : S\rightarrow \aaa ^1$, a {principal $G/S$-bundle with $\lambda$-connection} on $X_{\hdot}\times S/S$ is a pair $(P_{\hdot},\nabla )$ where 
$P_{\hdot}$ is a collection of principal $\partial_1^{\ast}(G)$-bundles $P(k)$ on $X_k\times S$,
with $\lambda$-connections $\nabla$ relative to $S$ on each $P(k)$, and compatibility
isomorphisms $(P(k),\nabla )\cong (X_{\phi})^{\ast}(P(m),\nabla )$
whenever $\phi : m\rightarrow k$ in $\Delta$, compatible with compositions of $\phi$.

If $G$ is a fixed linear algebraic group scheme then apply the previous
paragraph with $G\times S/S$. Recall that there are notions of semistability
and vanishing of rational Chern classes for principal $G$-bundles on the projective
varieties $X_k$. The combination of these two conditions is independent of the choice
of polarization, and functorial for pullbacks. It will be called ``semiharmonic type''
in Definition \ref{semiharmonic} below. Say that a principal $G\times S/S$-bundle with 
$\lambda$-connection on $X_{\hdot}\times S/S$, is of semiharmonic type if its fibers
over all $s\in S$ are so. 

The moduli stack $\Mm _{\rm Hod}(X_{\hdot},G)$
is the functor from schemes $S/\aaa ^1$ to
$1$-groupoids, which to $S\stackrel{\lambda}{\rightarrow}\aaa ^1$ 
associates the $1$-groupoid of 
principal $G\times S/S$-bundles with $\lambda$-connection of semiharmonic type.   Specializing
to $\lambda =0$ and $\lambda = 1$ yields the Hitchin and de Rham moduli stacks
respectively.

Equivalences between moduli stacks which are natural in the variable $X$ translate
in the simplicial setting to equivalences. This gives the Riemann-Hilbert equivalence
$$
\Mm _{DR}(X_{\hdot},G)^{\rm an} \cong \Mm _{B}(X_{\hdot},G)^{\rm an}
$$
which in turn allows us to construct the analytic moduli stack 
$$
\Mm _{DH}(X_{\hdot},G)\rightarrow \pp ^1
$$
by the Deligne-Hitchin glueing \cite{hfnac}.

For $G=GL(n)$ letting $n$ vary we obtain the categories of local systems of vector spaces,
or of bundles with $\lambda$-connection. In this case we may consider all morphisms
not necessarily isomorphisms, and the same considerations as above apply.

Say that $X_{\hdot}$ is connected if its topological realization is connected. 
If $x:Spec(\cc )\rightarrow X_0$ is a basepoint, we obtain a map of Artin stacks 
$$
x^{\ast}:\Mm _{\eta}(X_{\hdot},G)\rightarrow BG .
$$
Let $R_{\eta}(X_{\hdot},x,G)$ denote the fiber of $x^{\ast}$ over 
the standard basepoint $0\in BG$. When $X_{\hdot}$ is
connected, the correspondences between $G$-torsors and 
bundles with integrable connection, or Higgs bundles of semiharmonic type,
imply that all points  $R_{\eta}(X_{\hdot},x,G)$ have trivial stabilizers. 

As in Proposition \ref{h1calc}, 
one should choose multiple base points in order to express 
$R_{\eta}(X_{\hdot},x,G)$ in terms of the representation spaces of the
components $X_k$. Choose a nonempty simplicial set ${\bf x}_{\hdot}$, finite
at each level, with a map
${\bf x}_{\hdot}\rightarrow X_{\hdot}$. An $\eta$-local system with coefficients
in $G$, such as a flat $G$-torsor for $\eta = B$ or a principal $G$-Higgs
bundle for $\eta = H$, restricts to a simplicial family of vector spaces
over   ${\bf x}_{\hdot}$. In all cases this corresponds to a flat
$G$-torsor on the realization $|{\bf x}_{\hdot}|$. A {framing} is
a trivialization of this flat $G$-torsor. If we assume that 
$|{\bf x}_{\hdot}|$ is homotopically discrete and choose a set of points
mapping isomorphically to its $\pi _0$, then a framing is the same thing as
a framing of the collection of fibers of our torsor over the given points.

Let $R_{\eta}(X_{\hdot},{\bf x}_{\hdot},G)$ denote the moduli stack of
$\eta$-local systems on $X_{\hdot}$ with coefficients in $G$, 
framed along ${\bf x}_{\hdot}$. 

\begin{proposition}
\label{repspacecalc}
With the above notations, the stack 
$R_{\eta}(X_{\hdot},{\bf x}_{\hdot},G)$ is the $2$-limit of 
the diagram $k\mapsto R_{\eta}(X_{k},{\bf x}_{k},G)$.

If furthermore the finite sets ${\bf x}_k$ meet each component of $X_k$ for 
$k=0,1,2$, then 
$R_{\eta}(X_{\hdot},{\bf x}_{\hdot},G)$ is the equalizer of
the two maps between the pieces for $k=0,1$:
\begin{equation}
\label{cartsquare}
R_{\eta}(X_{\hdot},{\bf x}_{\hdot},G) \rightarrow  R_{\eta}(X_{0},{\bf x}_{0},G)
\twoarrows R_{\eta}(X_{1},{\bf x}_{1},G).
\end{equation}
In particular, it is a quasiprojective scheme. 
\end{proposition}
\begin{proof}
The first part is formal. The $2$-limit only depends on the first three terms.
If the simplicial basepoint meets all components of the $X_k$ then 
the terms $R_{\eta}(X_{k},{\bf x}_{k},G)$ are quasiprojective
schemes \cite{Moduli}, so the $2$-limit is just the equalizer.
\end{proof} 

As discussed in Proposition 
\ref{h1calc} and Example \ref{coordinateplanes} for local systems
(that is $\eta = B$),
the condition that the basepoint 
meets the components of $X_2$ is necessary for this to be true even though it doesn't 
then enter into
the formula. 

Suppose the set of basepoints is smaller, for example a single $x$.
In the Betti case 
$R_{\eta}(X_{\hdot},x,G)$ is a quasiprojective scheme. 
In fact it is just the usual affine scheme 
of representations of $\pi _1(|X_{\hdot}|,x)$. By the Riemann-Hilbert
correspondence, separability follows for the de Rham case $\eta =DR$.
For Higgs bundles,
a geometrical argument seems to be needed and will be formulated in the
next theorem.

\begin{lemma}
\label{isomqp}
Suppose $Y$ is a smooth projective variety and $P,Q$ are principal
$G$-bundles with $\lambda$-connection on $Y\times S/S$ for a quasiprojective
base scheme $S$.
Then the functor which to $S'\rightarrow S$ associates the set of isomorphisms
between $P|_{Y\times S}$ and $Q|_{Y\times S}$ is represented by a 
quasiprojective $S$-scheme ${\rm Iso}_{Y\times S/S}(P,Q)$ affine over $S$.
\end{lemma}
\begin{proof}
Embedd $G\subset GL(n)$. Morphisms between linear bundles with $\lambda$-connection
are representable by vector schemes. Isomorphisms are then parametrized by pairs of
morphisms going both ways whose composition is the identity, and the condition that 
the isomorphism respect the reduction of structure group to $G$ is a closed condition.
\end{proof}

\begin{theorem}
\label{repqp}
Suppose $X_{\hdot}$ is a connected simplicial smooth projective variety,
with a nonempty
simplicial basepoint ${\bf x}_{\hdot}\rightarrow X_{\hdot}$.
Then $R_{\eta}(X_{\hdot},{\bf x}_{\hdot},G)$ is a 
quasiprojective (in particular separated) scheme.

If $z\in X_k$, let ${\bf x}'_{\hdot} ={\bf x}_{\hdot}\sqcup \langle z \rangle$.
Then $G$ acts freely on $R_{\eta}(X_{\hdot},{\bf x}'_{\hdot},G)$ by change of
framing at $z$, and the quotient is $R_{\eta}(X_{\hdot},{\bf x}_{\hdot},G)$.
\end{theorem}
\begin{proof}
Notice that the second statement follows from the first, because 
$R_{\eta}(X_{\hdot},{\bf x}'_{\hdot},G)$ is the $G$-bundle of frames
of the universal bundle over 
$R_{\eta}(X_{\hdot},{\bf x}_{\hdot},G)$.

No argument is needed for $\eta = B$, so we will
be treating $G$-principal $\lambda$-connections. 
The condition of semiharmonic type,
i.e. semistability and vanishing of Chern classes, is
assumed everywhere. 

For any simplicial basepoint
${\bf x}_{\hdot}$, even if it doesn't meet all components of the $X_k$,
let $\Rr_{\eta} (X_{\leq 1}, {\bf x}_{\leq 1},G)$ denote the moduli stack of 
$\eta$-local systems on the $1$-skeleton of $X_{\hdot}$, trivialized over 
the $1$-skeleton of ${\bf x}_{\hdot}$. In general it might be a stack, as occurs
when ${\bf x}_{\hdot}$ is empty for example. 

It parametrizes degeneracy-compatible descent data on $X_1\twoarrows X_0$,
that is to say pairs $(P,\phi )$ where $P$ is a $G$-bundle with $\lambda$-connection
of semiharmonic type on $X_0$, 
and $\phi : \partial _0^{\ast}(P)\cong \partial _1^{\ast}(P)$
such that $s_0^{\ast}(\phi )=1$ where $s_0: X_0\rightarrow X_1$ is 
the degeneracy. This condition needs to be included here or else
one would have to talk 
about $(X_{\leq 1})_2$ which is nonempty but has only degenerate pieces. 

If ${\bf x}_{\hdot}$ meets all connected components of the $X_k$ then 
compatibility with the degeneracy is automatic, and
$\Rr_{\eta} (X_{\leq 1}, {\bf x}_{\leq 1},G)$ is the
equalizer \eqref{cartsquare}
occuring 
in Proposition \ref{repspacecalc}. In this case it is a quasiprojective scheme rather
than a stack. 

For components of the simplicial basepoint of the form $\langle y\rangle \cong h([2])$
for $y\in X_2$, the $1$-skeleton is not contractible: rather it is 
the boundary triangle of the $2$-simplex. This case is what leads to new equations
for the representation varieties when we add in points to the simplicial basepoint,
so it is worth looking at more closely. Let $T_{\hdot}:= h([2])$ denote
the contractible simplicial $2$-simplex. Its boundary or $1$-skeleton $T_{\leq 1}$
is a triangle.
Let $t_0\in T_0$ denote the $0$-th vertex. 

The inclusion of the boundary triangle into the $2$-simplex induces the 
map
\begin{equation}
\label{trimap}
\ast = \Rr_{\eta }(T_{\hdot}, (t_0)_{\leq 1},G)\rightarrow 
\Rr_{\eta}(T_{\leq 1}, (t_0)_{\leq 1}, G)=G
\end{equation}
which is inclusion of the identity element as a point in $G$. 
This may be seen directly from the configuration of three points in $T_0$
corresponding to vertices of the triangle
and six points of $T_1$, three degenerate ones located at the vertices
and three corresponding to the nondegenerate edges. 
A principal $G$-bundle over this configuration together with its face and degeneracy maps
corresponds to a flat $G$-bundle on the boundary of the triangle.
When the trivialization at $t_0$ is included, it corresponds to a monodromy element in 
$G$, and it extends to all of $T$ if and only if the monodromy element is trivial.

Continue with the proof of the theorem. 
Suppose given a simplicial basepoint 
of the form ${\bf x}=\langle x ^1\rangle \sqcup \cdots \sqcup \langle x^r\rangle$
with $x^1\in X_0$ and $x^i\in X_{k^i}$ for $k^i\in \{ 0,1,2 \}$. 
Let $G'({\bf x}_{\hdot}):= \prod _{j=2}^r G$ with the $j$-th term acting
by change of framing over the $0$-th vertex of 
$x^j$. Thus $G'({\bf x}_{\hdot})$ acts
on $\Rr_{\eta}(X_{\leq 1}, {\bf x}_{\leq 1},G)$. 

If ${\bf x}_{\hdot}$ meets all connected components of $X_0$, $X_1$ and $X_2$
then Proposition \ref{repspacecalc} tells us that  
$R(X_{\hdot}, {\bf x}_{\hdot},G)=\Rr_{\eta}(X_{\leq 1}, {\bf x}_{\leq 1},G)$
hence
$$
R(X_{\hdot}, x^1,G) = 
\Rr_{\eta}(X_{\leq 1}, {\bf x}_{\leq 1},G)\stackquot G'({\bf x}_{\hdot}).
$$
The goal is to show that this is quasiprojective.

Start with a simpler choice of simplicial basepoint.
Order the connected components of $X_0$ as $X_0^1,\ldots , X_0^a$,
such that for any $2\leq i \leq a$ there exists a connected component $X_1^i$
of $X_1$ with $\partial _0(X^i_1)\subset X_0^{j}$ for $j<i$ and 
$\partial _1(X^i_1)\subset X_0^{i}$.
Choose $y^1\in X_0^1$, and for $2\leq i\leq a$ choose $y^i\in X_1^i$. Then set  
$$
{\bf y}_{\hdot}= \langle y^1\rangle \sqcup \cdots \sqcup \langle y^a\rangle .
$$
Consider first the equalizer 
$$
\Rr '\rightarrow R(X_0,{\bf y}_0,G) \twoarrows R(\coprod _{i=2}^a X^i_1, {\bf y}_1,G).
$$
The quotient $\Rr '\stackquot G({\bf y}_{\hdot})$ is a moduli stack
parametrizing $a$-tuples of $G$-principal $\lambda$-connections $P^i$
on the $X^i_0$, 
with a choice of framing for $P^1$ over $x^1$, together with choices of isomorphisms
between the restrictions $\partial _0^{\ast}(P^{i-1})$ and $\partial_1^{\ast}
(P^{i-1})$ over $X_1^i$ for $i=2,\ldots , a$. 

Let $V_k$ denote the moduli stack of $k$-uples $(P^1,\ldots , P^k)$ with isomorphisms
as above. We prove by induction on $k$ that it is a quasiprojective scheme,
starting with $k=1$ which is the case of principal $\lambda$-connections
over the smooth projective variety $X_0^1$ \cite{Moduli}.

There is a universal object over $U_k\times X^1_0\times \cdots \times X^a_0$, 
in particular its restriction to
the next basepoint $\partial _0(y^{k+1})$ is a principal $G$-bundle over $U_k$.
The representation variety $R(X_0^{k+1},\partial _1(y^{k+1}),G)$ has a $G$-action,
so we can twist it to get a fibration $V_{k+1}\rightarrow U_k$ with fiber $R(X_0^{k+1},\partial _1(y^{k+1}),G)$. Similarly, twisting gives a fibration $W_{k+1}\rightarrow U_k$
with fiber $R(X_1^{k+1},y^{k+1},G)$. Restriction of the universal bundle is a section 
$U_k\rightarrow W_{k+1}$, and restriction from $X_0^{k+1}$ is 
a morphism $V_k\rightarrow W_{k+1}$. Specifying a $k+1$-tuple $(P^1,\ldots , P^{k+1})$
is equivalent to specifying a point in $V_{k+1}$ whose restriction is the same as
that of $P^k$. In other words, 
the next moduli space is the fiber product 
$$
U_{k+1}= V_{k+1}\times _{W_{k+1}}U_k.
$$
Hence $U_{k+1}$ is a quasiprojective scheme. 
This completes the inductive step. At $k=a$, this shows that
$$
U_a=\Rr '\stackquot G({\bf y}_{\hdot})
$$ 
is quasiprojective.

Now $\Rr_{\eta}(X_{\leq 1}, {\bf y}_{\leq 1},G)\stackquot G'({\bf y}_{\hdot})$ 
is affine over $\Rr '\stackquot G({\bf y}_{\hdot})$ parametrizing isomorphisms
between the restrictions $\partial _0^{\ast}$ and $\partial _1^{\ast}$ of
the universal object, to the other components of $X_1$. 
Lemma \ref{isomqp} applied to the union of other components, gives that 
$\Rr_{\eta}(X_{\leq 1}, {\bf y}_{\leq 1},G)\stackquot G'({\bf y}_{\hdot})$ is
quasiprojective. 

To finish, proceed by induction starting from ${\bf y}$
and successively adding points until we get to
a simplicial basepoint meeting all the required components.
It suffices analyze what happens when we pass from ${\bf x}_{\hdot}$ 
to ${\bf x}_{\hdot}\sqcup \langle z \rangle$ for $z\in X_k$, $k=0,1,2$.
For in $X_0$ or $X_1$, the moduli problem solved by
$\Rr_{\eta}(X_{\leq 1}, {\bf y}_{\leq 1}\sqcup \langle z \rangle_{\leq 1},G)
\stackquot G'({\bf y}_{\hdot}
\sqcup \langle z \rangle )$ 
is the same as that solved by 
$\Rr_{\eta}(X_{\leq 1}, {\bf y}_{\leq 1},G)\stackquot G'({\bf y}_{\hdot})$,
plus a choice of framing over $z$, but also modulo the action of an extra copy of $G$
on this choice of framing. Therefore 
$$
\Rr_{\eta}(X_{\leq 1}, {\bf y}_{\leq 1}\sqcup \langle z \rangle _{\leq 1},
G)\stackquot G'({\bf y}_{\hdot}
\sqcup \langle z \rangle )
\cong 
\Rr_{\eta}(X_{\leq 1}, {\bf y}_{\leq 1},G)\stackquot G'({\bf y}_{\hdot})
$$
and quasiprojectivity for ${\bf y}_{\hdot}$ implies quasiprojecxtivity 
for ${\bf y}_{\hdot}\sqcup \langle z \rangle$. 

Consider therefore the case $z\in X_2$. Then 
$\langle z \rangle$ is the $2$-simplex $T$ considered above. 
A descent datum on $X_{\leq 1}$ restricts to one over $T_{\leq 1}$,
hence to a monodromy element as discussed after equation \eqref{trimap} above. 
The moduli problem solved by 
$\Rr_{\eta}(X_{\leq 1}, {\bf y}_{\leq 1}\sqcup \langle z \rangle _{\leq 1},G)
\stackquot G'({\bf y}_{\hdot}
\sqcup \langle z \rangle )$ is the moduli problem for
$U:=\Rr_{\eta}(X_{\leq 1}, {\bf y}_{\leq 1},G)\stackquot G'({\bf y}_{\hdot})$, plus
a trivialization of this $G$-bundle on the boundary of the triangle
$T_{\leq 1}$, modulo choice of
framing at one point. 

Over $U\times X_{\leq 1}$ there is a universal object which restricts to
a $G$-bundle on $U\times T_{\leq 1}/U$. The condition that
the monodromy be trivial is a closed condition over $U$. To prove this it suffices
to do it etale-locally, but then we can assume that there is a trivialization of the
restriction to one vertex; the monodromy becomes a function to $G$ 
such that the inverse image of $\{ 1_G\}$ (see \eqref{trimap})
is the required closed subset. From all
of this we conclude that 
$$
\Rr_{\eta}(X_{\leq 1}, {\bf y}_{\leq 1}\sqcup \langle z \rangle _{\leq 1},G)
\stackquot G'({\bf y}_{\hdot}
\sqcup \langle z \rangle )
\subset
\Rr_{\eta}(X_{\leq 1}, {\bf y}_{\leq 1},G)\stackquot G'({\bf y}_{\hdot})
$$ 
is a closed subscheme. Again, quasiprojectivity on the right implies it on the left.
This completes the induction step. 

By induction, we can go to
the case of a simplicial basepoint ${\bf x}_{\hdot}$ meeting all components of $X_k$
for $k=0,1,2$. We have shown that 
$\Rr_{\eta}(X_{\leq 1}, {\bf x}_{\leq 1},G)\stackquot G'({\bf x}_{\hdot})$ is
a quasiprojective scheme. However, it is equal to $R$ 
in this case by
Proposition \ref{repspacecalc}, which completes the proof that
$$
R(X_{\hdot}, x^1,G) = 
R(X_{\hdot}, {\bf x}_{\hdot},G)\stackquot G'({\bf x}_{\hdot})
$$
is quasiprojective. This finishes the proof of the theorem in case of a single
basepoint. For any nonempty simplicial basepoint, going back in the other direction
corresponds to looking at frame bundles over 
$R(X_{\hdot}, x^1,G)$, which are quasiprojective too. 
\end{proof}

This proof shows how the components of $X_2$ lead to additional equations
for the representation scheme, via the monodromy elements over triangles $T_{\leq 1}$.
In case of a simplicial scheme $X_{\hdot}$ such that each $X_k$ is simply connected,
the fundamental group is the same as that of the simplicial set $k\mapsto \pi _0(X_k)$
and the above procedure shows how the elements of $\pi _0(X_2)$ act
as relations. 

\begin{corollary}
Suppose $X_{\hdot}$ is connected with each $X_k$ being a smooth
projective variety. Choose a basepoint $x\in X_0$.
Then $R_{\eta}(X_{\hdot},x,G)$ is a quasiprojective
scheme for $\eta = B,DR,H,Hod$, an analytic space for $\eta = DH$. The group $G$ acts
on it and the quotient stack is $\Mm _{\eta}(X_{\hdot},G)$.
For the cases $\eta = H,Hod,DH$ there is an action of $\Gm$ on both the representation
scheme $R$ and the quotient stack $\Mm$. 
\end{corollary}
\begin{proof}
Apply the previous theorem with the nonempty basepoint $\langle x \rangle$.
The group actions are obtained from the universal property. 
\end{proof}

For an explicit description of the 
representation scheme, let $y^1=x\in X_0$ be the first basepoint. 
Choose $y^j\in {\bf x}_{m(j)}$ for
$j=2,\ldots , r$ such that the collection meets all components of $X_0$, $X_1$ and
$X_2$. Let ${\bf y}_{\hdot}=\coprod _{j=1}^r\langle y^j\rangle$ 
be the corresponding simplicial basepoint. 
Choose representatives $x^j\in \langle y^j\rangle _0$.

\begin{corollary}
\label{repcalccor}
With these notations, Proposition \ref{repspacecalc} allows us to calculate 
$R_{\eta}(X_{\hdot},{\bf y}_{\hdot},G)$. Then, applying Theorem \ref{repqp} recursively,
the group $\prod _{j=2}^rG$ acts freely on $R_{\eta}(X_{\hdot},{\bf y}_{\hdot},G)$
by change of framings at the points $x^j$,
and the quotient is $R_{\eta}(X_{\hdot},x,G)$. 
It extends to an action of the group $\prod _{j=1}^rG$ with 
$$
R_{\eta}(X_{\hdot},{\bf y}_{\hdot},G)\stackquot (G\times \prod _{j=1}^rG)\cong 
R_{\eta}(X_{\hdot},x,G)\stackquot G \cong \Mm  _{\eta}(X_{\hdot}, G).
$$
\end{corollary}

Choosing only basepoints in $X_0$ gives 
a slightly different description refering directly to Lemma \ref{isomqp}.

\begin{corollary}
\label{gooddescrip}
Choose basepoints $y^j\in X_0$ meeting all the connected components of $X_0$,
and let ${\bf y}_{\hdot}=\coprod _{j=1}^r\langle y^j\rangle$.
Then 
$$
\Rr _{\eta}(X_{\leq 1},{\bf y}_{\leq 1},G)\rightarrow 
R_{\eta}(X_0, \{ y^j\} , G)
$$
is an affine map parametrizing $G$-principal $\lambda$-connections $P$ on
$X_0$, framed at the $y^j$, together with isomorphisms $\phi :\partial _0^{\ast}(P)\cong
\partial _1^{\ast}(P)$ on $X_1$ compatible with the degeneracies (or equivalently, with the
framings on $X_1$). Furthermore 
$$
R _{\eta}(X_{\hdot},{\bf y}_{\hdot},G)
\subset \Rr _{\eta}(X_{\leq 1},{\bf y}_{\leq 1},G)
$$
is a closed subvariety parametrizing the $(P,\phi )$ such that $\phi$ satisfies the
cocycle condition on $X_2$.
\end{corollary}

Turn now to the study of the universal categorical quotients
of the moduli stacks, going back to \cite{git} and 
more particularly \cite{LubotskyMagid} for the moduli of representations. 
Consider first the abstract setup of an algebraic stack $\Mm$,
similarly to the work of Iwanari \cite{Iwanari}. A morphism $\Mm \rightarrow M$
is a {universal categorical quotient in the category of schemes} if $M$ is a scheme, and if
for any schemes $Y$ and $Z$ with a map $Z\rightarrow M$, a map
$$
\Mm \times _M Z \rightarrow Y
$$
factors through a unique map $Z\rightarrow Y$. A universal categorical quotient is obviously unique. 

In our situation, the moduli stack is a quotient stack $\Mm = R\stackquot G$ with
$R$ a quasiprojective scheme. In this case, Seshadri defines the notion of 
good quotient \cite{Seshadri1} and notes that Mumford's construction
of the quotient for the set of semistable points \cite{git} is good.  
A good quotient is separated and quasiprojective, and the points correspond to closed
orbits of the $G$-action. 

\begin{lemma}
\label{affineover}
Suppose $V$ is a quasiprojective scheme with $G$ action, 
such that all points are semistable with respect to a linearized line bundle $L$.
Suppose $\varphi :F\rightarrow G$ is a $G$-equivariant affine map.
Then all points of $F$ are semistable for $\varphi ^{\ast}(L)$, 
so there is a good quotient $F/G$ too. 
\end{lemma}
\begin{proof}
If $z\in F$ then $\varphi (z)$ is semistable by hypothesis.
This means that there is a section $f\in H^0(V,L^{\otimes n})^G$ such that $V_{f\neq 0}$
is an affine neighborhood of $\varphi (z)$. It pulls back to a $G$-invariant section on 
$F$ and $F_{\varphi ^{\ast}(f)\neq 0}=\varphi ^{-1}(V_{f\neq 0})$ is affine. Thus
$z$ is semistable. 
\end{proof}

\begin{theorem}
Suppose $Z_{\hdot}$ is a simplicial scheme with split degeneracies such that the
$Z_k$ are smooth projective varieties. Suppose $Z_{\hdot}$ is
connected with a basepoint $z$. Then for $\eta = B,DR,H,Hod$ there is a linearized
line bundle such that all points
of $R_{\eta}(Z_{\hdot},z,G)$ are semistable. Therefore $\Mm _{\eta}(Z_{\hdot},G)$
admits a universal categorical quotient which is a good quotient 
$$
M_{\eta}(Z_{\hdot},G)=R_{\eta}(Z_{\hdot},z,G)/G .
$$
\end{theorem}
\begin{proof}
Choose points $z=y^1,\ldots , y^b$ in all the connected components of $Z_0$
and let ${\bf y}_{\hdot}$ be the corresponding simplicial basepoint of $Z_{\hdot}$.  
Then the action of $G^b$ on $R_{\eta}(Z_0,\{ y^j\} , G)$ by change of framings,
linearizes a line bundle $L_0$ for which all points are semistable. By
Corollary \ref{gooddescrip} the map 
$$
R_{\eta}(Z_{\hdot}, {\bf y}_{\hdot},G)\rightarrow 
R_{\eta}(Z_0,\{ y^j\} , G)
$$
is a $G^b$-equivariant affine map. Therefore $L_0$ pulls back to a $G^b$-linearized line
bundle $\widetilde{L}$ on $R_{\eta}(Z_{\hdot}, {\bf y}_{\hdot},G)$
for which all points are semi\-stable and there exists a good quotient, as pointed out in 
Lemma \ref{affineover}. The quotient map factors through a good quotient by $G^{b-1}$ first,
then the quotient by $G$:
\begin{equation}
\label{quotsequence}
R_{\eta}(Z_{\hdot}, {\bf y}_{\hdot},G)\rightarrow 
R_{\eta}(Z_{\hdot}, y^1,G) \rightarrow 
M_{\eta}(Z_{\hdot}, G).
\end{equation}
In the middle is a quasiprojective scheme representing the corresponding functor,
by Theorem \ref{repqp}.
The line bundle
$\widetilde{L}$ descends to a $G$-linearized bundle $L$ on $R_{\eta}(Z_{\hdot},y^1,G)$, which
is the pullback of a bundle on the good $G^b$-quotient $M=M_{\eta}(Z_{\hdot}, G)$. 
All of the maps in \eqref{quotsequence}
are affine maps since the middle variety is covered by the affine $G^b$-quotients of the
inverse images of affine sets defined by sections of the line bundle on $M$. 
It follows that all points of $R_{\eta}(Z_{\hdot},y^1,G)$ are semistable, and
$M$ is a good quotient of $R_{\eta}(Z_{\hdot}, ^1,G)$ by the action of $G$.
\end{proof}

\begin{corollary}
The universal categorical quotients
$\Mm _{\rm Hod}(X_{\hdot},G)\rightarrow M_{\rm Hod}(X_{\hdot},G)$ 
glue together to give a separated analytic universal categorical quotient
$$
\Mm _{\rm DH}(X_{\hdot}, G)\rightarrow M_{\rm DH}(X_{\hdot},G)
$$
which is the Deligne-Hitchin twistor space for representations of $\pi _1(X_{\hdot})$ in $G$. 
\end{corollary}

One can interpret these things as a kind of {\em weight filtration}. 
Suppose $X_{\hdot}$ is a simplicial scheme such that each 
$X_k$ is a smooth projective variety. The morphisms
$$
\Mm _{\eta}(X_{\hdot}, G)\rightarrow \Mm _{\eta}(X_0,G)
$$
$$
M _{\eta}(X_{\hdot}, G)\rightarrow M _{\eta}(X_0,G)
$$
and, for $x\in X_0$ a basepoint,
$$
R _{\eta}(X_{\hdot},x, G)\rightarrow R _{\eta}(X_0,x,G),
$$
induce equivalence relations on the left hand sides. Define the weight 
filtration to be the equivalence relation\footnote{It might be interesting 
to use the derived fiber product here instead,
but that would go beyond our present scope.}
$$
W\Mm _{\eta}(X_{\hdot}, G):= \Mm _{\eta}(X_{\hdot}, G)\times _{\Mm _{\eta}(X_{0}, G)}
\Mm _{\eta}(X_{\hdot}, G),
$$
and similarly for $WM_{\eta}$ and $WR_{\eta}$. 
The $WM_{\eta}$ and $WR_{\eta}$ are equivalence relations on $M_{\eta}$ and $R_{\eta}$
respectively. Because of the stackiness, $W\Mm _{\eta}$ will in general have a structure of
groupoid in the category of stacks. However, the arguments given above show that the map
$\Mm _{\eta}(X_{\hdot}, G)\rightarrow \Mm _{\eta}(X_0,G)$ is representable and affine.

\section{Hodge and harmonic theory}

Classical results and techniques from Hodge theory apply also to Deligne-Mumford
stacks, see \cite{ToenThesis} \cite{MatsukiOlsson} for example, and more
generally to simplicial manifolds as in \cite{Dupont} \cite{Jeffrey}. 
Similarly, nonabelian harmonic theory for local systems applies 
to a simplicial smooth projective variety,
with a few modifications, by working on each level. Many proofs in this section will
be shortened or left to the reader. 

Suppose $X_{\hdot}$ is a simplicial
smooth projective variety. A simplicial Higgs bundle $(E_{\hdot}, \theta )$
is a collection of Higgs
bundles $E_k$ of rank $n$ on $X_k$, together with compatibility isomorphisms 
for each $k\rightarrow m$ in $\Delta$ in the same way as for local systems. 

If $G$ is a linear algebraic group, a principal $G$-Higgs
bundles $(P_{\hdot}, \theta )$ on $X_{\hdot}$ is a simplicial family
of principal $G$-Higgs bundles on the $X_k$. The preceding definition is recovered
for $G=GL(n)$. 
Make the corresponding definitions for $\lambda$-connections over $\lambda \in \aaa ^1$. 

\begin{definition}
\label{semiharmonic}
A principal $G$-Higgs bundle $(P,\theta )$ on a smooth projective variety $X$ is of  
semiharmonic type if it is semistable with
Chern clases vanishing in rational cohomology. 
This condition is independant of the
choice of K\"ahler class.

A simplicial principal $G$-Higgs bundle $(P_{\hdot},\theta )$ over a simplicial smooth projective variety $X_{\hdot}$ is said to be of semiharmonic type if each $(P_k,\theta )$ is of semiharmonic type on $X_{k}$. 
\end{definition}

Given a principal $G$-bundle with $\lambda$-connection $(P_{\hdot},\nabla )$, say
that it is of semiharmonic type if it satisfies the previous definition when $\lambda =0$;
the condition is automatically true for $\lambda \neq 0$ since we consider only
the compact case here. 

Applying the equivalence of categories from \cite{hbls} level by level
gives a simplicial version. 

\begin{proposition}
\label{correspondence}
Suppose $X_{\hdot}$ is a connected simplicial smooth projective variety. 
There is an equivalence of tannakian categories between the category of  
simplicial Higgs bundles of semiharmonic type on $X_{\hdot}$ and the category
of local systems. This equivalence is compatible with pullback along morphisms of simplicial 
varieties, in particular it preserves the fiber functors of restriction to a basepoint. 
For any linear group $G$ this induces an equivalence between the categories of
simplicial  principal $G$-Higgs bundles of  semiharmonic type on $X_{\hdot}$ and
the category of $G$-torsors over $|X_{\hdot}|$. 
\end{proposition}

Say $G$ is reductive. 
A principal Higgs bundle of semiharmonic type will be called
polarizable if each $(P_k,\theta )$
admits a harmonic reduction of structure group to the maximal compact of $G$, or equivalently if it is polystable. Say that $(P_{\hdot}, \theta )$ is {\em strongly polarizable} if there is a simplicial family of harmonic
reductions of structure group $h_k$ compatible under the transition maps for $k\rightarrow m$ in $\Delta$. 
Similarly, say that a local system is
{strongly polarizable} if there exists a compatible
collection of harmonic metrics $h_k$ on $L_k$.
We use this terminology interchangeably for the corresponding local system $L$
on $|X_{\hdot}|$. 

The equivalence of categories of Proposition \ref{correspondence} preserves the
conditions of polarizability and strong polarizability. For polarizable objects
the equivalence can be expressed in terms of harmonic bundles on 
$X_{\hdot}$, in other words simplicial families 
denoted $\Ee _{\hdot}$ of harmonic bundles $(\Ee _{k}, \partial , \overline{\partial}, \theta ,\overline{\theta})$ on $X_k$, together with pullback isomorphisms compatible with cohomology. 

The category of harmonic bundles on a simplicial scheme $X_{\hdot}$ maps by
an equivalence of category to the subcategory of polarizable local systems $L_{\hdot}$ 
on $X_{\hdot}$, i.e. ones such that each $L_k$ is semisimple on $X_k$. 
It also maps by an equivalence of categories to the category of termwise
polystable $\lambda$-connections, for any $\lambda \in \aaa ^1$. 
Among other things, 
these functors with the same formulae as in the usual smooth projective case,
provide us with a collection of {prefered sections} of the family of analytic moduli stacks
$\Mm _{\rm DH}(X_{\hdot},G)\rightarrow \pp ^1$, and their images  which are
sections of the family of moduli spaces $M _{\rm DH}(X_{\hdot},G)\rightarrow \pp ^1$.

\begin{proposition}
The category of strongly polarizable local systems is tannakian and semisimple.
Restriction to any basepoint $x\in |X_{\hdot}|$
provides a fiber functor, and the corresponding
affine algebraic group $\varpi ^{\spre}_1(X_{\hdot},x)$ is reductive. 
The monodromy representation of a strongly polarizable local system is
semisimple. 
\end{proposition}
\begin{proof}
Suppose $(L_{\hdot}, h_{\hdot})$ is a strongly polarized local system
on $X_{\hdot}$. 
If $U_{\hdot}\subset L_{\hdot}$ is a sub-local system then 
the simplicial family $V:k\mapsto U_k^{\perp}$ of orthogonal complements
with respect to the $h_k$ forms a complement, $L_{\hdot}=U_{\hdot}\oplus V_{\hdot}$.
\end{proof}

The following example shows that semisimplicity doesn't necessarily
hold for local systems which are only polarizable; and furthermore that
semisimple local systems are not necessarily strongly polarizable. 

\begin{example}
Suppose $X_{\hdot}$ is a simplicial smooth projective variety such that
each $X_k$ is simply connected. Then every local system $L$ on 
$|X_{\hdot}|$ is
polarizable, but a local system is strongly polarizable if and only if it is unitary.
\end{example}

In fact, semisimplicity doesn't necessarily imply polarizability, either. 

\begin{example}
Let $X_{\hdot}$ be the simplicial resolution of a nodal curve, with $X_0$ the
normalization of genus $g>1$. Then there are local systems which are not
semisimple on $X_0$, but where the additional monodromy transformation at the
node makes the full monodromy representation semisimple.
\end{example}

On the other hand, for hypercoverings of normal DM-stacks,
$$
\pi _1(X_0,x)\rightarrow \pi _1(|X_{\hdot}|,x)
$$
has image of finite index and
polarizability, semisimplicity and
strong polarizability are the same---see Lemma \ref{fiex} and 
Theorem \ref{unibranch} below.

If $X_{\hdot}$ is a simplicial smooth scheme, a variation of Hodge structure
$V_{\hdot}$ over $X_{\hdot}$ consists of specifying a variation of Hodge structure 
$V_{k}$ on each $X_k$, together with functoriality isomorphisms 
$\phi ^{\ast}(V_k)\cong V_m$ whenever $\phi :[k]\rightarrow [m]$ is a map in 
$\Delta$, satisfying the usual compatibility condition. 
We say that $V_{\hdot}$ is {polarizable} if each $V_k$ is polarizable. We say that
$V_{\hdot}$ is {strongly polarizable} if there exist polarizations $h_k$ on 
each $V_k$ which are compatible with the functoriality isomorphisms.

\begin{lemma}
\label{cstaraction}
Suppose $X_{\hdot}$ is a simplicial smooth projective variety. Suppose $L_{\hdot}$ is a 
polarizable 
local system on $X_{\hdot}$ corresponding to the Higgs bundle $(E_{\hdot},\theta )$. 
Then a structure of polarizable VHS on $L_{\hdot}$ 
is exactly given by a trivialization of the $\cc ^{\ast}$ action $\varphi _t:(E_{\hdot},\theta )\cong (E_{\hdot},t\theta )$. 
A strongly polarizable local system which is a fixed point, corresponds to a strongly 
polarizable variation of Hodge structure.
\end{lemma}

\begin{remark}
\label{spvhs}
Suppose $V_{\hdot}$ is a strongly polarizable VHS. Then the monodromy group of the
underlying representation of $\pi _1(|X_{\hdot}|)$ is contained in some $U(p,q)$. 
However, this is not necessarily the case for a polarizable VHS which is not strongly
polarizable. 
\end{remark}

\begin{conjecture}
\label{vhsht}
For a strongly polarizable variation of Hodge structure,
the real Zariski closure of the image of $\pi _1(|X_{\hdot}|,x)$ in $GL(L(x))$ 
is a group of Hodge type. 
\end{conjecture}

If $X_{\hdot}$ is a simplicial variety whose components are simply connected, then
any local system is trivial on each $X_k$, in particular setting $V^{0,0}_k:=L_k$
gives a polarizable VHS, which will not however usually be strongly polarizable. 
So in general the existence of a polarizable VHS doesn't lead to restrictions on
the representation or the fundamental group. 
For that, one requires the 
finite index condition \ref{finiteindex}, 
as will be discussed in the next section on normal DM-stacks.
That condition  implies Conjecture \ref{vhsht}.

The following lemma shows that lack of strong polarizability is an
obstruction to extending a local system to a smooth ambient variety. 

\begin{lemma}
Suppose $X_{\hdot}\rightarrow Z$ is a morphism from a simplicial smooth projective variety,
to a smooth quasiprojective variety $Z$; for example when $X_{\hdot}$ is a proper surjective
hypercovering of a closed subscheme of $Z$. If $L_{\hdot}$ is a semisimple local system
on $X_{\hdot}$ which is the pullback of a local system on $Z$, then it is strongly
polarizable.
\end{lemma}
\begin{proof}
If $L_{\hdot}$ is the pullback of a local system $L_Z$, then it is also the
pullback of the associated-graded of the Jordan-H\"older series for $L_Z$,
so we may assume $L_Z$ semisimple; it then has a harmonic metric 
\cite{TMochizuki3} which restricts
to a strong polarization of $L_{\hdot}$.
\end{proof}

There are other obstructions of a similar nature.
Suppose $X_{\hdot}$ is a simplicial smooth projective variety, and $x\in X_1$ is a point.
It  has two images $\partial _0x,\partial _1x\in X_0$. If $(E_{\hdot},\theta )$
is a Higgs bundle on $X_{\hdot}$ then 
$$
E_0(\partial _0x) \cong E_1(x) \cong E_0(\partial _1x) 
$$
so the Higgs field $\theta$ on $E_0$ over $X_0$ provide separately commutative actions
of the two tangent spaces $T_{\partial _0x}X_0$ and $T_{\partial _1x}X_0$
on $E_1(x)$. Say that $(E_{\hdot}, \theta )$ {satisfies the commutativity obstruction}
if these two actions commute with each other, in other words $\theta (\partial _0x)(v_0)$ and  $\theta (\partial _1x)(v_1)$ commute as endomorphisms of $E_1(x)$
whenever $v_0\in T_{\partial _0x}X_0$ and $v_1\in T_{\partial _1x}X_0$.
The following lemma identifies this as another obstruction to extending a local system to a smooth ambient variety.

\begin{lemma}
Suppose $f:X_{\hdot}\rightarrow Z$ is a morphism from a 
simplicial smooth projective variety
to a smooth quasiprojective variety $Z$, and $L_{\hdot}$ is a local system on $X_{\hdot}$
corresponding to a Higgs bundle $(E_{\hdot},\theta )$. If $L_{\hdot}$ is the pullback
of a local system on $Z$ then $(E_{\hdot},\theta )$ satisfies the commutativity
obstruction at each $x\in X_1$.
\end{lemma}
\begin{proof}
If $L_{\hdot}$ is a pullback from $Z$, then $(E_{\hdot},\theta )$ is the
pullback of a Higgs bundle $(F,\varphi )$ on $Z$ (this works even if $Z$ is only
quasiprojective by \cite{TMochizuki3}). 
For any $x\in X_1$,
the actions of both tangent spaces
$T_{\partial _0x}X_0$ and $T_{\partial _1x}X_0$ factor through the action of $T_{f(x)}Z$
on $E_1(x)\cong F(f(x))$ given by $\varphi _{f(x)}$. 
\end{proof}

To give a concrete example, suppose $X$ is a nodal curve embedded in a smooth 
variety $Z$. A local system $L$ on $X$ restricts to a local system corresponding
to a Higgs bundle $(E_0,\theta )$ on the normalization
$X_0=\tilde{X}$. At each node $x\in X$ we obtain two endomorphisms of $L(x)$
given by the Higgs field $\theta $ applied to the tangent vectors along the two branches
going through $x$. The commutativity obstruction says that these should commute,
as will be the case if the Higgs bundle is a  pullback from $Z$. 

Look now at the local structure of the space of representations. 
If $\Gg$ is an algebraic stack and $p\in \Gg (\cc )$ is a closed point,
the {tangent space} $T_p\Gg$ is defined as the set of 
pairs $(f,e)$ where 
$$
f:{\rm Spec}\cc [\varepsilon ]/\varepsilon ^2 \rightarrow \Gg
$$
and $e:f|_{{\rm Spec}\cc}\cong p$ is an isomorphism of points, up to
natural equivalences of the $f$ respecting the isomorphisms $e$. 
If $\Gg$ is a moduli stack then the tangent space is usually known as the {deformation space}: a point consists of an infinitesimal deformation with isomorphism between the
central fiber and the original object in question.

Suppose $X_{\hdot}$ is a connected simplicial smooth projective variety, with
basepoint $x\in X_0$. The tangent space to the moduli stack $\Mm _B(X_{\hdot}, G)$
of $G$-local systems on $|X_{\hdot}|$ at $L$ is $H^1(X_{\hdot}, {\rm ad}(L))$
where ${\rm ad}(L)$ is the adjoint local system, equal to $End(L)$ in the linear
case and derived from the adjoint action of $G$ on $Lie (G)$ in general.
See \cite{Sikora} for a discussion of some fine points on tangent spaces of moduli
of local systems.  

Combining differential forms on the various simplicial levels gives a complex of
forms on a simplicial variety, as is known from \cite{Dupont} and 
\cite{Jeffrey}. This may be applied here.
If $(E_{\hdot},\theta )$ is a Higgs bundle on $X_{\hdot}$, define the {Dolbeault
cohomology} $H^i_{\rm Dol}(X_{\hdot},E_{\hdot},\theta )$ to be the cohomology 
of the total complex obtained by adding together the Dolbeault complexes $A^{\hdot}_{\rm Dol}(X_k,E_k,\theta )$ (or any equivalent functorial family of complexes computing the
same hypercohomology) on each $X_k$ and adding the alternating sum of face maps to the differential. 
More generally if $(E,\nabla )$ is a bundle with $\lambda$-connection then
we can define the {de Rham cohomology} $H^i_{\rm DR}(X_{\hdot}, E_{\hdot},\nabla )$
using the de Rham complexes on each $X_k$. The simplicial version of
Biswas and Ramanan's calculation of the deformation space \cite{BiswasRamanan} holds:

\begin{lemma}
\label{defhiggs}
Suppose $(P_{\hdot},\theta )$ is a principal $G$-Higgs bundle on $X_{\hdot}$ of
semiharmonic type. Let $({\rm ad}(P),\theta )$ denote
the linear Higgs bundle obtained from the adjoint representation.
The tangent space to the moduli stack $\Mm _H(X_{\hdot},G)$ is naturally identified as
$H^1_{\rm Dol}(X_{\hdot}, {\rm ad}(P),\theta )$. The corresponding statement holds
for the relative tangent space to the moduli stack $\Mm _{\rm Hod}(X_{\hdot},G)$
of $\lambda$-connections over any $\lambda \in \aaa ^1$.
\end{lemma}

Suppose $x\in X_0$ is a point. It may be viewed as a simplicial morphism 
$x\rightarrow X_{\hdot}$ from the constant one-point simplicial scheme to $X_{\hdot}$.
The {relative Dolbeault complex} is the cone on the map
$$
\bigoplus _{j,k} A^j_{\rm Dol}(X_k,E_k,\theta ) \rightarrow E_0(x),
$$
or equivalently the kernel of this map,
giving a complex which calculates the cohomology relative to the basepoint.  
Again the same may be said for de Rham cohomology.

\begin{remark}
\label{defrep}
The tangent space to $R_{\eta}(X_{\hdot},x,G)$
is given by the relative cohomology $H^{1}_{\eta}(X_{\hdot},x,{\rm ad}(\rho ))$
of the required type. 
\end{remark}

\begin{proposition}
\label{mhsrepvar}
Suppose $X_{\hdot}$ is a simplicial smooth projective variety, connected,
with basepoint $x\in X_0$. Suppose $V_{\hdot}$ is
a polarizable variation of Hodge structure.
Then the complete local ring $\widehat{\Oo}_{\rho ,x}$ of the formal 
completion of the representation variety 
$R_B(X_{\hdot},x, GL(n))$ at the monodromy representation $\rho$ of
$V_{\hdot}$ has a natural and functorial mixed Hodge structure generalizing
that of \cite{EyssidieuxSimpson}.
\end{proposition}
\begin{proof}
This is a sketch of proof. 
Suppose first that ${\bf x}_{\hdot}\rightarrow X_{\hdot}$ is a 
simplicial basepoint meeting all components of $X_0$, $X_1$ and $X_2$,
so Proposition \ref{repspacecalc} applies. 
The complete local ring $\widehat{\Oo}_{\rho ,{\bf x}_{\hdot}}$ of the
formal completion of $R_B(X_{\hdot},{\bf x}_{\hdot}, GL(n))$
at $\rho$ has a unique mixed Hodge structure compatible with the maps
in the cartesian square \eqref{cartsquare} and the mixed Hodge structures on
the local rings of the other three pieces given by \cite{EyssidieuxSimpson}.
This is because the maps in \eqref{cartsquare} are morphisms of mixed
Hodge structures, and the local ring of the fiber product is the 
tensor product of the local rings of the other three pieces so it inherits
an MHS. 

Write ${\bf x}_{\hdot} = \langle x\rangle \sqcup \langle y^1\rangle \sqcup \cdots
\sqcup \langle y^a\rangle$. Write $Y:= \{ y^1,\ldots , y^a\}$.
At $y\in Y$ let $V_y$ denote the Hodge structure fiber of
$V_{\hdot}$ at $y$. This determines a mixed Hodge structure on the formal completion
of $GL(V_y)$ at the identity.  For these mixed Hodge structures, the action of 
$\prod _{y\in Y}GL(V_y)$ on $R_B(X_{\hdot},{\bf x}_{\hdot}, GL(n))$
is compatible with the mixed Hodge structures. The action is free and 
$$
R_B(X_{\hdot},x, GL(n))=R_B(X_{\hdot},{\bf x}_{\hdot}, GL(n))/\prod _{y\in Y}GL(V_y).
$$
The complete local ring $\widehat{\Oo}_{\rho ,x}$ is thus the subring of
invariants in $\widehat{\Oo}_{\rho ,{\bf x}_{\hdot}}$ under the formal action
of $\prod _{y\in Y}GL(V_y)$, so there is an exact sequence
$$
0\rightarrow \widehat{\Oo}_{\rho ,x}\rightarrow 
\widehat{\Oo}_{\rho ,{\bf x}_{\hdot}} \rightarrow 
\widehat{\Oo} _{W, (\rho , 1,\ldots , 1)}
$$
where 
$$
W= R_B(X_{\hdot},{\bf x}_{\hdot}, GL(n))\times \prod _{y\in Y}GL(V_y).
$$
The map on the right is a map of MHS so $\widehat{\Oo}_{\rho ,x}$ acquires a MHS.
\end{proof}

\section{The normal case}

A normal variety is geometrically unibranched. Conversely, if 
$X$ is geometrically unibranched then its normalization
$\tilde{X}\rightarrow X$ is 
a one-to-one map and induces an homeomorphism of topological realizations.
This localizes in the etale topology so the same hold when $X$ is a DM-stack.

A more general condition in the situation of a simplicial variety
is the following ``finite index condition''.

\begin{condition}
\label{finiteindex}
For $X_{\hdot}$ a simplicial smooth projective variety, 
the present condition says that:
\newline
(1)\, for any two components $X_0^i$ and $X_0^j$ of $X_0$, there
is a component $X_1^{ij}$ of $X_1$ which dominates them by $\partial _0$ and $\partial _1$ respectively;
\newline
(2)\, for any basepoint
$x\in X^i_0$, the image of $\pi _1(X^i_0,x)\rightarrow \pi _1(|X_{\hdot}|,x)$
has finite index; and
\newline
(3)\, every polarizable local system on $X_{\hdot}$ is strongly polarizable. 
\end{condition}

Condition 
(2) is independent of the choice of basepoint, because a path in $X_0^i$
can be lifted to $X_1^{ij}$ and projected to a path in $X_0^j$. 

One could conjecture that Condition (3) is a consequence of the other two conditions,
or perhaps some other natural geometric condition. I couldn't find an argument, 
but Theorem \ref{unibranch} will say that it holds for standard resolutions of geometrically unibranched proper DM-stacks. 

\begin{lemma}
\label{fiex}
If  $X_{\hdot}$ is a simplicial smooth projective variety satisfying Condition 
\ref{finiteindex}, then the following conditions for
a local system $L_{\hdot}$ on $X_{\hdot}$ are equivalent:
\newline
$(a)$\, $L_{\hdot}$ is semisimple;
\newline
$(b)$\, there exists a component $X'_0$ of $X_0$ such that $L_0|_{X'_0}$ is semisimple;
\newline
$(c)$\, $L_{\hdot}$ is polarizable;
\newline
$(d)$\, $L_{\hdot}$ is strongly polarizable.

For a $G$-torsor, 
the monodromy is reductive if and only if the monodromy of its restriction to $X'_0$
is reductive. 
\end{lemma}
\begin{proof}
Semisimplicity of a representation is equivalent to semisimplicity of
its restriction to a finite-index subgroup.
Therefore, from part (2) of \ref{finiteindex}, $(b)$ implies $(a)$ and $(a)$ implies:
\newline
$(b)'$\, for any component $X'_0$ of $X_0$, $L_0|_{X'_0}$ is semisimple,
\newline
a condition which clearly implies $(b)$. 
Also,
$(c)$ implies $(b)'$ since polarizability and semisimplicity are the
same on a smooth quasiprojective variety. The pullback of a polarizable
local system is again polarizable, so $(b)'$ implies $(c)$. 
By part (3) of \ref{finiteindex}, $(d)$ is equivalent to $(c)$. 
\end{proof}

Condition \ref{finiteindex} holds in a
wide variety of cases. In preparation for the proof, here is a version of Zariski's
connectedness.

\begin{lemma}
\label{zariskicon}
Suppose $Z$ is a smooth variety with a projective 
map to a connected geometrically
unibranched DM-stack $X$, such that every component of $Z$ dominates $X$. Suppose $U\subset Z$ is a dense open subset. Then every connected component of $Z\times _XZ$ meets $U\times _XU$.
\end{lemma}
\begin{proof}
Suppose $X'\rightarrow X$ is an etale covering, then the same statement for $Z\times _XX'\rightarrow X'$ implies the statement for $Z\rightarrow X$. Therefore we may
assume that $X$ is a quasiprojective scheme, also it can be assumed reduced.  

Let $Z\stackrel{g}{\rightarrow} V\stackrel{h}{\rightarrow} X$ be the Stein factorization:
$h$ is finite and $g$ has connected fibers. All irreducible components of $V$
dominate $X$, so they
have the same dimension as $X$. There is an
open dense set $W\subset V$ such that $U$ meets all the fibers
$g^{-1}(w)$ for $w\in W$. 

Suppose $(z_1,z_2)\in Z\times _XZ$, with $z_1,z_2\mapsto x\in X$.
Let $v_1=g(z_1)$ and $v_2=g(z_2)$. Thus $(v_1,v_2)\in V\times _XV$. Let $N_1$ and
$N_2$ be small usual analytic
neighborhoods of $v_1$ and $v_2$ respectively in $V$. Their
images $h(N_i)\subset X$ 
are germs of closed subvarieties of the same dimension as $X$,
so by the hypothesis that $X$ is geometrically unibranched, the $h(N_i)$
must contain neighborhoods of $x$. But $W\subset V$ is a dense Zariski-open
subset, so $h(N_1\cap W)\cap h(N_2\cap W)\neq \emptyset$. By successively reducing
the size of the neighborhoods, we can choose
a sequence of points $(w_1(j),w_2(j))_{j\in \nn}\in W\times _XW$ approaching $(v_1,v_2)$ for $j\rightarrow \infty$. Lift $w_i(j)$ to points $y_i(j)\in U$. Since $g$ is proper,
a subsequence of $(y_1(j),y_2(j))\in U\times _XU$ converges to some 
$(z'_1,z'_2)\in Z\times _XZ$ lying over $(v_1,v_2)\in V\times _XV$. But the fibers of $g$
are connected (that is where Zariski's connectedness theorem is used), 
so $z'_1$ is connected to $z_1$ in $g^{-1}(v_1)$ and  
$z'_2$ is connected to $z_2$ in $g^{-1}(v_2)$. Therefore $(z'_1,z'_2)$ lies in the same
connected component of $Z\times _XZ$ as $(z_1,z_2)$. However,  $(z'_1,z'_2)$
is also a limit of points in $U\times _XU$, so the component of $(z_1,z_2)$ meets 
$U\times _XU$. 
\end{proof}

\begin{theorem}
\label{unibranch}
Suppose that $X$ is a proper singular DM-stack which is reduced, connected and 
geometrically unibranched,
that is the analytic germ of an etale chart at any singular point is irreducible.
Then, for a proper surjective hypercovering by smooth projective varieties
$Z_{\hdot}\rightarrow Y$ constructed as in Theorem \ref{fullresolution}, 
the finite index condition \ref{finiteindex} holds.
\end{theorem}
\begin{proof} 
We may assume that $X$ is reduced (the reduced substack has the same topological type).
Also it is irreducible, because connected and geometrically unibranched. 
If a hypercovering is constructed as in Theorem \ref{fullresolution}
then it satisfies condition (1). 
Suppose $V\rightarrow X$ is a surjective etale map. Elements of the fundamental
group $\pi _1(X,x)$ can be viewed as paths which are piecewise continuous on $V$,
related by the equivalence relation $V\times _XV$ at the jumping points. Furthermore
the jumping points can be assumed general, i.e. where $V$ is smooth.
Since $V$ is geometrically unibranched, paths can be moved away from the
singularities and in fact, into any dense open substack. 
There exists an open dense substack 
$U\subset X$ which is smooth and 
a gerb over its smooth coarse moduli space $U^{\bf c}$.
Then $U$ is also connected and paths can be moved into $U$, so $\pi _1(U)$
surjects onto $\pi _1(X)$. Now suppose $f:Z_{\hdot}\rightarrow X$ is a proper
surjective hypercovering, in particular
by Proposition \ref{topdescentstack}, $\pi _1(|Z_{\hdot}|)\cong\pi _1(X)$. 
Some connected component $Z'_0$ dominates $X$,
from which it follows that $\pi _1(f^{-1}(U)\cap Z'_0)\rightarrow \pi _1(U^{\bf c})$
has image of finite index. 
Thus the image is of finite index in $\pi _1(U)$, and in turn
the image of $\pi _1(Z'_0)$ in 
$\pi _1(X)=\pi _1(|Z_{\hdot}|)$ has finite index. 

For (3), given a polarizable local system $L_{\hdot}$ on $Z_{\hdot}$, 
we need to construct a strong polarization, that is a collection of harmonic
metrics $h_k$ on $L_k$ compatible with the restrictions. Eventually adding an extra
component to $Z=Z_0$, we may assume that there is a component $Z^1\subset Z$ containing
an open set $U^1\subset Z^1$ such that $U^1\rightarrow X$ is a finite etale Galois cover
over its image $U_X$ which is in the smooth locus of $X$, and in fact in the locus where
$X$ is a gerb over the smooth part of $X^{\bf c}$.
 
Let $\Phi$ be the Galois group acting on $U^1$; 
by equivariant resolution of
singularities \cite{Villamayor} \cite{BierstoneMilman}, we may assume 
that it extends to an action on $Z^1$. Therefore by averaging over $\Phi$ an
initial choice of harmonic metric on $L_{Z^1}$, we obtain a $\Phi$-invariant harmonic
metric $h'$ over $U^1$. For each component $Z^i$ of $Z$, choose a component $R^{1i}$
mapping by dominant maps $\partial _0:R^{1i}\rightarrow Z^1$ and $\partial _1:R^{1i}\rightarrow Z^i$. Then $\partial _1^{\ast}(L_{Z^i})$ is a local system 
on $R^{1i}$ isomorphic by the descent datum, to $\partial _0^{\ast}(L_{Z^1})$.
In general, given a harmonic metric on the pullback of a semisimple local system
by a dominant map of smooth projective varieties, there is a unique harmonic
metric downstairs whose pullback is the given one. So there is a harmonic metric $h_i$
on $L_{Z^i}$ with $\partial _1^{\ast}(h_i)=\partial _0^{\ast}(h_1)$ on $R^{1i}$.
Together these define a harmonic metric $h$ on $L_0$ over $Z=Z_0$. 

The descent datum $\varphi$ gives a continuous
isomorphism between the two pullbacks 
${\rm pr}_0^{\ast}(L_0)$ and ${\rm pr}_1^{\ast}(L_0)$  over
$Z\times _XZ$ as pointed out in Remark \ref{ztimesz}. Let 
$K\subset Z\times _XZ$ be the subset of points $(z_1,z_2)$ where
$\varphi _{\ast}{\rm pr}_0^{\ast}(h)={\rm pr}_1^{\ast}(h)$.
The strong polarizability condition says that $K$ should be all of $Z\times _XZ$.

The subset $K$ is closed in the usual topology, since it results from the comparison
of two continuously varying metrics. It is also open, indeed if $(z_1,z_2)\in K$ 
and $(y_1,y_2)$ is a nearby  point, then there is a connected smooth projective
algebraic curve $C\rightarrow Z\times _XZ$ passing through 
$(z_1,z_2)$ and $(y_1,y_2)$. The two pullback metrics induce harmonic metrics on 
$L|_C$ which agree over $(z_1,z_2)$, but a harmonic metric on a smooth
projective variety is determined by
its value at one point, so the two metrics agree over $(y_1,y_2)$ too. This
shows that $(y_1,y_2)\in K$, so $K$ is open. It follows that $K$ is a union of
connected components of $Z\times _XZ$. 

On the other hand,
the invariance property of $h_1$ means essentially 
that over $U_X$ it is pulled back from a harmonic metric on the local system $L$ over $U_X$, so an argument with the descent data will
show that the collection of $h^i$ are compatible with the descent data on the 
open set $U\subset Z$ which is the inverse image of $U_X$. Thus $U\times _XU\subset
K$. Lemma \ref{zariskicon} implies that $K=Z\times _XZ$ so $L$ is strongly polarizable. 
\end{proof}

If the condition ``finite index'' is replace by ``surjective'' then
there is a closed immersion of representation spaces. 

\begin{lemma}
\label{immersion}
Suppose $X_{\hdot}$ is a simplicial smooth projective variety, connected, and let 
$(X'_0,x)$ be a connected component
with basepoint in $X_0$. Suppose that $\pi _1(X'_0,x)\rightarrow \pi _1(X_{\hdot},x)$
is surjective. Then the map
$$
R_B(X_{\hdot}, x,G)\rightarrow R_B(X'_0,x,G)
$$
is a closed immersion. 
\end{lemma}
\begin{proof}
It is just given by the equations saying that the elements of the kernel
of the map on fundamental groups, have trivial image. 
\end{proof}

This lemma would apply,
for example, to geometrically unibranched DM-stacks with 
quasiprojective moduli space and trivial generic stabilizer. 
At points corresponding to variations of Hodge structures, 
the closed immersion
expresses the mixed Hodge structure on the complete local ring of 
$R_B(X_{\hdot}, x,G)$ as a quotient of that of $R_B(X'_0,x,G)$ by a mixed Hodge ideal.

Suppose $X_{\hdot}$ is a simplicial smooth projective variety, connected, satisfying
the finite index condition \ref{finiteindex}.
Suppose $P$ is a $G$-principal bundle with $\lambda$-connection,
or a $G$-torsor on $|X_{\hdot}|$. Fix a
basepoint $x\in X_0$ and a framing for $P(x)$.  One should be able
to construct, following \cite{Jeffrey}, a
Lie algebra of forms with coefficients in 
${\rm ad}(P)$ $A^{\hdot}_{\eta}(X_{\hdot}, {\rm ad}(P))$, augmented towards
$Lie (G)$, which controls the deformation theory of $P$ in the sense
of Goldman-Millson. Say that
a dgla is formal in degrees $\leq 1\frac{1}{2}$ if it is joined to a complex with
zero differential, by morphisms inducing isomorphisms on $H^0$ and $H^1$ and injections 
on $H^2$. This is enough to get a control of the structure of the representation space. 

\begin{conjecture}
In this situation, if $P$ is  
polarizable and $X_{\hdot}$ satisfies the finite index condition
\ref{finiteindex}, then the above dgla is formal in degrees $\leq 1\frac{1}{2}$. Furthermore
in this case there are natural quasiisomorphisms between the Dolbeault dgla
controling deformations of the $G$-principal Higgs bundle and the de Rham and
Betti dgla's controling deformations of the associated $G$-torsor. 
\end{conjecture}

To get around this conjecture we can prove directly one of the 
main consequences, but without making any statement about
quadraticity. 

\begin{lemma}
\label{samesings}
Suppose $X_{\hdot}$ is a simplicial scheme with smooth projective levels
satisfying the finite index condition \ref{finiteindex}. 
Suppose $\rho : \pi _1(X_{\hdot})\rightarrow G$ is a semisimple
representation,
corresponding to principal $G$-Higgs bundle $(P,\theta )$. 
The local analytic structures of $M_H(X_{\hdot},G)$ at $(P,\theta )$ and
of $M_B(X_{\hdot},G)$ or $M_{DR}(X_{\hdot}, G)$ at $\rho$ are the same.
\end{lemma}
\begin{proof}
In the case of a smooth projective variety,
the formal local structures of the representation spaces for $\eta = B$ and $\eta = H$,
at points corresponding to a semisimple representation and its corresponding Higgs
bundle, are canonically isomorphic. The isomorphism respects the
group actions by change of frames, and is 
functorial for morphisms of smooth projective varieties. 

This comes from the formality
isomorphism on the Goldman-Millson dgla's for a smooth projective variety.
From the expression of \eqref{cartsquare} we get the same canonical isomorphisms
$$
R_B(X_{\hdot},{\bf x}_{\hdot}, GL(n))^{\wedge , \rho}\cong 
R_H(X_{\hdot},{\bf x}_{\hdot}, GL(n))^{\wedge , (P,\theta )}.
$$
From Condition \ref{finiteindex}, $\rho$ is a point where the stabilizer is reductive.
Using Luna's etale slice theorem as in \cite{EyssidieuxSimpson}, 
and taking the quotient by the stabilizer, gives the required local formal
isomorphism $M_B^{\wedge , \rho}\cong M_H^{\wedge , (P,\theta )}$.
\end{proof}

Condition \ref{finiteindex} implies that the categorical equivalence between
Higgs bundles and local systems gives a homeomorphism of character varieties,
joining together two different complex structures to give a quaternionic structure
as in \cite{Hitchin}. One expects that some condition such as \ref{finiteindex}
is necessary here, because of the non-continuity of the correspondence
at non-semisimple points, see the Counterexample of \cite{Moduli} (II, p. 39). 

\begin{theorem}
\label{hyperkahler}
Suppose
$X_{\hdot}$ is a simplicial smooth projective variety, connected and which satisfies
the finite index condition \ref{finiteindex}. Suppose $G$ is a linear
reductive group. Then the 
points of the various coarse moduli spaces $M_{\eta}(X_{\hdot},G)$
parametrize polarizable $G$-local systems. The
correspondence between Higgs bundles
and local systems gives a homeo\-morphism of coarse moduli spaces
$$
M_H(X_{\hdot},G) ^{\rm top} \cong M_B(X_{\hdot},G) ^{\rm top}.
$$
There are stratifications of $M_H$, $M_{DR}$, and $M_B$ by locally closed smooth subvarieties which correspond to each other by the above homeomorphism and the
Riemann-Hilbert isomorphism between $M_{DR}^{\rm an}$ and $M_B^{\rm an}$,
such that the Hitchin and Betti complex structures combine to give a 
hyperk\"ahler structure on each stratum. 
\end{theorem}
\begin{proof}
This is a sketch of proof.
The correspondence preserves semisimplicity so it gives a map from the points of
$M_H$ to the points of $M_B$. 
Proceed as in \cite{Moduli} to get the homeomorphism, using the 
real subspaces of $R^h_B\subset R_B$ and $R^h_H\subset R_H$ 
consisting of framings compatible with
harmonic metrics. The moduli spaces are quotients of $R_B^h$ and $R^h_H$
by compact groups. 
This argument will give, furthermore, that the map
$$
M_{\eta}(X_{\hdot}, G)\rightarrow M_{\eta}(X_0,G)
$$
is a proper map of topological spaces, from which it follows that it is a proper
map of schemes. Since, for $\eta =B$, these are affine, we get in fact that
the emap is finite. 
 
Define a canonical stratification by
starting with the open set of smooth points (of the reduced subscheme) where furthermore
the restriction map to $M_{\eta}(X_0,G)$ is etale onto its image, looking at the complement,
and continuing with the same construction.  Lemma \ref{samesings}
shows that a point $\rho \in M_B$ will be at the same depth of this stratification as
its corresponding point $(P,\theta )\in M_H$. The images of the strata are canonically
defined locally closed subvarieties of $M_{\eta}(X_0,G)$. As such, they 
are compatible with all of the complex structures, so they are 
hyperk\"ahler subvarieties of the hyperk\"ahler structure of Hitchin-Fujiki \cite{Fujiki}.
Being etale over those of
$M_{\eta}(X_0,G)$, the strata in $M_{\eta}(X_{\hdot}, G)$ 
have hyperk\"ahler structures too. 
\end{proof}

The homeomorphism gives  
continuity of the $\cc ^{\ast}$ action. 

\begin{corollary}
\label{cstarcont}
Suppose
$X_{\hdot}$ is a simplicial smooth projective variety, connected and which satisfies
the finite index condition \ref{finiteindex}.
Then the action of $\cc ^{\ast}$ is continuous on the character variety
$M_B(X_{\hdot},G)$. 

In particular, if $\rho$ is a semisimple representation of
$\pi _1(|X_{\hdot}|)$ which is locally rigid,
then it is fixed by the action of $\cc ^{\ast}$ so it 
underlies a strongly polarizable variation of Hodge structure.

The real Zariski closure of its monodromy group is of Hodge type. Therefore,
lattices in real groups not of Hodge type cannot occur as $\pi _1(|X_{\hdot}|)$.
\end{corollary}
\begin{proof}
The action is algebraic on $M_H$ so by the homeomorphism of the previous
theorem it is continuous on $M_B$. The rest follows as in \cite{hbls}. 
For the last part, note that the real Zariski closure of
the monodromy group of $\pi _1(X_0)$ has finite index in
the real Zariski closure of the monodrom on $\pi _1(|X_{\hdot}|)$,
so the conditions of Hodge type are equivalent, one concludes using
\cite{hbls} for $X_0$. 
\end{proof}

By Theorem \ref{unibranch}, these restrictions, analogous to those
for smooth projective varieties, apply in particular to
any normal or even geometrically unibranched DM-stack. 

An interesting question is whether other restrictions on fundamental groups of
compact K\"ahler manifolds, including many works
such as Gromov's \cite{Gromov}---see
the full discussion of \cite{AmorosBurgerEtAl}---extend to
the $\pi _1(|X_{\hdot}|)$ for $X_{\hdot}$ satisfying the finite index condition
\ref{finiteindex}. A weaker question is to what extent these restrictions hold
for smooth proper DM-stacks. Is the class of 
fundamental groups of smooth proper DM-stacks different from the classes 
of compact K\"ahler groups, or fundamental groups of smooth projective varieties?
And how do these compare with the classes of fundamental groups of normal projective
varieties, normal DM-stacks, the $\pi _1(|X_{\hdot}|)$ for $X_{\hdot}$ satisfying the finite index condition
\ref{finiteindex}, etc?

\section{The smooth case}

Look now at the above constructions for the case when $X$ is a smooth proper
Deligne-Mumford stack. This was our main and original motivation, even though
for expositional reasons we have concentrated on the simplicial case up until now.
It is one of the cases which has attracted the most
attention in the literature. 
For example, Biswas-G\'omez-Hoffmann-Hogadi \cite{BGHH} treat 
local systems
over an abelian gerb. If $X$ is a smooth projective variety with
simple normal crossings divisor $D$, then the Cadman-Vistoli
root stacks which have been discussed
previously are smooth and proper. Local systems on root stacks correspond to
parabolic bundles (with rational weights), 
so the numerous works concerning parabolic bundles may
be viewed as treating local systems on the root stacks, as will be discussed
in detail in the second half of this section. 

Fix a connected smooth proper DM-stack $X$, and let  $Z_{\hdot}\rightarrow X$ 
be a proper surjective hypercovering
such that the $Z_k$ are smooth projective varieties given by Theorem \ref{fullresolution}.
The first terms $(Z,R,K)$ are assumed to form a partial simplicial resolution
constructed according to the recipe above Theorem \ref{fullresolution},
starting from a surjective-where-etale
morphism $Z\rightarrow X$ from a smooth projective
variety of Theorem \ref{maincovering}. 

For $\eta = B,DR,H,Hod,DH$ the 
moduli stacks $\Mm _{\eta}(Z_{\hdot},G)$ may be interpreted as moduli stacks
of the various kinds of local systems on $X$ 
$$
\Mm _{\eta}(Z_{\hdot},G)\cong \Mm _{\eta}(X,G),
$$
indeed bundles with $\lambda$-connection
(resp. local systems)
on $Z_{\hdot}$ descend to bundles with $\lambda$-connection (resp. local systems)
on $X$, by Lemma \ref{vbdescent2} (resp. Lemma \ref{lsdescends}).
Semistability for Higgs bundles requires some further discussion below. 

Letting $z\in Z$ be a lift of the basepoint $x\in X$, the same may
be said of the representation schemes 
$$
R_{\eta}(Z_{\hdot},z,G) \cong R_{\eta}(X_{\hdot},x,G).
$$

Local systems on $X$ may be identified with representations of
Noohi's fundamental group $\pi _1(X,x)$ defined in \cite{NoohiPi1},
which is the same as the fundamental group of the topological realization
$|X|$. 
So the Betti moduli stacks can be expressed
$$
R_B(X,x,G)=Hom (\pi _1(X,x),G)
$$
$$
\Mm _{\rm B}(X) = Hom (\pi _1(X,x),G)\stackquot G.
$$
We have the Riemann-Hilbert correspondence between 
local systems and vector bundles with integrable algebraic connection
$$
\Mm _{\rm B}(X)^{\rm an}\cong \Mm _{\rm DR}(X)^{\rm an}
$$
which 
may be constructed over $Z$ and then descended down to $X$. 

A smooth proper DM-stack satisfies Condition \ref{finiteindex}, by Theorem \ref{unibranch},
so polarizability, strong polarizability and semistability are the same
by Lemma \ref{fiex}. More generally all the results of the previous section apply. 

In order to give an intrinsic description of the moduli stack of Higgs bundles
$\Mm _H(X,G)$,
a notion of semistability is needed.

Nironi has introduced a very interesting notion of {projective DM-stack} \cite{Nironi}.
This allows him to generalize the theory of moduli of vector bundles and similar objects,
by applying the same techniques as in the projective case. Our technique applies to any
proper smooth DM-stack, but doesn't give as much as what Nironi can do: we are constrained
to consider only moduli spaces of objects with vanishing Chern classes, which correspond
in some way to representations of the fundamental group, while Nironi's techniques 
in the case of a ``projective'' DM-stack (in his sense)
would allow consideration of moduli spaces of vector bundles with arbitrary Chern classes. 

On a general smooth proper
DM-stack $X$ we don't have a K\"ahler class to use for defining semistability,
but due to the fact that we are interested in flat bundles here i.e. $c_2=0$,
there are various ways of getting around that: either require semistability 
for {\em some} variety mapping to $X$, or for {\em all} varieties mapping to $X$.

\begin{definition}
\label{sepostable}
Suppose $(E,\theta  )$ is a Higgs bundle on a smooth proper DM-stack $X$.
We say that it is potentially semistable (resp. potentially polystable) if there exists a polarized projective variety
$Y$ and a surjective map $g:Y\rightarrow X$ such that the Higgs bundle $g^{\ast}(E,\theta )$ is
slope-semistable (resp. slope-polystable) on $Y$ with respect to the given polarization. 
\end{definition}

In general this notion will not be very well behaved: even if $X$ is a projective variety itself, we are allowing semistability
with respect to an arbitrary polarization. However, when the Chern classes vanish then the condition no longer depends on
a choice of polarization so we can expect that it gives a reasonable condition on a DM-stack too.
Recall that Vistoli's theorem provides the notion of rational Chern classes on $X$, see \cite{IyerSimpson1}. Thus, 
the condition $c_i(E)=0$ in $H^{2i}(|X|,\qq )$ makes good sense. 

The following condition for Higgs bundles
has been introduced and 
extensively considered by Bruzzo, Hern\'andez, Otero and others
\cite{BruzzoHernandez} \cite{BruzzoOtero}. They relate it to a condition of numerical effectivity, as was originally considered for vector bundles by Demailly, Peternell, and Schneider \cite{DemaillyPeternellSchneider}.

\begin{definition}
\label{seplutable}
Suppose $(E,\theta  )$ is a Higgs bundle on a smooth proper DM-stack $X$.
We say that it is {pluri-semistable} (resp. {pluri-polystable}) if 
for every curve $Y$ and map $g:Y\rightarrow X$ the Higgs bundle $g^{\ast}(E,\theta )$ is
slope-semistable (resp. slope-polystable) on $Y$ with respect to the polarization
which, for a curve, is unique up to scalars. 
\end{definition}

\begin{remark}
If $(E,\theta )$ is pluri-semistable (resp. pluri-polystable) then for any
polarized smooth projective variety $Y$ and map $g:Y\rightarrow X$, 
$g^{\ast}(E,\theta )$ is
slope-semistable (resp. slope-polystable) on $Y$ with respect to the given polarization. 
In particular $(E,\theta )$ is potentially semistable (resp. potentially polystable).
\end{remark}

Potential semistability implies pluri-semistability when the rational Chern classes
vanish, and these conditions are also related
to Higgs-nefness of the bundle and its dual, see Bruzzo-Otero \cite[Theorem 4.7]{BruzzoOtero}.

\begin{lemma}
\label{polystableok}
Suppose $(E,\theta )$ is a potentially semistable (resp. potentially polystable) 
Higgs bundle on $X$, with $c_i(E)=0$ in rational cohomology
for $i=1,2$. Then it is pluri-semistable (resp. pluri-polystable).
In particular for any map from a smooth projective variety $g:Y\rightarrow X$, 
the pullback $g^{\ast}(E,\theta )$  is a successive extension of
stable Higgs bundles and corresponds to a
representation of $\pi _1(Y)$ via \cite{hbls}. The rational Chern classes vanish
for all $i$. 
\end{lemma}

Bruzzo and co-authors have formulated the following partial converse
(for instance it would be the implication in the other direction
in Bruzzo-Otero \cite[Theorem 4.7]{BruzzoOtero}), 
which we call the {\em Bruzzo conjecture}:

\begin{conjecture}[Bruzzo conjecture]
If $(E,\theta )$ is a pluri-semistable Higgs bundle over a smooth proper DM-stack, 
then $c_i(E)=0$ in rational cohomology for all $i$.
\end{conjecture}

This conjecture would generalize to Higgs bundles the theorem of Demailly, Peternell and
Schneider who prove it for vector bundles i.e. when $\theta =0$ \cite{DemaillyPeternellSchneider}.

After this discussion of semistability, we can formulate more precisely
the moduli problems solved by $\Mm _{\rm H}(X,G)$ and $\Mm _{\rm Hod}(X,G)$.

\begin{definition}
A $G$-principal Higgs bundle on $X$ is of semiharmonic (resp. harmonic)
type, if 
its Chern classes vanish in rational cohomology, and if it is
potentially or equivalently pluri-semistable (resp. pluri-polystable). 
This definition extends to $\lambda$-connections too. 
\end{definition}

If $P$ is a $G$-principal Higgs bundle on $X$ then its pullback to $Z_{\hdot}$ is
of semiharmonic type if and only if $P$ is. Hence,
the moduli stack $\Mm _{\rm H}(X,G)$ parametrizes principal Higgs
$G$-bundles of semiharmonic type; and the moduli stack 
$\Mm _{\rm Hod}(X,G)\rightarrow \aaa ^1$
parametrizes principal $G$-bundles with $\lambda$-connection of semiharmonic type. 

\begin{theorem}
\label{smoothcasecorr}
Proposition \ref{correspondence} gives a tannakian Kobayashi-Hitchin
correspondence between Higgs bundles
of semiharmonic type on $X$
and local systems on $X$. The Higgs bundles of harmonic
type correspond to the semisimple local systems, these conditions being the
same as (strong) polarizability on both sides.  For these polarizable objects, 
harmonic metrics exist which set up the correspondence via the same
differ\-ential-geometric structures as in the case of varieties, over the etale local charts.
The resulting map between moduli spaces is a homeomorphism and determines
a hyperk\"ahler structure. 
\end{theorem}
\begin{proof}
By Condition \ref{finiteindex} and Lemma \ref{fiex}, 
polarizability, strong polarizability 
and semisimplicity are equivalent in the categories of local systems or
Higgs bundles of semiharmonic type. Those tannakian categories are equivalent by
Proposition \ref{correspondence}. Given a Higgs bundle of harmonic type,
its pullback to each $Z_k$ is of harmonic type, so it has a unique structure
of harmonic bundle. Furthermore, by strong polarizability, a compatible
collection of metrics $h_k$ may be chosen. Then from the condition that
$Z\rightarrow X$ is surjective where etale, and the subsequent choice of the rest of
$Z_{\hdot}$, the bundle, the harmonic metric, and various 
connection operators descend to $X$. Over etale charts in $X$, in particular those
which are contained in $Z$, these structures satisfy the usual axioms for a harmonic metric.
They give in particular the corresponding flat connection.
The same discussion works starting from a semisimple local system. 
For the homeomorphism and hyperk\"ahler structure, apply Theorem \ref{hyperkahler}. 
\end{proof}

Suppose $X$ is a smooth variety and $D\subset X$ is a divisor with normal crossings. 
Hermitian Yang-Mills theory and the Kobayashi-Hitchin correspondence 
have been considered for parabolic bundles on $(X,D)$ by many authors
\cite{Biquard}
\cite{LiNarasimhan} 
\cite{Nakajima} \cite{TMochizuki2} \cite{TMochizuki3}
\cite{Poritz} \cite{SteerWren}.  
These theories may be related to
the the corresponding theories over a smooth proper Cadman-Vistoli
root stack, something that was basically observed by Daskalopoulos and Wentworth
quite some time ago \cite{DaskalopoulosWentworth}.

Let $Z\rightarrow X$ be the
root stack corresponding to denominators $n_i$ for the irreducible components $D_i$ of $D$. 
As in the original article of Seshadri \cite{Seshadri}, a vector bundle on $Z$ corresponds to a parabolic bundle
on $(X,D)$ such that the weights along $D_i$ are in $\frac{1}{n_i}\zz$. 
This correspondence has been used and studied
by many authors, see for example Boden \cite{Boden},
Balaji {\em et al}
\cite{BalajiDeyParthasarathi}, Biswas \cite{Biswas}, Borne \cite{Borne05} \cite{Borne07}
as well as \cite{IyerSimpson1} and
\cite{IyerSimpson2}. 

An important condition for a parabolic structure is to be {locally abelian}, that is near any multiple intersection
point of $D$, the parabolic structure should decompose as a direct sum of parabolic line bundles. Borne and Vistoli \cite{Borne05} \cite{Borne07} have recently
improved our understanding of this condition by the following result.

\begin{theorem}[Borne]
 Suppose $E = \{E_{\alpha _1,\ldots , \alpha _m}\}$ is a parabolic torsion-free sheaf (that is a system of torsion-free sheaves
and inclusions satisfying the conditions of semicontinuity and twisting by
the divisor components). 
Then $E$ is a locally abelian parabolic bundle, if and only if all of the
component sheaves $E_{\alpha _1,\ldots , \alpha _m}$ are locally free. 
\end{theorem}
\begin{proof}
If $E$ is locally abelian then automatically the components are bundles,
so the task is to prove that if each $E_{\alpha _1,\ldots , \alpha _m}$ is a vector bundle, then the parabolic structure is locally abelian. 

This is Borne's Proposition 2.3.10 \cite{Borne07}. For the proof, he uses the following
main statement which he attributes to Vistoli \cite[Lemma 2.3.11]{Borne07}: suppose $E\subset F \subset E(D)$ is
a pair of inclusions of locally free sheaves, with $D$ a smooth
divisor. Then $F/E$ and $E(D)/F$ are locally free sheaves on $D$.
The proof in turn refers to the formula of Auslander-Buchsbaum in EGA.  
\end{proof}

A {\em parabolic $\lambda$-connection} is a locally abelian parabolic bundle
$E_{\hdot}$ together with a $\lambda$-connection operator 
$$
\nabla : E_{\alpha _1,\ldots , \alpha _m}\rightarrow 
E_{\alpha _1,\ldots , \alpha _m}\otimes \Omega ^1_X(\log D).
$$ 
One defines the parabolic degree and hence the notion of parabolic stability.
Moduli spaces for parabolic vector bundles, parabolic Higgs bundles, and
parabolic connections have been studied in many places:
\cite{Seshadri} \cite{MehtaSeshadri} 
\cite{MaruyamaYokogawa} \cite{BodenYokogawa} 
\cite{Yokogawa}
\cite{Nitsure} \cite{Konno} \cite{Nakajima}
\cite{Thaddeus} 
\cite{BalajiBiswasNagaraj}
\cite{InabaIwasakiSaito}
is a certainly non-exhaustive list. 

Given a semistable parabolic $\lambda$-connection, the residual data are locally
constant along the non-intersection points $y$ of the divisor components $D_i$. 
Thus one can speak of the residue of $(E,\nabla )$ along $D_i$. It is a pair
$$
{\rm res}_{D_i,y}(E,\nabla )=
\left(
\bigoplus _{\alpha \in (-1,0]}{\rm gr}^{D_i}_{\alpha}(E(y)),{\rm res}(\nabla )\right)
$$ 
consisting of a vector
space graded by a finite number of parabolic weights $\alpha \in (-1,0]$,
together with an endomorphism ${\rm res}(\nabla )$. The graded piece 
${\rm gr}_{\alpha}(E(y))$ is the fiber at $y$ of the quotient $E_{\alpha }/E_{\alpha -\epsilon}$, and the residue of $\nabla$ comes from the action on this graded piece.
Here $y$ is in a single divisor component $D_i$ so the parabolic structure near
$y$ is reduced to a single index, indicated 
for the notation by a superscript ${\rm gr}^{D_i}$.

Say that $(E,\nabla )$ has semisimple residues, if the 
${\rm res}(\nabla )$ are semisimple endomorphisms. 
Note that this is a weaker condition than asking that
the residue be semisimple for $\nabla$ considered as a logarithmic connection 
on one of the component vector bundles $E_{\alpha _1,\ldots , \alpha _m}$,
because this bigger residual endomorphism might have a unipotent factor which acts by
strictly decreasing the parabolic weight. 

One can more generally define the notion of parabolic bundle on a smooth DM-stack with
respect to a normal crossings divisor, a viewpoint which is useful for the inductive
kind of argument used in \cite{IyerSimpson2}. On the other hand, a parabolic bundle
all of whose weights are integers, may be viewed as a usual parabolic bundle.
 
The bundles with $\lambda$-connection on the root stack $Z=X[\frac{D_1}{n_1},\ldots , \frac{D_m}{n_m}]$ are exactly the pullbacks of parabolic bundles from $(X,D)$ such 
that the pullback has
integer weights and trivial residue of the connection. Making this condition explicit 
gives the following proposition. 

\begin{proposition}
\label{paratranslate}
Suppose $\lambda \in \cc$ and $n_i$ are strictly positive integers. 
Pullback gives an equivalence of categories, preserving the conditions of (semi)stability
and the Chern classes, between:
\newline
---parabolic $\lambda$-connections on $(X,D)$ such that the parabolic weights along $D_i$ are in $\frac{1}{n_i}\zz$ and the residue 
of the connection on each parabolic graded piece is semisimple with a single eigenvalue
given as follows:
$$
{\rm res}^{D_i}_{\alpha}(\nabla ) = \lambda \alpha \cdot \mbox{{\em {\large 1}}}_{{\rm gr}^{D_i}_{\alpha}(E)};
$$
and
\newline
---bundles with $\lambda$-connection on the root stack $Z=X\left[ \frac{D_1}{n_1},\ldots , \frac{D_m}{n_m}\right]$.
\end{proposition}

One may translate using this equivalence between the parabolic and stack-theoretic
points of view, in particular the numerous works on 
harmonic theory and moduli for parabolic bundles become
relevant for the problem we are considering here. Particularly so in the basic case of
a root stack. A further discussion of the details, such as the behavior of the
harmonic metrics near $D_i$, would take us too far afield and these aspects 
are amply treated
already in the many available references. 

The analogue of parabolic structures for principal $G$-bundles is not completely
straightforward: one needs to introduce the notion of parahoric structure, and this
is the subject of current ongoing research by several authors \cite{BoalchParahoric}
\cite{BalajiSeshadri}.

For smooth proper $X$ it is natural to formulate Poincar\'e duality. 
The importance of Poincar\'e duality for the study
of fundamental groups has become apparent in recent works of Bruno Klingler.
The coarse moduli space of a smooth proper DM-stack $X$ is a proper rational homology
manifold. The cohomology of the stack is the same as that of its coarse moduli space, so it is easy to see that Poincar\'e duality holds for $H^{\hdot}(X,\qq )$. This has been remarked for example by Abramovich, Graber and Vistoli in \cite{AGV}, and was undoubtedly one of the reasons
for Deligne's comment about rational homology manifolds in \cite{hodge3}. Still, for
cohomology with coefficients in a local system it is better to have an intrinsic proof
such as was given by Behrend. 

\begin{theorem}
Suppose $X$ is a connected smooth proper DM-stack of dimension $n$. Then
the fundamental class of $X$ gives a canonical isomorphism
$H^{2n}(X,\cc )\cong \cc$; and for any local system $L$ on $X$, the cup product
followed by the trace $L\otimes L^{\ast}\rightarrow \cc$ gives a perfect pairing
$$
H^i(X,L)\times H^{2n-i}(X,L^{\ast})\rightarrow H^{2n}(X,\cc )\cong \cc .
$$
\end{theorem}
\begin{proof}
We refer to Behrend \cite{BehrendDuality}. 
\end{proof}

Poincar\'e duality allows us to prove the purity of the mixed twistor structure
on cohomology. 

\begin{corollary}
\label{pureHS}
Suppose $X$ is a connected smooth proper DM-stack. If $f:Z\rightarrow X$
is a dominant morphism from another smooth proper DM-stack (in particular $Z$
could be a smooth projective variety) then for any local system $L$, 
pullback along $f$
is an injection
$$
f^{\ast}: H^i(X,L)\hookrightarrow H^i(Z,f^{\ast}L).
$$
If $L$ is a pure variation of
Hodge structure of weight $w$, then the 
mixed Hodge structure on $H^i(X,L)$ is pure of weight $i+w$. 
\end{corollary}
\begin{proof}
Suppose ${\rm dim}(X)=n$ and 
$p:Y\rightarrow Z$ is a surjective morphism from a connected smooth 
projective variety, also of dimension $n$. This exists by Theorem \ref{maincovering}.
There is an open subset $U\subset X$ over which $p$ is a finite etale covering
of degree $d$.  A top degree cohomology class on $X$ may be represented by a
form which is compactly supported in $U$, so
the pullback map 
$$
\cc \cong H^{2n}(X,\cc )\stackrel{p^{\ast}}{\rightarrow} H^{2n}(Y,\cc )\cong \cc 
$$
is multiplication by $d$. If $p_!$ denotes the Poincar\'e dual of $p^{\ast}$
the standard argument shows that $p_!(p^{\ast}u)=d\cdot u$ for $u\in H^i(X,L)$, 
implying that $p^{\ast}$ is injective.

Suppose $f:Z\rightarrow X$ is a dominant morphism of smooth proper DM-stacks. 
Choose a surjective map $q:V\rightarrow Z$ from a smooth projective variety 
with ${\rm dim}(V)={\rm dim}(Z)$. Let $Y$ be a general complete intersection of hyperplane
sections in $Z$, with ${\rm dim}(Y)={\rm dim}(X)$. The projection $p:Y\rightarrow X$ is
surjective so by the previous discussion $p^{\ast}$ is injective; it follows
that $f^{\ast}$ is injective. 

If $L$ is  a variation of pure Hodge structure of weight $w$,
the pullback map
$$
p^{\ast}:H^i(X,L)\rightarrow H^i(Y,p^{\ast}(L))
$$ 
is an injective morphism of mixed Hodge structures, whose target is pure of weight $i+w$,
therefore $H^i(X,L)$ is pure of weight $i+w$.  
\end{proof}

\section{Mixed twistor theory}

Deligne's theory of \cite{hodge3} goes over to mixed twistor structures. This is useful
for looking at the topology of simplicial smooth projective varieties, so we gives some
details here expanding upon the places where it was mentioned in \cite{twistor}.  
It will allow us to generalize Corollary \ref{pureHS} to a purity 
statement, Corollary \ref{pureTS}, 
for any semisimple local system. The
development presented here is 
undoubtedly subsumed in a theory of ``mixed twistor modules''
generalizing Saito's mixed Hodge modules 
as was done by  Sabbah for the pure case \cite{Sabbah}.

A D-mixed twistor complex is a filtered complex of sheaves of $\Oo _{\pp ^1}$-modules
$(\Ff ^{\hdot}, {W}_{\hdot})$ on $\pp ^1$ such that 
$$
H^i({W}_{n}\Ff ^{\hdot} / {W}_{n-1}\Ff ^{\hdot})
$$
is a semistable vector bundle of slope $n+i$ on $\pp ^1$, nonzero
for only finitely many $(i,n)$. 

A {B-mixed twistor complex} is a filtered complex of sheaves of $\Oo _{\pp ^1}$-modules
$(\Ff ^{\hdot}, W_{\hdot})$ on $\pp ^1$ such that 
$$
H^i(W_{m}\Ff ^{\hdot} / W_{m-1}\Ff ^{\hdot})
$$
is a semistable vector bundle of slope $m$ on $\pp ^1$, nonzero
for only finitely many $(i,m)$. 

In our notations D stands for Deligne and B for Beilinson: the D-mixed Hodge complexes
were defined by Deligne \cite{hodge3}, whereas Beilinson's treatment \cite{BeilinsonMHC},
see also Huber \cite{Huber}, refers to the B-mixed notion. See also Zucker \cite{Zucker04},
where the notion of relaxed MHC is introduced. 

I have often wondered about how to express the relationship between these two notions.
Although this materiel is well-known to experts, it seems likely that some readers
will find it useful to review the relationship. This explanation is easier to follow
in the case of mixed twistor structures, since we can work within the abelian category
of sheaves on $\pp ^1$, avoiding concerns about strictness of maps between filtered vector spaces. 

Consider first the passage from a D-mixed twistor complex to the mixed twistor structure
on cohomology. 
Recall that the spectral sequence of a filtered complex $(\Ff ^{\hdot},W_{\hdot})$
has 
$$
E_0^{k,l}:= W_{-k}\Ff ^{k+l}/W_{-k-1}\Ff ^{k+l}
$$
with differential $d_0: E_0^{k,l}\rightarrow E_0^{k,l+1}$
induced by the differential $d$ of $\Ff ^{\hdot}$. Then 
$$
E_1^{k,l}(\Ff^{\hdot}, W_{\hdot})=H^{k+l}(W_{-k}/W_{-k-1}).
$$
The differential $d_1:E_1^{k,l}\rightarrow E_1^{k+1,l}$
is, with different indices, the connecting map 
$$
H^i(W_m/W_{m-1})\rightarrow H^{i+1}(W_{m-1}/W_{m-2})
$$
coming from the short exact sequence of complexes
$$
0\rightarrow  W_{m-1}/W_{m-2}\rightarrow W_{m}/W_{m-2}\rightarrow W_{m}/W_{m-1} 
\rightarrow 0.
$$
Going back to the indices $k,l$ we obtain the expression
\begin{equation}
\label{e2expression}
E^{k,l}_2 (\Ff ^{\hdot}, W_{\hdot}) = \frac{{\rm ker}\left( 
H^{k+l}(W_{-k}/W_{-k-1})\rightarrow H^{k+l+1}(W_{-k-1}/W_{-k-2})
\right) }{{\rm im}\left( 
H^{k+l-1}(W_{1-k}/W_{-k})\rightarrow H^{k+l}(W_{-k}/W_{-k-1})
\right)} .
\end{equation}
The next differential is 
$$
d_2: E^{k,l}_2\rightarrow E^{k+2,l-1}_2.
$$
Finally, the spectral sequence abuts to $H^{k+l}(\Ff ^{\hdot})$ with the
filtration induced by $W_{\hdot}$, more precisely defined as
$$
W_mH^i(\Ff ^{\hdot}):= 
{\rm  im}\left( H^i(W_m\Ff ^{\hdot})\rightarrow H^i(\Ff ^{\hdot}) \right) .
$$
This all works in the context of filtered complexes in any abelian category.
In our case we work with the abelian category of sheaves of $\Oo _{\pp ^1}$-modules
over $\pp ^1$. 

If $(\Ff^{\hdot}, W_{\hdot})$ is a D-mixed twistor complex, then by definition
$$
E_1^{k,l}(\Ff^{\hdot}, W_{\hdot})=H^{k+l}(W_{-k}/W_{-k-1})
$$
is a semistable vector bundle of slope $(k+l-k)=l$ on $\pp ^1$. 
In particular the differential $d_1:E_1^{k,l}\rightarrow E_1^{k+1,l}$
is a strict morphism between semistable vector bundles of the same slope $l$,
so its kernel and cokernels are also semistable vector bundles of slope $l$,
and indeed the expression \eqref{e2expression} expresses $E^{k,l}_2$
as the quotient of a semistable bundle by another one of the same slope. 

\begin{corollary}
\label{Dtwist}
If $(\Ff ^{\hdot},W_{\hdot})$ is a D-mixed twistor complex, then 
$d_1$ is a strict morphism between semistable vector bundles of slope $l$,
and the cohomology $E^{k,l}_2$ of the resulting complex is a semistable vector
bundle of slope $l$. Furthermore, $d_r=0$ for all $r\geq 2$ and the spectral
sequence degenerates at $E_2$. We obtain the expression
$$
\mbox{\eqref{e2expression}} = E^{k,l}_2 = 
\frac{W_{-k}H^{k+l}(\Ff ^{\hdot})}{W_{-k-1}H^{k+l}(\Ff ^{\hdot})}.
$$
Hence, if we set 
$$
W^B_mH^i(\Ff ^{\hdot}):= W_{m-i}H^i(\Ff ^{\hdot})=
{\rm  im}\left( H^i(W_m\Ff ^{\hdot})\rightarrow H^i(\Ff ^{\hdot}) \right)
$$
then $(H^i(\Ff ^{\hdot}),W^B_{\hdot})$ is a mixed twistor structure. 
\end{corollary}
\begin{proof}
The first sentence is what was seen above. But then $d_2$ is a map from a semistable 
vector
bundle of slope $l$ to a semistable vector bundle of slope $l-1$, so $d_2=0$.
Hence $E_3=E_2$, $d_3$ becomes a morphism from a bundle of slope $l$ to one of slope $l-2$
so it vanishes, and so on. Inductively $d_r=0$ for all $r\geq 2$ so the spectral 
sequence degenerates at $E_2$. Thus $E_2^{k,l}=Gr^W_{-k}H^{k+l}$ and this is a
semistable bundle of slope $l$. Changing the indices, this says that $Gr_m^WH^i$
is semistable of slope $m+i$. To get a mixed twistor structure we have to shift the
filtration; the new filtration may be denoted $W^B_{\hdot}$ because it will coincide
with the filtration obtained from the Beilinson picture (see below). We have 
$$
Gr ^{W^B}_m H^i= Gr ^W_{m-i}H^i
$$
which is semistable of slope $(m-i+i)=m$, which exactly says that $H^i(\Ff ^{\hdot})$
together with its filtration $W^B_{\hdot}$ is a mixed twistor structure. 
\end{proof}

We now look at how this works if we first pass to the Beilinson point of view.

\begin{corollary}
\label{Btwist}
If $(\Ff ^{\hdot}, W^B_{\hdot})$ is a B-mixed twistor complex then the
spectral sequence degenerates at 
$$
E^{k,l}_1(\Ff ^{\hdot}, W^B_{\hdot}) = H^{k+l}(Gr ^{W^B}_{-k}\Ff ^{\hdot})
$$
which are semistable bundles of slope $-k$ on $\pp ^1$. The induced filtration
$$
W^B_mH^i(\Ff ^{\hdot}):= {\rm im} \left(
H^i(W^B_m\Ff ^{\hdot})\rightarrow H^i(\Ff ^{\hdot})
\right)
$$
gives a mixed twistor structure on $H^i(\Ff ^{\hdot})$. 
\end{corollary}
\begin{proof}
Follow the proof of Corollary \ref{Dtwist}, noting that 
$H^{k+l}(Gr ^{W^B}_{-k}\Ff ^{\hdot})$ is semistable of slope $-k$ by the definition 
of B-mixed twistor complex. The differential $d_1$ vanishes already because it decreases
the slope.
\end{proof}

Given a D-mixed twistor complex $(\Ff ^{\hdot}, {W}_{\hdot})$ with the 
differential of the complex $\Ff ^{\hdot}$ denoted by $d$,
we obtain a
B-mixed twistor complex $(\Ff ^{\hdot}, {W}^B_{\hdot})$ by setting
$$
W^B_{m}(\Ff ^i):= 
$$
$$
{\rm ker}\left(
d:W_{m-i}\Ff ^i \rightarrow \frac{W_{m-i}\Ff ^{i+1}}{W_{m-i-1}\Ff ^{i+1}}
\right) .
$$ 
This is a subobject of 
$W_{m-i}\Ff ^i$.
Note in passing that if $d=0$ this is just the same as $W_{m-i}$ explaining
the notation in Corollary \ref{Dtwist} above. 
The filtration $W^B_{\hdot}$ is usually
called $Dec(W_{\hdot})$, cf \cite{hodge3} and \cite{Huber}. 

\begin{lemma}
\label{D2B}
The above construction starting from a D-mixed twistor complex yields
a B-mixed twistor complex. More particularly,
$$
E^{k,l}_1(\Ff ^{\hdot}, W^B_{\hdot})= H^{k+l}(Gr ^{W^B}_{-k}\Ff ^{\hdot}) = 
E^{2k+l,-k}_2(\Ff ^{\hdot},W_{\hdot}).
$$
The spectral sequence for $(\Ff ^{\hdot}, W^B_{\hdot})$ degenerating at $E_1$
abuts to $W^B$ of Corollary \ref{Btwist} which in this case is
the same filtration as the $W^B$ of Corollary
\ref{Dtwist}.  Furthermore 
the construction $W\mapsto W^B$ is multiplicative: if 
$(\Ff ^{\hdot},W_{\hdot})$ and $(\Gg ^{\hdot},W_{\hdot})$ are
D-mixed twistor complexes then 
$$
W^B_{\hdot}(\Ff ^{\hdot}\otimes \Gg ^{\hdot}) 
$$
is the tensor product filtration of $W^B_{\hdot}$ on $\Ff ^{\hdot}$ and 
$W^B_{\hdot}$ on $\Gg ^{\hdot}$.
\end{lemma}
\begin{proof}
Unfortunately the $E_0$ term for $W^B$ doesn't coincide with the 
$E_1$ term for $W$, instead there is an extra acyclic complex there. 
This is just Proposition 1.3.4 of \cite{hodge2} applied to the abelian category
of sheaves of $\Oo _{\pp ^1}$-modules, but we write things out more explicitly.
See also the discussion of Lemma 1.3.15 of \cite{hodge2}, as well as certainly other
more recent references. 

We have
$$
E^{k,l}_0(\Ff ^{\hdot},W^B_{\hdot}) =
\frac{W^B_{-k}\Ff ^{k+l}}{W^B_{-k-1}\Ff ^{k+l}}
$$
$$
=
\frac{{\rm ker}\left(
d:W_{-2k-l}\Ff ^{k+l}\rightarrow W_{-2k-l}\Ff ^{k+l+1}/W_{-2k-l-1}
\right)}{{\ker}\left(
d:W_{-2k-l-1}\Ff ^{k+l}\rightarrow W_{-2k-l-1}\Ff ^{k+l+1}/W_{-2k-l-2}
\right)}.
$$
In particular there is a natural projection
$$
\begin{array}{ccl}
E^{k,l}_0(\Ff ^{\hdot},W^B_{\hdot})& \rightarrow &H^{k+l}(Gr^W_{-2k-l}\Ff ^{\hdot})
\;\;\;\; =\;\;\;\;  E^{2k+l, -k}_1(\Ff ^{\hdot}, W_{\hdot}) \\
& & 
=
\frac{{\rm ker}\left(
d:W_{-2k-l}\Ff ^{k+l}/W_{-2k-l-1}\rightarrow W_{-2k-l}\Ff ^{k+l+1}/W_{-2k-l-1}
\right)}{{\rm im}\left(
d:W_{-2k-l}\Ff ^{k+l-1}/W_{-2k-l-1}\rightarrow W_{-2k-l}\Ff ^{k+l}/W_{-2k-l-1}
\right)} .
\end{array}
$$
Let $\Uu ^{k,l}$ be the kernel, that is we have an exact sequence
$$
0\rightarrow \Uu ^{k,l}\rightarrow E^{k,l}_0(\Ff ^{\hdot},W^B_{\hdot})\rightarrow 
E^{2k+l, -k}_1(\Ff ^{\hdot}, W_{\hdot})\rightarrow 0.
$$
In the above expressions, look at the subobject of $W^B_{-k}\Ff ^{k+l}/W^B_{-k-1}$
determined by the image of $W_{-2k-l-1}$. It is contained in $\Uu ^{k,l}$,
and is of the form 
$W_{-2k-l-1}\Ff ^{k+l}/{\rm ker}(d)$ which is naturally isomorphic to the image, 
denoted by ${\rm im}(Gr^W_{-2k-l-1}(d^{k+l})$, of 
$$
d:Gr^W_{-2k-l-1} \Ff ^{k+l}\rightarrow Gr^W_{-2k-l-1} \Ff ^{k+l+1}.
$$
On the other hand, 
$$
\frac{W^B_{-k}\Ff ^{k+l}}{W_{-2k-l-1} \Ff ^{k+l}} = 
{\rm ker}\left( 
d:Gr ^W_{-2k-l}\Ff ^{k+l}\rightarrow Gr ^W_{-2k-l}\Ff ^{k+l+1}
\right) .
$$
The kernel of the projection from here to $H^{k+l}(Gr ^W_{-2k-l}\Ff ^{\hdot})$ is by definition the image denoted ${\rm im}(Gr^W_{-2k-l}(d^{k+l-1})$ of 
$$
d:Gr^W_{-2k-l} \Ff ^{k+l-1}\rightarrow Gr^W_{-2k-l} \Ff ^{k+l}
$$
This leads to an exact sequence
$$
0\rightarrow {\rm im}(Gr^W_{-2k-l-1}(d^{k+l})) 
\rightarrow \Uu ^{k,l} \rightarrow {\rm im}(Gr^W_{-2k-l}(d^{k+l-1})) \rightarrow 0.
$$
On the other hand, note that 
$$
d:W_{-2k-l-1}\Ff ^{k+l-1}\rightarrow \{0 \} \subset E^{k,l}_0(\Ff ^{\hdot},W^B_{\hdot})
$$
so $d$ induces a map
$$
Gr^W_{-2k-l-1}\Ff ^{k+l-1}\rightarrow \Uu ^{k,l}.
$$
The kernel of $d:Gr^W_{-2k-l-1}\Ff ^{k+l-1}\rightarrow Gr^W_{-2k-l-1}\Ff ^{k+l}$
maps to zero in $\Uu ^{k,l}$ in view of the original expression for 
$E^{k,l}_0(\Ff ^{\hdot}, W^B_{\hdot})$. Thus $d$ induces a map
$$
{\rm im}(Gr^W_{-2k-l-1}(d^{k+l-1})\rightarrow \Uu ^{k,l}.
$$
This splits the previous exact sequence, so we get a direct sum decomposition
$$
\Uu ^{k,l} = {\rm im}(Gr^W_{-2k-l-1}(d^{k+l}))\oplus 
{\rm im}(Gr^W_{-2k-l}(d^{k+l-1})).
$$
The differential 
$$
d_0: E^{k,l}_0(\Ff ^{\hdot},W^B_{\hdot})\rightarrow 
E^{k,l+1}_0(\Ff ^{\hdot},W^B_{\hdot})
$$
sends $\Uu ^{k,l}$ to $\Uu ^{k,l+1}$ and on there it
is equal to the splitting map defined above, identifying 
$$
{\rm im}(Gr^W_{-2k-l}(d^{k+l-1})\subset \Uu ^{k,l}
$$
with
$$
{\rm im}(Gr^W_{-2k-l}(d^{k+l-1})\subset 
\Uu ^{k,l+1}.
$$
It follows that the complex
$$
\cdots \stackrel{d_0}{\rightarrow} \Uu ^{k,l-1}\stackrel{d_0}{\rightarrow}  \Uu ^{k,l}
\stackrel{d_0}{\rightarrow}\cdots 
$$
is acyclic. The map 
$$
E^{k,l}_0(\Ff ^{\hdot},W^B_{\hdot})\rightarrow 
E^{2k+l, -k}_1(\Ff ^{\hdot}, W_{\hdot})
$$
which is compatible with the differential $d_0$ on the left and $d_1$ on the
right, so for $k$ fixed it induces a map of complexes. The kernel of this map of
complexes is the acyclic complex formed by the $\Uu ^{k,l}$. We get an isomorphism
on cohomology, which is to say an isomorphism between the next terms in 
the spectral sequence:
$$
E^{k,l}_1(\Ff ^{\hdot},W^B_{\hdot})\stackrel{\cong}{\rightarrow}
E^{2k+l, -k}_2(\Ff ^{\hdot}, W_{\hdot}).
$$
Using our hypothesis that $(\Ff ^{\hdot},W_{\hdot})$ is
a D-mixed twistor complex, recall from Corollary \ref{Dtwist} that
$E^{2k+l, -k}_2(\Ff ^{\hdot}, W_{\hdot})$ are semistable vector bundles of slope
$-k$ on $\pp ^1$. We get the same property for 
$$
E^{k,l}_1(\Ff ^{\hdot},W^B_{\hdot})= H^{k+l}(Gr ^{W^B}_{-k}\Ff ^{\hdot}),
$$
which is exactly the property required to say that $(\Ff ^{\hdot}, W^B_{\hdot})$ is
a B-mixed twistor complex. The remaining statements of the lemma may be verified
from the above discussion. 
\end{proof}

Let $MTC^D$ (resp. $MTC^B$) be the category of D-mixed (resp. B-mixed) twistor
complexes. 
Notice that a complex in the category $MTS$ in the category of mixed twistor structures,
is in particular a B-mixed twistor complex. Thus we have a functor 
$$
{\rm Cpx}(MTS) \rightarrow MTC^B,
$$
and Beilinson shows that this gives an equivalence of derived categories. 
There isn't a natural lift along the functor $Dec : MTC^D\rightarrow MTC^B$,
but $Dec$ also induces an equivalence of derived categories. 

The difference between $MTC^D$ and $MTC^B$ may be seen in the loss of information going
from $MTC^D$ to $MTC^B$: a D-mixed
twistor complex yields the associated $E^{2k+l,-k}_1$ terms of the
spectral sequence, which are themselves semistable bundles of slope $-k$ on $\pp ^1$.
However, the $E^{k,l}_0$-terms of the spectral sequence for the associated B-mixed   
twistor complex are extensions of these bundles by terms $\Uu ^{k,l}$ of an acyclic
complex. From here, one cannot in general recover the $E^{2k+l,-k}_1$ term of the 
original D-mixed twistor complex. It is a question of taste, how much one wants
to consider this extra information as a part of the geometrical structure.
For a given singular variety, if we choose different simplicial
resolutions, the Deligne $E^{2k+l,-k}_1$-terms might be different.
On the other hand, Deligne's $E_2$ terms, which are the same as Beilinson's $E_1$
terms, are invariant as may be stated in the following corollary. 

\begin{corollary}
Suppose $(\Ff ^{\hdot},W_{\hdot})\stackrel{\phi}{\rightarrow} 
(\Gg ^{\hdot}, W_{\hdot})$ is
a morphism of D-mixed twistor complexes. Suppose that $\phi$
induces an isomorphism (resp. injection, resp. surjection)
on cohomology $\phi : H^i(\Ff ^{\hdot})\cong H^i(\Gg ^{\hdot})$.
Then the map induced by $\phi$ 
$$
E^{k,l}_2(\Ff ^{\hdot}, W_{\hdot}) \rightarrow E^{k,l}_2(\Gg ^{\hdot}, W_{\hdot})
$$
is an isomorphism (resp. injection, resp. surjection) of pure vector bundles on $\pp ^1$.

Suppose $(\Ff ^{\hdot},W^B_{\hdot})\stackrel{\phi}{\rightarrow} 
(\Gg ^{\hdot}, W^B_{\hdot})$ is
a morphism of B-mixed twistor complexes. Suppose that $\phi$
induces an isomorphism (resp. injection, resp. surjection)
on cohomology $\phi : H^i(\Ff ^{\hdot})\cong H^i(\Gg ^{\hdot})$.
Then the map induced by $\phi$ 
$$
E^{k,l}_1(\Ff ^{\hdot}, W_{\hdot}) \rightarrow E^{k,l}_1(\Gg ^{\hdot}, W_{\hdot})
$$
is an isomorphism (resp. injection, resp. surjection)
of pure vector bundles on $\pp ^1$. 
\end{corollary}
\begin{proof}
In both cases, the indicated terms of the spectral sequence are equal to the
associated graded pieces of the mixed twistor structure given by Corollaries
\ref{Dtwist} and \ref{Btwist}. Strictness for maps between mixed twistor structures
\cite{twistor} says that injectivity and surjectivity pass to the associated-graded
pieces. 
\end{proof}

Suppose we are given a functor 
$$
G: \Delta \rightarrow MTC^D
$$
denoted by $k\mapsto (G^{\hdot}(k),W_{\hdot})$. 
Then Deligne defines the {total complex} 
$$
{\bf tot}(G)^j:=  \bigoplus _{i+k=j} G^i(k)
$$
with weight filtration 
$$
W^{{\rm Dec}_1}_m {\bf tot}(G)^j:= \bigoplus _{i+k=j}W_{m-k}G^i(k).
$$
The differentials of ${\bf tot}(G)^{\hdot}$ are obtained by combining the differentials
of $G^{\hdot}(k)$ with the alternating sums of the simplicial face maps. 

If we are given a functor 
$$
G: \Delta \rightarrow MTC^B, \;\;\; k\mapsto (G^{\hdot}(k),W^B_{\hdot})
$$ 
then Beilinson considers the same total complex with differential
$$
{\bf tot}(G)^j:=  \bigoplus _{i+k=j} G^i(k)
$$
but with weight filtration 
$$
W^B_m {\bf tot}^B(G)^j:= \bigoplus _{i+k=j}W_{m}G^i(k).
$$

\begin{proposition}
\label{cosimp}
For a cosimplicial D-mixed twistor complex $G$, the
total complex 
$({\bf tot}(G)^{\hdot}, W^{{\rm Dec}_1}_{\hdot})$ is again a D-mixed twistor complex, inducing
a mixed twistor structure on $H^i({\rm tot}(G)^{\hdot})$.

For a cosimplicial
B-mixed twistor complex $G$, the total complex  
$({\bf tot}(G)^{\hdot}, W^B_{\hdot})$ is again a B-mixed twistor complex, inducing
a mixed twistor structure on $H^i({\bf tot}(G)^{\hdot})$. 

If we start with a cosimplicial D-mixed twistor complex $G$
and let $W^BG(k)$ be the filtration of $G^{\hdot}(k)$ considered in Lemma \ref{D2B},
varying functorially in $k$ to give 
a cosimplicial B-mixed twistor complex. Both of these induce the same
mixed twistor structure on $H^i({\bf tot}(G)^{\hdot})$.
\end{proposition}
\begin{proof}
As in \cite{hodge3}. 
\end{proof} 

\begin{remark}
Given a cosimplicial D-mixed twistor complex $G$, we get a D-mixed twistor complex
$({\bf tot}(G),W^{{\rm Dec}_1}_{\hdot})$ from the first paragraph of the proposition, then 
a B-mixed twistor complex by Lemma \ref{D2B}. On the other hand, applying the construction
of Lemma \ref{D2B} levelwise we get a cosimplicial B-mixed twistor complex,
which also gives a B-mixed twistor complex by the second paragraph of the proposition.
These will not in general be the same; the third paragraph of the proposition says
that they still induce the same weight filtration on the total cohomology. 
\end{remark}

Suppose $X$ is a smooth projective variety. A 
polarizable pure variation of twistor structure
of weight $w$ (VTS) is just a semisimple local system $L$ on $X$. The weight $w$ may be chosen
arbitrarily, and determines the realization of $L$ into a family of twistor
structures parametrized by $x\in X$, which we denote by $L^w$. 
See \cite{twistor} \cite{Sabbah}. 

Given a VTS $L^w$ of weight $w$, we obtain a D-mixed twistor complex 
$(A^{\hdot} _{\rm tw}(X,L^w), W_{\hdot})$ 
as follows. 
The weight filtration will be trivially concentrated
in degree $w$, that is to say 
\begin{equation}
\label{weightdef}
W_mA^{i} _{\rm tw}(X,L^w) = \left\{\begin{array}{ll}
A^{i}_{\rm tw}(X,L^w) & m\geq w \\
0 & m<w . \end{array} \right.
\end{equation}
So we just have to define the complex $A^{\hdot} _{\rm tw}(X,L^w)$. 
Let $\Ll$ be the $C^{\infty}$ bundle underlying the local system, and put
$$
A^i_{\rm tw}(X,L^w):= A^i(X,L)\otimes _{\cc}\Oo _{\pp ^1}(w+i).
$$
Since $L$ is
semisimple, it has a structure of harmonic bundle \cite{Corlette} \cite{DonaldsonApp}, 
giving a decomposition of the
flat connection $d$ on $\Ll$ into
$$
d=\partial + \overline{\partial} + \theta +\overline{\theta} 
$$
in the notations of \cite{hbls}. 
Let $\lambda , \mu : \Oo _{\pp ^1}\rightarrow \Oo _{\pp ^1}(1)$ denote the two sections
vanishing respectivly at $0$ and $\infty$. Then 
$$
d_{\rm tw} := \lambda ( \partial + \overline{\theta}) 
+ \mu (\overline{\partial} + \theta ) = \lambda D' + \mu D''
$$
defines an operator
$$
d_{\rm tw}: A^i(X,L)\otimes _{\cc}\Oo _{\pp ^1}(w+i)\rightarrow
A^{i+1}(X,L)\otimes _{\cc}\Oo _{\pp ^1}(w+i+1)
$$
which is to say a differential for $A^{\hdot}_{\rm tw}(X,L^w)$. 

The variation of twistor structure on $L$ corresponds to a prefered section of
the twistor moduli stack $\Mm _{DH}(X,GL(n))$, and the above
complex is the Deligne-Hitchin glueing of the complexes calculating cohomology
of $\lambda$-connections on $X$ and $\overline{X}$, see \cite{Sabbah} \cite{twistor}.

\begin{lemma}
\label{harmonicMTC}
The complex $(A^{\hdot}_{\rm tw}(X,L^w), d_{\rm tw})$ together with
the weight filtration $W_{\hdot}$ of \eqref{weightdef} concentrated trivially in degree $w$,
is a D-mixed twistor complex.
\end{lemma}
\begin{proof}
See \cite{twistor}. 
Recall from \cite{hbls} that the cohomology of $d$ is the same as
that of $({\rm ker}(\partial + \overline{\theta}), \overline{\partial} + \theta )$
or symmetrically $({\rm ker}(\overline{\partial} + \theta), \partial + \overline{\theta}  )$, these cohomologies are isomorphic to the spaces of harmonic forms,
and in fact there is a $D'D''$-lemma. From these, the cohomology bundle 
$H^i(A^{\hdot}_{\rm tw}(X,L^w),d_{\rm tw})$ is isomorphic to the
cohomology of the sequence
$$
A^{i-2}(X,L)\stackrel{D'D''}{\rightarrow}
A^i(X,L)\stackrel{(D',D'')}{\rightarrow}
A^{i+1}(X,L)\oplus A^{i+1}(X,L),
$$
all tensored with $\Oo _{\pp ^1}(w+i)$. 
Hence
$$
H^i(A^{\hdot}_{\rm tw}(X,L^w),d_{\rm tw}) = H^i(X,L)\otimes \Oo _{\pp ^1}(w+i).
$$
The D-mixed twistor property follows immediately.
\end{proof}

We now complete the twistor analogue of the main construction of Hodge III \cite{hodge3}. 
If $L_{\hdot}$ is a local system on a simplicial smooth projective variety $X_{\hdot}$
such that each $L_k$ is a semisimple local system on $X_k$, then for any integer $w$
$L_{\hdot}$ has a structure of polarizable variation of pure twistor structure
of weight $w$ denoted $L^w_{\hdot}$. 

\begin{corollary}
\label{simpMTS}
In this situation, the cohomology $H^i(X_{\hdot}, L^w_{\hdot})$ has a
natural mixed twistor structure whose underlying bundle over $\pp ^1$ is
obtained by the Deligne-Hitchin glueing. 

This mixed twistor structure 
is functorial for morphisms between local systems, compatible with cup-product,
and contravariantly functorial for morphisms of simplicial varieties in the following way. 
Suppose $f:X_{\hdot}\rightarrow Y_{\hdot}$ is
a morphism of simplicial smooth projective varieties, and that $L_{\hdot}$
is a local system on $Y_{\hdot}$ with each $L_k$ semisimple. 
Fix a weight $w$. Then
$f$ induces a map of mixed twistor structures 
$$
H^i(Y_{\hdot}, L^w_{\hdot})\rightarrow H^i(X_{\hdot}, f^{\ast}(L)_{\hdot}^w).
$$
If the map on cohomology is an isomorphism then it is an isomorphism of mixed twistor
structures. 
\end{corollary}
\begin{proof}
The D-mixed twistor complexes of Lemma \ref{harmonicMTC} are contravariantly 
functorial, so they
fit together into a cosimplicial D-mixed twistor complex. 
(Complexes of forms on simplicial manifolds are discussed in \cite{Dupont}  \cite{Jeffrey}.)
By Proposition \ref{cosimp} this gives a total D-mixed twistor complex inducing 
a mixed twistor structure on cohomology. Functoriality follows from the construction and
the last phrase comes from the strictness property for mixed twistor structures
\cite{twistor}. 
\end{proof}

\begin{corollary}
\label{pureTS}
Suppose $X$ is a connected smooth proper DM-stack. 
If $L$ is semisimple local system considered as a pure variation of
twistor structure of weight $w$, then the 
mixed twistor structure on $H^i(X,L)$ is pure of weight $i+w$. 
\end{corollary}
\begin{proof}
Choose a dominant morphism from a smooth projective variety $p:Z\rightarrow X$.
By Corollary  \ref{pureHS}, the morphism on cohomology 
$$
H^i(X,L)\rightarrow H^i(Z,p^{\ast}(L))
$$
is injective. This is a morphism of mixed twistor structures
and the one on the right is pure, so the one on the left is pure too.
\end{proof}

It would clearly be interesting to develop a theory of variations of mixed twistor
structures over simplicial varieties, leading to a mixed twistor structure on the
total cohomology. This would go beyond our present scope; but see \cite{twistor} for
a discussion of VMTS on a single smooth variety.

\section{Finite group actions}

In this section, we discuss some examples which may
be obtained by considering finite group actions.

Suppose $\Phi$ is a finite group acting on a connected smooth projective
variety $X$, and $G$ is a complex linear algebraic
group. Then $\Phi$ acts on the moduli stacks $\Mm _{\eta}(X,G)$
preserving all of the
various structures. 

\begin{lemma}
In the above situation, the substack of stacky fixed points is identified with
the moduli stack for the quotient $Y=X\stackquot \Phi$:
$$
\Mm _{\eta}(X,G)^{\Phi} \cong \Mm _{\eta}(Y,G).
$$
\end{lemma}
\begin{proof}
A stacky fixed point in $\Mm _{\eta}(X,G)$ is defined as an object together with
a compatible action of $\Phi$ covering the action on $X$, 
which is exactly the same as an object with descent data down to $Y$. 
\end{proof}

One can  observe that for a global quotient stack, 
the fundamental group is also the fundamental group of a smooth projective variety,
and the fixed point stack of the preceding lemma can be interpreted in this way.
This observation was already present in Daskalopoulos-Wentworth \cite{DaskalopoulosWentworth}. 

\begin{proposition}
If $Y=X\stackquot \Phi$ is a global quotient 
stack for a group $\Phi$ acting on a connected smooth projective variety $X$,
then we can construct a connected smooth projective variety $Z$ and a map $f:Z\rightarrow Y$
inducing an isomorphism  $\pi _1(Z,z)\cong \pi  _1(Y,f(z))$.
In particular, $\Mm _{\eta}(Y,G)\cong \Mm _{\eta}(Z,G)$. 
\end{proposition}
\begin{proof}
Indeed, there exists a smooth projective variety $U$ with $\pi _1(U)=G$ by Serre's construction \cite{SerrePi}, see also Browder and Katz \cite{BrowderKatz}. Let $P$ be the universal cover of $U$, so $P$ is simply
connected and has a free action of $G$. Now put 
$$
Z:= Y\times P / G.
$$
This is a smooth projective variety provided with a map $f:Z\rightarrow Y$ which is a fiber bundle in the etale
topology of $Y$. The fiber $P$ is simply connected fiber, so the long exact sequence of homotopy groups implies that $f$ induces an isomorphism on $\pi _1$. 
\end{proof}

A more subtle question concerns the quotient of the group action. 
The group action preserves all of the structure on the moduli stack, 
hence for example the subset of smooth points of the moduli space quotient
$M_{\eta}(X,G)/\Phi$ admits a hyperk\"ahler structure. 
This suggests that 
$\Mm _{\eta}(X,G)\stackquot \Phi$ should itself be viewed as a kind of ``nonabelian $1$-motive''.
We look at how to realize it as a connected component of a moduli stack.

Consider first the case where a finite group $\Phi$ acts on a group $G$ but
acts trivially on $X$. Let $H = G\rtimes \Phi$ be the semidirect product fitting into
the split exact sequence 
$$
1\rightarrow G \rightarrow H \rightarrow \Phi \rightarrow 0.
$$
This induces a sequence of maps of moduli stacks
$$
\Mm _{\eta}(X,G)\rightarrow \Mm _{\eta}(X,H)\rightarrow \Mm _{\eta}(X,\Phi ).
$$
The trivial $\Phi$-torsor has $\Phi$ as group of automorphisms,
so it corresponds to a map
$$
B\Phi \rightarrow \Mm _{\eta}(X,\Phi ).
$$

\begin{lemma}
\label{actsongroup}
With the above notations we have a cartesian square of moduli stacks
for $\eta =B,DR,H \ldots $ refering to any type of local system 
$$
\begin{array}{ccc}
\Mm _{\eta}(X,G) \stackquot \Phi & \rightarrow & \Mm _{\eta}(X,H) \\
\downarrow && \downarrow \\
B\Phi & \rightarrow & \Mm _{\eta}(X,\Phi ).
\end{array}
$$
\end{lemma}

Suppose now that $\Phi$ acts on our smooth projective variety $X$
with DM-stack quotient $Y:=X\stackquot \Phi$; in fact $X$ could also
be a DM-stack itself. 

Let $G\wr \Phi$ denote the wreath product
(these have been used for geometry, cf 
\cite{WangZhou}), 
that is the semidirect product of
$\Phi$ with its permutation action on $\prod _{\Phi}G$. 
Elements are denoted $(v,(g_w)_{w\in \Phi})$. There is a canonically
split projection $G\wr \Phi \rightarrow \Phi$, which induces a map on 
moduli spaces.

There is an action of $\Phi$ on $G\wr \Phi$, combining its adjoint action 
on itself, its translation action on $\prod _{w\in \Phi}G$, and its given action
on $G$. The formula is
$$
\varphi \in \Phi : (v,(g_w)_{w\in \Phi})\mapsto (\varphi v\varphi ^{-1}, (\varphi (g_{\varphi ^{-1}w}))_{w\in \Phi}).
$$
Let $H:= (G\wr \Phi )\rtimes \Phi$ be the semidirect product for this action. 

The covering $X\rightarrow Y$ is a $\Phi$-torsor which induces a point
denoted 
$$
[X]=\ast \rightarrow \Mm _{\eta}(Y,\Phi )
$$
in any of the moduli spaces of $\Phi$-local systems over $Y$, which all parametrize
$\Phi$-torsors since $\Phi$ is a finite group. 

Use first this torsor and the group $G\wr \Phi$ to transform the action of $\Phi$
on $\Mm _{\eta}(X,G)$ to an action on the group only, the case of Lemma \ref{actsongroup}.

\begin{proposition}
\label{changeaction}
Let $Y:= X\stackquot \Phi$ be the DM-stack quotient. Then 
for any type of local system $\eta$ we have a cartesian
diagram of moduli stacks 
$$
\begin{array}{ccc}
\Mm _{\eta}(X,G) & \rightarrow & \Mm _{\eta}(Y,G\wr \Phi )\\
\downarrow & & \downarrow \\
\left[ X \right] & \rightarrow & \Mm _{\eta}(Y,\Phi ).
\end{array}
$$
This is compatible with the action of $\Phi$, given on $X$ and $G$,
thereby induced on $G\wr\Phi$, and by the adjoint action on $\Phi$ for the
lower right corner. 
\end{proposition}
\begin{proof}
If $P$ is a principal $G$-bundle over $X$, 
a group element $w\in \Phi$ 
tranlates it to a new one $w^{\ast}P$ defined by
$(w^{\ast}P)_x:= P_{w^{-1}x}$. 
We get a principal $\prod _{w\in \Phi}G$-bundle over $X$
$$
\prod _{w\in \Phi}w^{\ast}P \rightarrow X,
$$
but $\Phi$ also acts on this bundle so it may be considered as
a principal $G\wr \Phi$-bundle over $Y$.
This construction respects structures of flat $\lambda$-connection or
a structure of topological local system, so it defines 
a map 
$$
\Mm _{\eta}(X,G) \rightarrow 
\Mm _{\eta}(Y,G\wr \Phi ).
$$
The image under the map $G\wr \Phi \rightarrow \Phi$ 
which in our notations is just projection
to the first coordinate, is naturally isomorphic to the covering $X$ considered
as a $\Phi$-torsor. 
This completes the construction of the commutative square in the proposition.
It is compatible with the various actions of $\Phi$. 

To finish the proof we have to show that it is cartesian. Suppose $Q$ is a $G\wr \Phi$-bundle over $Y$, projecting to a $\Phi$-torsor provided with an isomorphism 
to $X$. This gives a map $Q\rightarrow X$ which is a $\prod _{w\in \Phi}G$-torsor
over $X$. Changing structure group by the projection at the identity element 
$$
\prod _{w\in \Phi}G \rightarrow G , \;\;\;\; (g_w)_{w\in \Phi}\mapsto g_1
$$
yields a $G$-torsor $P$ over $X$. This construction provides the required isomorphism
between $\Mm _{\eta}(X,G\wr \Phi )$ and the fiber product in the cartesian square. 
\end{proof}

The semidirect product $H= (G\wr \Phi )\rtimes \Phi$ fits into an exact sequence
$$
1\rightarrow \prod _{w\in \Phi}G \rightarrow H \rightarrow (\Phi \rtimes \Phi )\rightarrow 1
$$
where the quotient is the semidirect product made using the adjoint action of $\Phi$ on
itself. The $\Phi$-torsor $X$ yields by extension of structure group a $\Phi \rtimes \Phi$-torsor $X\times ^{\Phi}(\Phi \rtimes \Phi)$ which projects to the trivial $\Phi$-torsor under the quotient map $\Phi \rtimes \Phi \rightarrow \Phi$. The group 
$\Phi$ acts by automorphisms on $X\times ^{\Phi}(\Phi \rtimes \Phi)$,
giving a map to the moduli stack
\begin{equation}
\label{bphimap}
B\Phi \rightarrow \Mm _{\eta}(X,\Phi \rtimes \Phi ).
\end{equation}

\begin{corollary}
\label{quotexpression}
If $\Phi$ acts on $G$ and $X$, setting $Y:= X\stackquot \Phi$ and 
$H:= (G\wr \Phi )\rtimes \Phi$, we have a cartesian square of algebraic stacks
$$
\begin{array}{ccc}
\Mm _{\eta}(X,G)\stackquot \Phi & \rightarrow &  \Mm _{\eta}(Y,H) \\
\downarrow & & \downarrow \\
B\Phi & \rightarrow & \Mm _{\eta}(X, \Phi \rtimes \Phi ).
\end{array} 
$$
\end{corollary}
\begin{proof}
Proposition \ref{changeaction} allows us to express the action of $\Phi$ on 
the moduli stack
$\Mm _{\eta}(X,G)$ as coming from an action on the group $G\wr \Phi$ only,
for local systems over the DM-stack quotient $Y=X\stackquot \Phi$.  
Lemma \ref{actsongroup} then gives the stack quotient of the moduli space as a pullback
over $B\Phi$. Combining the two pullbacks amounts to taking the pullback 
over the map \eqref{bphimap}. 
\end{proof}

This corollary motivates the introduction of DM-stacks for looking at group actions on
the moduli of local systems over
a smooth projective variety $X$. 

\begin{question}
What are the properties of the induced square 
$$
\begin{array}{ccc}
M _{\eta}(X,G)/ \Phi & \rightarrow & M _{\eta}(Y,H )\\
\downarrow & & \downarrow \\
\ast & \rightarrow & M _{\eta}(Y,\Phi \rtimes \Phi ).
\end{array}
$$
of coarse moduli spaces?
\end{question}

The moduli space $M_{\eta}(X,\Phi \rtimes \Phi )$ is discrete. 
It would be good to be able to say that 
$M _{\eta}(X,G)/ \Phi$ is identified as an irreducible component of
$M _{\eta}(Y,H )$ but that seems to be a perhaps somewhat delicate question about
character varieties.

\section{Fundamental groups of irreducible varieties}

Many years ago, Domingo Toledo asked the following question: is every finitely
presented group the fundamental group of an irreducible singular variety?
In this section we give a streamlined argument to show that the answer is `yes'. 

Take note of the following construction. Suppose $X$ is quasiprojective,
$Z$ a closed subscheme, and $r:Z\rightarrow Y$ a finite morphism.
Then there is a scheme $W$ obtained by ``contracting along $r$''. More precisely,
$W$ is provided with a morphism $p:X\rightarrow W$ and a factorization
$$
\begin{array}{ccc}
Z & \hookrightarrow & X \\
\downarrow && \downarrow \\
Y & \rightarrow & W
\end{array}
$$
which is universal, that is to say it is a cocartesian square in the category of schemes.
Furthermore $Y\hookrightarrow W$ is a closed embedding, the above square is also cartesian,
$W$ is separated of finite type over $\cc$, and the morphism $p$ is finite. 
The coproduct may be denoted by 
$$
W = X/r = X\cup ^ZY.
$$
The associated diagram of topological spaces 
$$
\begin{array}{ccc}
Z^{\rm top} & \hookrightarrow & X^{\rm top} \\
\downarrow && \downarrow \\
Y^{\rm top} & \rightarrow & W^{\rm top}
\end{array}
$$
is also cocartesian and cartesian. 

From this we get the
Brown-Van Kampen statement: that for any $0\leq n \leq \infty$
the diagram of $n$-groupoids 
$$
\begin{array}{ccc}
\Pi _n(Z^{\rm top}) & \rightarrow & \Pi _n(X^{\rm top}) \\
\downarrow && \downarrow \\
\Pi _n(Y^{\rm top}) & \rightarrow & \Pi _n(W^{\rm top})
\end{array}
$$
is cocartesian in the $n+1$-category of $n$-groupoids. For $n=\infty$ this
just says that the previous diagram of spaces is a homotopy pushout. 

For $n=1$, the diagram of fundamental groupoids
$$
\begin{array}{ccc}
\Pi _1(Z^{\rm top}) & \rightarrow & \Pi _1(X^{\rm top}) \\
\downarrow && \downarrow \\
\Pi _1(Y^{\rm top}) & \rightarrow & \Pi _1(W^{\rm top})
\end{array}
$$
is a cocartesian diagram in the $2$-category of groupoids.

\begin{theorem}
Suppose $\Upsilon$ is a finitely presented group. Then there is
an irreducible projective variety $W$ with $\pi _1(W^{\rm top})\cong \Upsilon$. 
\end{theorem}
\begin{proof}
Suppose $\Upsilon$ is a finitely presented group. 
It may be realized as the fundamental group of a $2$-dimensional simplicial complex $A$.
Here $A$ consists of a set of vertices, plus a subset of pairs
of vertices called the edges, and a subset of triples of vertices called the triangles,
such that the edges of the triangles are contained in the set of edges. Such a complex
$A$ is realized into a topological space $|A|$ in an obvious way.

We furthermore may assume that every vertex is contained in some edge, every 
edge is contained in some triangle, and the set of triangles is connected by the
adjacency relation (two triangles being adjacent if they share the same edge).

Let $G$ be the dual graph whose points are the triangles, and whose edges are the edges
common to two triangles. Choose a maximal tree $T\subset G$. This determines a set of
edges of $A$. Define the {unfolding} of $A$ along $T$ denoted by $\tilde{A}$
to be the simplicial complex formed by the triangles of $A$ joined together along
only those edges corresponding to elements of $T$. 

Observe that the topological realization $|\tilde{A}|$ is simply connected, being a union
of triangles inductively joined along single edges according to the tree pattern. 
On the
other hand, the $1$-skeleta are $1$-dimensional simplicial complexes provided
with a map preserving the structure of simplicial complex
$$
\tilde{A}_1\rightarrow A_1
$$
which induces a map on realizations
$$
|\tilde{A}_1|\rightarrow |A_1|.
$$
The diagram of spaces 
$$
\begin{array}{ccc}
|\tilde{A}_1| & \rightarrow & |\tilde{A}| \\
\downarrow && \downarrow \\
 |A_1| & \rightarrow & |A|
\end{array}
$$
is cocartesian, so the corresponding diagram of fundamental groupoids
$$
\begin{array}{ccc}
\Pi _1(|\tilde{A}_1|) & \rightarrow & \Pi _1(|\tilde{A}| )\\
\downarrow && \downarrow \\
 \Pi _1(|A_1|)& \rightarrow & \Pi _1(|A|)
\end{array}
$$
is cocartesian. Note however that $\Pi _1(|\tilde{A}| )=\ast$ is trivial and
$\Pi _1(|A|)$ is equivalent to the group $\Upsilon$. Thus $\Upsilon$ is expressed as
the homotopy contraction of $\Pi _1(|A_1|)$ along $\Pi _1(|\tilde{A}_1|)$.

Now $A_1$ and $\tilde{A}_1$ are just graphs and the map preserves the edge structure.
Hence there are configurations of lines $Y$ and $Z$, that is to say $Y=\bigcup Y_i$
and $Z=\bigcup Z_j$ with $Y_i\cong \pp ^1$ and $Z_j\cong \pp ^1$, such that the
$Y_i$ correspond to edges of $A_1$ meeting at points corresponding to the vertices of $A_1$,
and the $Z_j$ correspond to edges of $\tilde{A}_1$ meeting at points corresponding to the vertices of $\tilde{A}_1$. 
The map $\tilde{A}_1\rightarrow A_1$ corresponds to a finite map $Z\rightarrow Y$.
We obtain a commutative diagram 
$$
\begin{array}{ccc}
\Pi _1(|\tilde{A}_1|) & \rightarrow & \Pi _1(Z^{\rm top} )\\
\downarrow && \downarrow \\
 \Pi _1(|A_1|)& \rightarrow & \Pi _1(Y^{\rm top})
\end{array}
$$
where the horizontal arrows are equivalences of groupoids. 

Embedd now $Z$ in a projective space $X$, and let $W$ be the quotient obtained
by contracting $X$ along $Z\rightarrow Y$. As $\Pi _1(X^{\rm top} )\sim \ast$,
it follows that the diagram 
$$
\begin{array}{ccc}
\Pi _1(Z^{\rm top}) & \rightarrow & \Pi _1(X^{\rm top}) \\
\downarrow && \downarrow \\
\Pi _1(Y^{\rm top}) & \rightarrow & \Pi _1(W^{\rm top})
\end{array}
$$
is the same as 
$$
\begin{array}{ccc}
\Pi _1(|\tilde{A}_1|) & \rightarrow & \Pi _1(|\tilde{A}| )\sim \ast\\
\downarrow && \downarrow \\
 \Pi _1(|A_1|)& \rightarrow & \Pi _1(|A|) \sim \Upsilon .
\end{array}
$$
Thus $\pi _1(W^{\rm top})\cong \Upsilon$, and $W$ is irreducible by construction. 
\end{proof}

\begin{question}
Is it possible to  construct an irreducible variety $W$ with $\pi _1(W)\cong \Upsilon$,
such that the singularities of $W$ are normal crossings?
\end{question}

\end{document}